\documentclass[oneside,a4paper,english, 11pt]{amsart}

\input{artmymacros.sty}

\title{Weight elimination in Serre-type conjectures}

\author{Daniel Le}
\address{Department of Mathematics,
University of Toronto,
40 St. George Street,
Toronto, ON M5S 2E4, Canada}
\email{le@math.toronto.edu}

\author{Bao V. Le Hung}
\address{Mathematics Department,
Northwestern University, 
2033 Sheridan Road,
Evanston, Illinois 60208, USA}
\email{lhvietbao@googlemail.com}

\author{Brandon Levin}
\address{Department of Mathematics,
University of Arizona, 
617 N. Santa Rita Avenue, 
Tucson, Arizona 85721, USA}
\email{bwlevin@math.arizona.edu}

\begin{document}

\begin{abstract} We prove the weight elimination direction of the Serre weight conjectures as formulated by \cite{herzig-duke} for forms of $U(n)$ which are compact at infinity and split at places dividing $p$ in generic situations.  That is, we show that all modular weights for a mod $p$ Galois representation are contained in the set predicted by Herzig.  Under some additional hypotheses, we also show modularity of all the ``obvious'' weights.  
\end{abstract} 

\maketitle

\tableofcontents

\section{Introduction}

Let $p$ be a prime. In 1973, Serre conjectured that every irreducible odd 2-dimensional $\overline{\F}_p$-representation $\rbar$ of $\Gal(\overline{\Q}/\Q)$ comes from a modular form.   He later refined the conjecture into the strong form which asserts that every such $\overline{r}$ arises from a modular form of a specific minimal weight and prime to $p$ level determined by the local properties of $\overline{r}$ \cite{serre-duke}. The recipe for the minimal weight is more subtle than the minimal level, and, as Serre suggested at the time, reflects the deeper structures of a ``mod $p$ Langlands philosophy.''  The landmark proof of Serre's original conjecture due to Kisin and Khare--Wintenberger relies crucially on knowing that the weak form implies the strong form. 

The first comprehensive conjecture for Hilbert modular forms is due to Buzzard--Diamond--Jarvis (BDJ) \cite{BDJ}.  The weight  $k \geq 2$ is replaced by the notion of \emph{Serre weights}, irreducible representations of $\GL_2(\F_q)$ in this case.  Furthermore, there is no longer a notion of minimal weight, rather, BDJ define a collection of Serre weights for which a given $\overline{r}$ should be modular.  The weight part (weak $\implies$ strong) of the BDJ conjecture is now a theorem due to work of many people \cite{gee-kisin, newton13, GLS, GLS15}. 

Building on work of Ash--Doud--Pollack \cite{ADP} and others, Herzig formulated a vast generalization of the weight part of Serre's conjecture to tame $n$-dimensional Galois representations which was further extended by Gee--Herzig--Savitt \cite{herzig-duke, GHS}.  In our earlier work with S. Morra \cite{LLLM, LLLM2}, we establish the weak implies strong conjecture for tame $3$-dimensional Galois representations (and for definite unitary groups unramified at $p$).  However, results for $n > 3$ were limited to a few partial results \cite{BLGG,gee-geraghty, Gao}.  In this paper, we establish the weight elimination direction of the weight part of Serre's conjecture for $n$-dimensional Galois representations in generic situations, namely, the set of modular weights is a subset of the set of weights predicted by \cite{herzig-duke}.

To describe these conjectures, let $F$ be an imaginary CM number field unramified at $p$.  Let $F^+$ be the maximal totally real subfield.  Assume $F^+ \neq \Q$ and that all primes of $F^+$ above $p$ split in $F$.   Let $G$ be a unitary group over $F^+$ which is isomorphic to $U(n)$ at each infinite place and split at each prime above $p$. Let $\overline{r}:\Gal(\overline{F}/F) \ra \GL_n(\overline{\F}_p)$ be an irreducible odd continuous representation.   

 A \emph{global Serre weight} is an irreducible representation $V$ of $\GL_n(\cO_{F^+, p})$ which are all of the form $\otimes_{v \mid p} V_v$ with $V_v$ an irreducible representation of $\GL_n(k_v)$ where $k_v$ is the residue field of $F^+$ at $v$.  In Definition \ref{defn:rbarmod}, we define what it means for $\overline{r}$ to be \emph{modular} of weight $V$.  Roughly speaking, this means the Hecke eigensystem associated to $\overline{r}$ appears in a space of automorphic forms for $G$ of weight $V$.  For each place $v$, fix a place $\tld{v}$ of $F$ dividing $v$, and define $\rhobar_v := \overline{r}|_{\Gal(\overline{F}_{\tld{v}}/F_{\tld{v}})}$.   We now state the main theorem.
\begin{thm} \label{ThmWE}  Let $\overline{r}:\Gal(\overline{F}/F) \rightarrow \GL_n(\overline{\mathbb{F}}_p)$ be an irreducible odd representation.  Assume $p$ is unramified in $F$ and that, for all places $v \mid p$ of $F^+$, $\rhobar_v$ is $(6n-2)$-generic (cf. Definition \ref{defi:generic}).  Then,
$$
\overline{r} \text{ is modular of weight } \otimes_{v \mid p} V_v   \implies  V_v \in W^{?}(\rhobar^{\mathrm{ss}}_v) \text{ for all } v \mid p
$$ 
where $W^{?}(\rhobar^{\mathrm{ss}}_v)$ is defined by \cite{herzig-duke}. 
\end{thm}
The set $W^{?}(\rhobar^{\mathrm{ss}}_v)$ is an explicit collection of irreducible representations of $\GL_n(k_v)$ given by a representation-theoretic recipe.   \cite[Conjecture 1.1]{herzig-duke} predicts that the reverse implication should also be true when $\rhobar_v$ is semisimple for all $v$  (assuming $\rbar$ is modular). In fact, we prove a partial converse which shows the modularity of a subset of the predicted weights as discussed below (Theorem \ref{thmobvintro}).  

\begin{rmk} 
\begin{enumerate} 
\item When $n =3$ and $p$ splits completely in $F^+$, this result is due to \cite{EGH, HLM, MP,LMP}.
\item When $\rhobar_v$ is not semisimple, there is no explicit conjecture but one expects there to be a strict subset $W^?(\rhobar_v) \subset W^?(\rhobar^{\mathrm{ss}}_v)$ which predicts the modular weights.  When $n =3$, this will be taken up in \cite{LLLM3}. 
\item Our methods are purely local and so at least with some technical assumptions a version of Theorem \ref{ThmWE} should hold in other global setups as well, for example, unitary Shimura varieties.   
\end{enumerate} 
\end{rmk}


A key feature of the BDJ conjecture and a motivation for the generalizations is the relation between the weight recipe and $p$-adic Hodge theory properties of the local representation $\rhobar_v$.  Let $K/\Qp$ be a finite unramified extension and consider $\rhobar:G_K \ra \GL_n(\overline{\F}_p)$.  When $n =2$, the weight recipe is in terms of existence of the crystalline lifts of $\rhobar$ in small  (between $[0, p]$) Hodge--Tate weights.   In Herzig's conjecture, one expects  (at least for $\rhobar$ semisimple) that $W^{?}(\rhobar)$ should also be predicted by the existence of crystalline lifts in small weights (cf. \S 1.5 or \S 5 of \cite{GHS} for a detailed discussion).  However, the range of ``small'' Hodge--Tate weights is now $[0, (n-1)p]$ and to determine reductions of $n$-dimensional crystalline representations in this range is still well-beyond the current technology in $p$-adic Hodge theory.  

We consider another local problem, namely, reductions of tamely crystalline representations with fixed (parallel) Hodge--Tate weights $\eta := (n-1,n-2, \ldots, 0)$.   A \emph{tamely crystalline} representation is a  $\overline{\Q}_p$-representation of $G_K$ which becomes crystalline when restricted to $G_L$ for $L/K$ a tame extension.  The descent from $L/K$ is then encoded in a tame inertial type $\tau$.  For short, we will call these representations of type $(\eta, \tau)$.  For generic $\tau$, we give a complete description of the semisimple Galois representations which are reductions mod $p$ of representations of type $(\eta, \tau)$ (cf. Theorem \ref{thm:admcrit} and Theorem \ref{thm:lift}).  

If $\rho$ is a Galois stable lattice in a representation of type $(\eta, \tau)$, then by deep results in integral $p$-adic Hodge theory due to Kisin \cite{KisinFcrys} building on work of Breuil, one can associate a semilinear algebra object, a Kisin module $\fM$ with Hodge type $\eta$ and descent data of type $\tau$.   A. Caraiani and  the third author \cite{CL} construct a moduli stack of Kisin modules $Y^{\eta, \tau}$ with Hodge type $\leq \eta$ and tame inertial type $\tau$.  An upper bound on which $\rhobar$ can arise as the reduction of a representation of type $(\eta, \tau)$ comes from a description of the special fiber of $Y^{\eta, \tau}$.   To describe the special fiber, \cite{CL} relates $Y^{\eta, \tau}$ to a local model $M(\eta)$ constructed by Pappas and Zhu \cite{PZ}.  The special fiber of $M(\eta)$ is a closed subscheme of an affine flag variety.   The coherence conjecture of Pappas and Rapoport, proved by Zhu \cite{Zhu}, allows \cite{PZ} to describe the special fiber as a union of affine Schubert cells resolving a deep and long-standing question in the subject.

By classifying mod $p$ Kisin modules of type $(\eta,\tau)$, we arrive at a combinatorial upper bound for the semisimple $\rhobar$ that arise as the reduction of a representation of type $(\eta, \tau)$ in terms of a subset $\Adm(\eta) \subset \widetilde{W}$ of the extended affine Weyl group of $\GL_n$ called the $\eta$-admissible elements originally introduced by Kottwitz and Rapoport. More precisely, we assign to any semisimple $\rhobar$ a relative position $\tld{w}(\rhobar, \tau)^* \in \widetilde{W}$ called the (\emph{dual}) \emph{shape} (Definitions \ref{affineadjoint}, \ref{defn:shape}, and \ref{defn:shaperhobar}). (Technically, the shape is a data associated of a  Kisin module with descent data, but in generic situations, there is a unique Kisin module which corresponds to $\rhobar$.)  


\begin{thm}\label{thm:liftint} $($Theorem $\ref{thm:lift})$
A $(6n-2)$-generic semisimple Galois representation $\rhobar$ has a lift of type $(\eta,\tau)$ if and only if $\tld{w}(\rhobar,\tau)^* \in \Adm(\eta)$.
\end{thm}

We first use the ``only if'' direction of Theorem \ref{thm:liftint} to reduce weight elimination (Theorem \ref{ThmWE}) to a representation theory/combinatorics problem.
For simplicity, assume that $K = \Qp$.  If $F(\lambda)$ is a Serre weight for $\GL_n(\F_p)$ with $p$-restricted highest weight $\lambda$ which is not in $W^{?}(\rhobar)$, one has to exhibit a tame $\GL_n(\cO_K)$-type  $\sigma(\tau)$ such that $F(\lambda)$ is a Jordan--H\"older factor of the reduction $\overline{\sigma}(\tau)$ and $\tld{w}(\rhobar, \tau)^*$ is not $\eta$-admissible.  Initially, this might seem daunting since there are many types which contain $F(\lambda)$. In fact, it suffices to consider only the types which ``obviously'' contain $F(\lambda)$, namely, the Deligne--Lusztig representations $R_s(\tld{w}_h\cdot \lambda+\eta)$ for $s \in W(\GL_n)$ (see \S \ref{sec:main} for undefined notation).   Precisely, we show that if for all $s \in W(\GL_n)$, the shape of $\rhobar$ relative to $R_s(\tld{w}_h \cdot \lambda+\eta)$ is $\eta$-admissible then $F(\lambda) \in W^{?}(\rhobar)$.   The argument uses alcove geometry to relate admissibility to a description of $W^{?}(\rhobar)$ in terms of dominant $p$-restricted alcoves and linkage due to Herzig. In turn, we use this relationship, together with Theorem \ref{thmobvintro} below and global arguments to establish the ``if'' direction of Theorem \ref{thm:liftint}.

We now discuss our second main theorem which represents partial progress towards the other direction of the weight part of Serre's conjecture.  Among the predicted weights, there is a distinguished subset $W_{\mathrm{obv}}(\rhobar) \subset W^{?}(\rhobar)$ called \emph{obvious} Serre weights.  Obvious weights are defined precisely in Definition 7.1.3 of \cite{GHS}, but roughly correspond to the Hodge--Tate weights (with an $\eta$-shift) in which $\rhobar$ has a crystalline lift which is the direct sum of inductions of characters from unramified extensions.   
When $n =2$, there are only obvious weights so the naive generalization of the weight part of Serre's conjecture would be that $W_{\mathrm{obv}}(\rhobar)$ are exactly the modular weights. Despite their name, the modularity of these weights is by no means obvious.  However, they are more easily accessed via automorphy lifting techniques. For example, \cite{gee-geraghty} obtains essentially complete results on modularity of ordinary obvious weights in the ordinary setting. In \cite{BLGG}, they prove modularity of the obvious weights when $n = 3$ and obtain partial results when $n > 3$. Using Theorem \ref{ThmWE} and a generalization of part of \cite{LLLM} on potentially crystalline deformation rings, we extend these results to $\GL_n$:

\begin{thm} \label{thmobvintro}  Let $\overline{r}:\Gal(\overline{F}/F) \rightarrow \GL_n(\overline{\mathbb{F}}_p)$ be an irreducible representation satisfying Taylor-Wiles conditions.  Assume $p$ is unramified in $F$ and that, for all places $v \mid p$, $\rhobar_v$ is semisimple and $(6n-2)$-generic. If $\overline{r}$ is modular of any obvious weight, then $\overline{r}$ is modular of all obvious weights.
\end{thm}

Before giving an overview of the paper, we summarize the relationship between the methods and results of this paper and those of \cite{LLLM} which is for $n = 3$. The results about shapes of mod $p$ Kisin modules and the triviality of the Kisin variety under genericity conditions in \S 2-3 of \emph{loc. cit.} generalize directly to the $n$-dimensional setting.  Establishing these generalizations is enough to the prove the only-if direction of Theorem \ref{thm:liftint}.  The only-if direction is the necessary local input to establish weight elimination (Theorem \ref{ThmWE}).   

To prove Theorem \ref{thmobvintro}, we compute certain potentially crystalline deformations which requires generalizing arguments of \S 4-5 of \emph{loc.~cit.} The deformation rings are computed explicitly in \emph{loc.~cit.} for all shapes.  Here, we compute the deformation rings only for shapes of the form translation by a permutation of $(0, 1, \ldots, n-1)$ (when $n =3$, this corresponds to the shapes $\alpha \beta \alpha \gamma$ and $\beta \gamma \alpha \gamma$ appearing in Tables at the end of \emph{loc. cit.}).  For these special shapes, the deformation rings turn out to be formally smooth and this is the necessary local input to prove Theorem \ref{thmobvintro}.  

We now give an overview of the paper.  In \S \ref{sec:background}, we begin with some background on affine Weyl groups, Serre weights, tame types, and inertial local Langlands.  In \S \ref{sec:local}, we prove the main local results in $p$-adic Hodge theory.   Sections \ref{sec:phi}-\ref{sec:gen} are direct generalizations to $n$-dimensional Galois representations of results of \cite{LLLM, LLLM2} on Kisin modules with descent data. In \S \ref{sec:pd}, we show, if $\rhobar$ has special shape with respect to $\tau$, then the potentially crystalline deformation ring of type $(\eta, \tau)$ is formally smooth and deduce the existence of potentially diagonalizable lifts.   

The main theorems appear in \S \ref{sec:main}.  Weight elimination is in \S \ref{sec:global}.  Modularity of the obvious weights is in \S \ref{sec:obv}, and in \S \ref{sec:red}, we complete the proof of the  Theorem \ref{thm:liftint}, the local reduction problem using global input. 

 \subsection{Acknowledgements} We would like to thank Matthew Emerton, Thomas Haines, Florian Herzig, and Stefano Morra for many helpful conversations, and Florian Herzig for detailed comments on an earlier draft of this paper.  We thank the referees for very detailed and helpful feedback which greatly improved the paper.
We would also like to thank the Institut Henri Poincar\'e for their hospitality during part of this project.  
The first author was supported by the National Science Foundation under agreements Nos.~DMS-1128155 and DMS-1703182 and an AMS-Simons travel grant. The second author was partially supported by the National Science Foundation under grant No.~DMS-1801963.

\subsection{Notation and Conventions}  \label{sec:not}

Fix $n \geq 2$. Let $p$ be a prime with $p > n$.  Fix a finite unramified extension $K/\Qp$ of degree $f$. Let $k$ denote the residue field of $K$ of cardinality $q = p^f$. Let $\cO_K := W(k)$ be the ring of integers of $K$. We denote the arithmetic Frobenius automorphism on $\cO_K$ by $\phz$, which acts as raising to $p$-th power on the residue field.  Let $G_K = \Gal(\overline{K}/K)$. Let $I_K$ denote the inertia subgroup and $W_K$ the Weil group.  

Let $E/\Qp$ be finite extension assumed to be sufficiently large such that for any unramified extension $K'/K$ of degree the order of an element of $S_n \times S_n$, $E$ contains a copy of $K'$.  Let $\cO$ be the ring of integers of $E$ with uniformizer $\varpi$ and residue field $\F$.   
We fix an embedding $\sigma_0$ of $K$ into $E$ (equivalently an embedding $k$ into $\F$). Define $\sigma_j = \sigma_0 \circ \phz^{-j}$. 

For $r \geq 1$, we fix a compatible system of $(p^{rf} -1)$st roots $\varpi_r = (-p)^{\frac{1}{p^{rf} -1}} \in \overline{K}$.   The choice of root $\varpi_1$ defines a character $\omega_{\varpi_1}:I_K \ra \cO_K^{\times}$.  Using our choice of embedding $\sigma_0$, we get a fundamental character of niveau $f$ 
\[
\omega_{f}:= \sigma_0 \circ \omega_{\varpi_1}:I_K \ra \cO^{\times}. 
\]
We fix once and for all a sequence $\underline{p} := (p_n)_{n\in\N}$  where $p_n\in \overline{\Q}_p$ verify $p_{n+1}^{p}= p_n$ and $p_0=-p$. We let $K_{\infty} := \underset{n\in\N}{\bigcup}K(p_n)$ and $G_{K_{\infty}} := \Gal(\overline{\Q}_p/K_{\infty})$. 

Let $\un{G} = \Res_{k/\F_p} \GL_n$.  Let $T \subset \GL_n$ be the diagonal torus and $\un{T} = \Res_{k/\F_p} T$, a maximal torus of $\un{G}$. Let $\un{Z} \subset \un{T}$ denote the center of $\un{G}$.  Let $W(\un{G}) = W(\GL_n)^{\Hom(k, \F)}$ denote Weyl group of $\un{G}$.   Similarly, let $X^*(\un{T})$ be the (geometric) characters of $\un{T}$, which is equipped with an action of Frobenius $\pi$. 
We have an action of $\pi$ on $W(\un{G})$ by the formula $\pi(w)(\pi(\nu)) = \pi(w(\nu))$. 
There are isomorphisms $W(\un{G}) \cong W(\GL_n)^f$ and $X^*(\un{T}) \cong X^*(T)^f$ where the $j$-th entry corresponds to the embedding $\sigma_j$. Under this identification, the action of $\pi$ is such that if $\nu = (\nu_j) \in X^*(\un{T})$,  then $\pi(\nu)_j = \nu_{j-1}$. 
Let $\Lambda_R \subset X^*(T)$ (resp. $\un{\Lambda}_R \subset X^*(\un{T})$) denote the root lattice for $\GL_n$ (resp. $\bun{G}$).

Let $W_a$ (resp. $\tld{W}$) denote the affine Weyl group and the extended affine Weyl group for $\GL_n$.  Similarly, we will use $\bun{W}_a \cong W_a^f$ and  $\bun{\tld{W}} \cong \tld{W}^f$ to denote the (extended) affine Weyl group of $\underline{G}$.   Recall that 
\[
W_a = \Lambda_R \rtimes W(\GL_n), \quad \tld{W} = X^*(T) \rtimes W(\GL_n)
\]
and similarly for $\bun{W}_a$ and $\bun{\tld{W}}$.   We use $t_{\nu} \in \tld{W}$ to denote translation by $\nu \in X^*(T)$. 
The action of $\pi$ on $X^*(\bun{T})$ and $W(\bun{G})$ extends naturally to $\tld{\bun{W}}$. 

Let $\un{R}^{+} \subset \un{R}$ (resp. $\un{R}^{+, \vee} \subset \un{R}^{\vee}$) denote the subset of positive roots (resp.~positive coroots) in the set of roots (resp.~coroots) with respect to the upper triangular Borel subgroup in each embedding.  Define dominant (co)characters with respect to this choice of positive roots. 

We fix an isomorphism $X^*(T) \cong \Z^n$ in the standard way, where the standard $i$-th basis element $\eps_i=(0,\ldots, 1,\ldots, 0)$ (with the $1$ in the $i$-th position) of the right-hand side correspond to extracting the $i$-th diagonal entry of a diagonal matrix. Dually we get a standard isomorphism $X_*(T)\cong \Z^n$, and let $\{\eps_i^\vee\}$ denote the dual basis.
Let $\eta_0 = (n-1, n-2, \ldots, 0) \in X^*(T)$ be a fixed lift of the half sum of positive roots for $\GL_n$.  Define $\eta = (\eta_0, \eta_0, \ldots, \eta_0) \in X^*(\un{T})$. In the paper, sometimes we will consider simultaneously the group  $\un{G}=\Res_{k/\F_p} \GL_n$ for multiple $k$. In such a situation, the symbol $\eta$ will be used for the above element for any of the groups that appear, and it will always be clear which group it occurs in. Note that $\eta_j=\eta_0$ for any $j$.

We will always denote the duality pairing between a free $\Z$-module and its dual (e.g. $X^*(T)$ and $X_*(T)$) by $\langle ,\rangle$.

For any ring $S$, we define $M_n(S)$ to be the set of $n\times n$ matrix with entries in $S$. If $\alpha=\eps_i-\eps_j$ is a root of $\GL_n$, we also call the $(i,j)$-th entry of a matrix $X\in M_n(S)$ the $\alpha$-th entry. We will make use of both notations $X_{ij}$ and $X_{\alpha}$ for this entry.

If $P$ is a statement, the symbol $\delta_P\in \{0,1\}$ takes value $1$ if $P$ is true, and $0$ if $P$ is false. 

If $W$ is a de Rham representation of $G_K$ over $E$, then for each $\kappa \in \Hom(K, E)$, we write $\mathrm{HT}_{\kappa}(W)$ for the multiset of Hodge--Tate weights labeled by embedding $\kappa$ normalized such that the $p$-adic cyclotomic character has Hodge--Tate weight $\{1\}$ for every $\kappa$.  For $\mu = (\mu_j) \in X^*(\un{T})$, we say that an $n$-dimensional representation $W$ has Hodge--Tate weights $\mu$ if 
\[
\mathrm{HT}_{\sigma_j}(W) = \{ \mu_{1, j}, \mu_{2, j}, \ldots, \mu_{n, j} \}.
\]

An inertial type is a representation $\tau:I_K \ra \GL_n(E)$ with open kernel and which extends to $W_K$. We say that an $n$-dimensional potentially semistable representation $\rho:G_K \ra \GL_n(E)$ has type $(\mu, \tau)$ if $\rho$ has Hodge--Tate weights $\mu$ and the Weil-Deligne representation $\mathrm{WD}(\rho)$ restricted to $I_K$ is isomorphic to $\tau$. Note that this differs from the conventions of \cite{GHS} via a shift by $\eta$.     

Let $\mathrm{Art}_K: K^{\times} \ra W_K^{\mathrm{ab}}$ denote the Artin map normalized so that uniformizers correspond to geometric Frobenius elements. For $\tau$ an inertial type, we use $\sigma(\tau)$ to denote the finite dimensional smooth irreducible $\overline{\Qp}$-representation of $\GL_n(\cO_K)$ associated to $\tau$ by the ``inertial local Langlands correspondence'' (cf. \S \ref{sec:ill}).  In fact, in all situations, $\sigma(\tau)$ will be defined over $E$.

If $V$ is a finite length representation, then we use $\JH(V)$ to denote the set of Jordan--H\"older factors.



\section{Background}\label{sec:background}

\subsection{Affine Weyl group} \label{sec:awg}

Fix the dominant base alcove the apartment of $(\GL_n, T)$ which defines a Bruhat order on $W_a$ denoted by $\leq$.  If $\Omega$ is the stabilizer of the base alcove, then $\tld{W} = W_a \rtimes \Omega$ and so $\tld{W}$ inherits a Bruhat order in the standard way: For $\tld{w}_1, \tld{w}_2\in W_a$ and $\tld{w}\in \Omega$, $\tld{w}_1\tld{w}\leq \tld{w}_2\tld{w}$ if and only if $\tld{w}_1\leq \tld{w}_2$, and elements in different right $W_a$-cosets are incomparable. We also have the natural generalization $\un{\Omega}$ for $\un{G}$ and a Bruhat order on $\bun{\tld{W}}$. 

We now recall the definition of the admissible set as introduced by Kottwitz and Rapoport:

\begin{defn} \label{defn:adm} Let $\lambda_0 \in X^*(T)$.  Then define 
\[
\Adm(\lambda_0) := \left\{ \tld{w} \in \tld{W} \mid \tld{w} \leq t_{w(\lambda_0)} \text{ for some } w \in W(\GL_n) \right\}.
\]
Similarly, if $\lambda = (\lambda_j) \in X^*(\un{T})$, then define $\Adm(\lambda) = \prod_j \Adm(\lambda_j) \subset \bun{\tld{W}}$. 
\end{defn}

When working on the Galois side, it is natural to work with the partially ordered group $\tld{\bun{W}}^{\vee}$ (resp.~$\tld{W}^{\vee}$) which is identified with $\tld{\bun{W}}$ (resp.~$\tld{W}$) as a group, but whose Bruhat order, also denoted by $\leq$, is defined by the \emph{antidominant} base alcove.
For any character $\mu \in X^*(\un{T})$, define the subset $\Adm^{\vee}(\mu) \subset \tld{\bun{W}}^{\vee}$ as in Definition \ref{defn:adm}.

\begin{defn} \label{affineadjoint}   Define a bijection $\tld{w} \mapsto \tld{w}^*$ between $\tld{\bun{W}}^{\vee}$ and $\tld{\bun{W}}$ as follows: 
\begin{enumerate}
\item For $w = (w_j) \in W(\un{G})$, define $w^* := (w^*_j) \in W(\un{G})$ by $w^*_j = w_{f-1-j}^{-1}$;
\item For $\nu = (\nu_j) \in X^*(\un{T})$, define $\nu^* := (\nu^*_j) \in X^*(\un{T})$  by $\nu^*_j = \nu_{f-1-j}$;
\item For $\widetilde{w} = wt_\nu \in \bun{\widetilde{W}}^{\vee}$, define $\widetilde{w}^* \in \bun{\widetilde{W}}$ by $\widetilde{w}^* := t_{\nu^*} w^*$.
\end{enumerate}
Note that $\widetilde{w} \mapsto \widetilde{w}^*$ is an anti-homomorphism.  
By specializing to the case $f=1$, we obtain a bijective anti-homomorphism between $\tld{W}^{\vee}$ and $\tld{W}$.
\end{defn}

We now record a few basic lemmas for later.  

\begin{lemma}\label{lem:adjineq}
We have $\tld{w}_1 \leq \tld{w}_2$ in $\tld{W}^{\vee}$ if and only if $\tld{w}_1^* \leq \tld{w}_2^*$ in $\tld{W}$.
\end{lemma}
\begin{proof}
Suppose that $\tld{w}_2$ has a reduced expression $(\prod_{\alpha\in I} s_\alpha)\tau$ where each $s_\alpha$ is an affine reflection along a wall of the antidominant base alcove and $\tau$ stabilizes the antidominant base alcove.
Then $\tld{w}_2^*$ is the product 
\begin{equation}\label{eqn:redexp}
\left(\prod_{\alpha\in I} {^{\tau^*}s_\alpha^*} \right)\tau^*,
\end{equation}
where $^{\tau^*}s_\alpha^* = \tau^*s_\alpha^* (\tau^*)^{-1}$ (the order of factors indexed by $I$ should of course be reversed from the reduced factorization of $\tld{w}_2$).
It is easy to check that $\tau^*$ is in $\Omega$, and that each $^{\tau^*}s_\alpha^*$ is an affine reflection along a wall of the dominant base alcove.
From this, we see that $\ell(\tld{w}_2^*)\leq \ell(\tld{w}_2)$ (note that the lengths are with respect to different sets of generating reflections).
By symmetry, we see that $\ell(\tld{w}_2^*)= \ell(\tld{w}_2)$ so that (\ref{eqn:redexp}) is a reduced expression.

Since $\tld{w}_1 \leq \tld{w}_2$ in $\tld{W}^{\vee}$ if and only if $\tld{w}_1$ has a reduced expression $(\prod_{\alpha\in J} s_\alpha)\tau$ where $J$ is some subsequence of $I$ and similarly for $\tld{w}_1^*$, the result follows.
\end{proof}

\begin{lemma} \label{lem:adjadm} For $\tld{w} \in \tld{\bun{W}}^{\vee}$, we have $\widetilde{w} \in \Adm^{\vee}(\mu)$ if and only if $\widetilde{w}^* \in \Adm(\mu^*)$.  
\end{lemma}
\begin{proof}
This follows from Lemma \ref{lem:adjineq}.
 \end{proof} 

\begin{lemma} \label{bound} Let $\lambda\in X^*(\un{T})$ be a dominant weight. If $t_{\nu} s \in \Adm(\lambda)$, then
\[
\max \{ | \langle \nu, \alpha^{\vee} \rangle | \mid \alpha^{\vee} \in \un{R}^{\vee} \} \leq \max \{ |  \langle \lambda, \alpha^{\vee} \rangle | \mid \alpha^{\vee} \in \un{R}^{\vee} \}
\]
\end{lemma}
\begin{proof} We reduce immediately to the case $f =1$.  By a result of Kottwitz-Rapoport (cf.~ Theorem 3.2 in \cite{HN02}),  $t_{\nu} s$ is $\lambda$-permissible which says in particular that $\nu$ in the convex hull of Weyl group orbit of $\lambda$.   The claim is that the inequality holds for any $\nu$ in the convex hull.   For this, we can replace $\nu$ with the dominant representative in its Weyl group orbit.   Convexity then says that $\lambda-\nu$ is an $\R$-linear combination of positive roots where all of the coefficients are nonnegative.  If $\alpha^{\vee}_h$ is the highest positive coroot, then for any positive coroot $\alpha$, we have 
\[
\langle \nu, \alpha^{\vee} \rangle \leq \langle \nu, \alpha^{\vee}_h \rangle. 
\]
It suffices then to observe that $\langle \nu, \alpha^{\vee}_h \rangle \leq \langle \lambda, \alpha^{\vee}_h \rangle$ which follows from that fact that $ \alpha^{\vee}_h$ is dominant.  
\end{proof}

Recall that \emph{$p$-alcoves} of $\un{G}$ are defined to be the connected components of \[ (X^*(\un{T}) \otimes \mathbb{R})  \backslash \{ x \mid \alpha^{\vee}(x + \eta) = pm \}_{\alpha^{\vee} \in \un{R}^{\vee}, m \in \Z}.   \]
Define the collection of $p$-restricted dominant weights
\[
X_1(\un{T}) = \{ \lambda \in X^*(\un{T}) \mid 0  \leq \langle \lambda, \alpha^{\vee} \rangle \leq p-1 \text{ for all simple positive coroots } \alpha^{\vee} \}.
\]
A $p$-alcove $C$ is called \emph{$p$-restricted} if $C \cap X^*(\un{T}) \subset X_1(\un{T})$.  We say that $\lambda \in X_1(\un{T})$ is \emph{regular $p$-restricted} if furthermore $\langle \lambda, \alpha^{\vee} \rangle < p-1$ for all simple positive coroots $\alpha^{\vee}$. Recall also that 
\[
X^0(\un{T}) = \{ \lambda \in X^*(\un{T}) \mid \langle \lambda, \alpha^{\vee} \rangle = 0 \text{ for all coroots } \alpha^{\vee} \}.
\]

\begin{defn} \label{defn:dot} Define the dot action of $\bun{\tld{W}}$ on $X^*(\un{T}) \otimes \mathbb{R}$  by 
\[
\tld{w} \cdot x = (w t_{\nu}) \cdot x = w(x + \eta + p \nu) - \eta.
\] 
\end{defn}
In the literature, this is often thought of as an action of $W(\un{G}) \rtimes p X^*(\un{T})$, but it will be convenient to include the $p$-scaling in the definition.  
Recall that the group $\bun{W}_a$ acts simply transitively on the collection of $p$-alcoves. Let $\bun{C}_0$ denote the dominant base $p$-alcove, i.e., the alcove containing $\un{0}.$  

\begin{defn}
\label{dfn:deep}
Let $\lambda\in X^*(\un{T})$ be a weight. We say that $\lambda$ lies $m$-deep in its alcove if there exists integers $n_{\alpha}\in \Z$ such that $pn_{\alpha}+m<\langle\lambda+\eta,\alpha^\vee\rangle< p(n_{\alpha}+1)-m$ for all positive coroots $\alpha^\vee\in \un{R}^{\vee, +}$.
\end{defn}
For example, a dominant weight $\lambda=(\lambda_j) \in X^*(\un{T})$ is $m$-deep in $\bun{C}_0$ if  $m<\langle\lambda_j+\eta_0,\alpha_j^\vee\rangle< p-m$ for all $j=0,\dots,f-1$ and all positive coroots $\alpha_j^\vee \in R^{\vee, +}$.

\subsection{Tame types and Serre weights} \label{sec:ttypes}

We begin with some setup.  An \emph{inertial type} $\tau:I_K \ra \GL_n(E)$ is a representation with open kernel which extends to the Weil group of $K$.  An inertial type is \emph{tame} if it factors through tame inertia. All our tame types will be defined over $\cO$.   

Tame inertial types have a combinatorial description which we will now recall (cf.~ \cite[(6.15)]{herzig-duke} or \cite[Definition 8.2.2]{GHS}). Let $(w, \mu) \in W(\un{G}) \times X^*(\un{T})$.  As in \cite[(4.1)]{herzig-duke} (see also the paragraph preceding \cite[Definition 10.1.12]{GHS}), for any $(\nu, \sigma) \in X^*(\un{T}) \rtimes W(\underline{G})$, define
\begin{equation} \label{sigmaconj}
^{(\nu, \sigma)} (w, \mu) = (\sigma w \pi(\sigma)^{-1}, \sigma(\mu) + p \nu - \sigma w \pi(\sigma)^{-1}\pi(\nu))
\end{equation}
and we write $(w, \mu) \sim (w', \mu')$ if there exists $(\nu, \sigma)$ such that $^{(\nu, \sigma)} (w, \mu) = (w', \mu')$.

Let $r$ be the order of an element of $S_n$.  For any such $r$, we choose an embedding $\sigma_0'$ of the unramified extension $K'/K$ of degree $r$ into $E$ extending $\sigma_0$.
Let $e' = p^{fr}-1$, $e = p^f-1$ and $f'=fr$.
Using our choice of $e'$-th root of $(-p)$ in \S \ref{sec:not}, we get a fundamental character $\omega_{f'}:I_K \ra \cO^{\times}$ such that $\omega_{f'}^{\frac{e'}{e}} = \omega_f$.
The following describes all isomorphism classes of tame inertial types for $K$.

\begin{defn} \label{defn:tau} Define an inertial type $\tau(w, \mu):I_K \ra \GL_n(\cO)$ 
as follows:  If $w = (s_0, \ldots, s_{f-1})$, then set $s_{\tau} = s_0 s_{f-1} s_{f-2} \cdots s_1 \in W(\GL_n)$  and  $\bm{\alpha} \in X^*(\un{T})$ such that $\bm{\alpha}_0 = \mu_0$, $\bm{\alpha}_{j} = s_1^{-1} s_2^{-1} \ldots s_j^{-1}(\mu_j)$ for $1 \leq j \leq f-1$.  Let $r$ denote the order of $s_{\tau}$.  Then, 
\begin{equation} \label{pres1}
\tau(w,\mu) \defeq \bigoplus_{1 \leq i \leq n} \omega_{f'}^{\sum_{0 \leq k \leq r-1} \bf{a}^{(0)}_{s_{\tau}^{k}(i)} p^{fk}} 
\end{equation}
where $\bf{a}^{(0)} := \sum_{j =0}^{f-1}  \bm{\alpha}_{j} p^j\in \Z^n$. Note that $(w, \mu) \sim ((s_{\tau}, 1,\ldots, 1), \bm{\alpha})$ and $\tau(w, \mu) \cong \tau((s_{\tau}, 1,\ldots, 1), \bm{\alpha})$ by construction. 
\end{defn} 

For any $\cO$-valued inertial type $\tau$, we use $\ovl{\tau}:I_K \ra \GL_n(\F)$ to denote the reduction to the residue field.   Note that since $\omega_{f'}$ is the Teichm\"uller lift of its reduction to $\F$, for tame inertial types, $\ovl{\tau}$ determines $\tau$.


We say that a pair $(w,\mu)\in W(\un{G}) \times X^*(\un{T})$ is \emph{good} if $(T_w,\theta_{w,\mu})$ is maximally split (see \cite[\S 9.2]{GHS} for the definitions of $(T_w,\theta_{w,\mu})$ and maximally split).
This definition is consistent with \cite[Definition 6.19]{herzig-duke} by \cite[Proposition 6.20]{herzig-duke}.
As in \cite[\S 9.2]{GHS}, which follows \cite{Jantzen}, we attach a Deligne-Lusztig representation to a \emph{good} pair $(w, \mu) \in  W(\underline{G}) \times X^*(\un{T})$ which we denote by $R_w(\mu)$.  
For any tame representation $\ovl{\tau}:I_K  \ra \GL_n(\F)$ which extends to $G_K$, there is an associated $E$-valued $\GL_n(\cO_K)$-representation $V(\ovl{\tau})$ defined in \cite{GHS} Proposition 9.2.1  (in \emph{loc.~ cit.}~ it is denoted $V_{\phi}(\ovl{\tau})$).  By Proposition 9.2.3 of \cite{GHS} if  $\ovl{\tau} \cong \ovl{\tau}(w, \mu)$, then  
\begin{equation} \label{Vmap} 
V(\ovl{\tau}) = R_w(\mu).
\end{equation} 

\begin{rmk} 
The condition that $(w,\mu)$ is good guarantees that the Deligne-Lusztig representation $R_w(\mu)$ is a genuine representation (and not only virtual, see \cite[Proposition 9.2.1]{GHS}).
The genericity condition defined below will guarantee that $R_w(\mu)$ is in fact (absolutely) irreducible over $E$. 
\end{rmk}

\begin{lemma}\label{lemma:maxsplit}
Suppose that $\mu-\eta\in X^*(\bun{T})$ is in alcove $\bun{C}_0$. 
Then $(w,\mu)$ is good for any $w\in W(\bun{G})$.
\end{lemma}
\begin{proof}
Let $(T_w^*,s)$ be the $F^*$-stable maximal torus of $\bun{G}^*$ and semisimple element $s \in {T_w^*}^{F^*}$ corresponding to $(T_w,\theta_{w,\mu})$ as in \cite[\S 9.2(ii)]{GHS}.
As in the proof of \cite[Lemma 10.1.10]{GHS}, let $s'$ be $(g_{F^*(w^{-1})}^*)^{-1}sF(g_{F^*(w^{-1})}^*)$ (recall that $(g_{F^*(w^{-1})}^*)^{-1}F(g_{F^*(w^{-1})}^*) \in N(\bun{T}^*)$ represents $F^*(w^{-1})$ as in \S 9.2 of \emph{ibid.}).
By the proof of \emph{loc.~cit.}, if the Weyl group of $T_w^*$ in $Z_{\bun{G}^*}(s)$, which is isomorphic to $\Stab_W(s')$, is trivial, then the claim follows.
Suppose that $s_\alpha$ is in $\Stab_W(s')$ (which is generated by reflections).
Let $d$ be the order of $w\pi$ as an automorphism of $X^*(\bun{T})$ (in particular, $f$ divides $d$).
Then by the proof of \emph{loc.~cit.},
\begin{equation}\label{eqn:GHSsum}
\sum_{i=0}^{d-1} p^i\langle \mu,(w\pi)^i \alpha^\vee\rangle
\end{equation}
is divisible by $p^d-1$.
Since $\mu-\eta$ is in $\bun{C}_0$, this divisibility forces $\langle \mu,(w\pi)^i \alpha^\vee\rangle$ to be $p-1$ for all $i$ or $1-p$ for all $i$.
Thus either $(w\pi)^i\alpha^{\vee}$ are highest coroots in $\un{R}^{\vee}$ for all $i$, or they are lowest coroots for all $i$ (note that $\un{R}^{\vee}$ has exactly $f$ highest coroots).
This in turn implies that $(w\pi)^i\alpha^{\vee}=\pi^i\alpha^{\vee}$ for all $0\leq i\leq d-1$, by comparing the unique non-zero component on both sides. Since $d\geq f$, this implies that $w$ fixes $\pi^i\alpha$ for all $i$, and that $s_{\pi^i\alpha}\in \Stab_W(s')$ for all $i$.
We conclude that $\Stab_W(s')$ is $\langle s_{\pi^i \alpha} \rangle_i$ and centralizes $w$.
From this, we see that $Z_{\bun{G}^*}(s)$ is isomorphic to $\Res_{k/\F_p} \GL_2 \times \bun{T}'$ for some torus $\bun{T}'$ and $T_w^*$ is $\bun{T}_2 \times \bun{T}'$ where $\bun{T}_2 \subset \Res_{k/\F_p} \GL_2$ is a maximally split torus.
Then by definition, $(T_w,\theta_{w,\mu})$ is maximally split. 
\end{proof}

\begin{prop} \label{prop:relabeltype} 
Let $(w,\mu)$ and $(w',\mu')$ be in $W(\bun{G})\times X^*(\bun{T})$.
If $(w, \mu) \sim (w', \mu')$, then the tame inertial types $\tau(w, \mu)$ and $\tau(w', \mu')$ are isomorphic.
If $\mu-\eta$ and $\mu'-\eta$ are in alcove $\bun{C}_0$ and the tame inertial types $\tau(w, \mu)$ and $\tau(w', \mu')$ are isomorphic, then $(w, \mu) \sim (w', \mu')$. 
\end{prop}
\begin{proof}
The first part follows from a direct computation.
For the second part, $(T_w,\theta_{w,\mu})$ and $(T_{w'},\theta_{w',\mu'})$ are maximally split by Lemma \ref{lemma:maxsplit}. 
The second part now follows from \cite[Proposition 9.2.1]{GHS}. 
\end{proof}


\begin{defn}
\label{defi:generic}
Let $\tau$ be a tame inertial type.
\begin{enumerate}
\item Define $\tau$ to be \emph{$m$-generic} if there is an isomorphism $\tau \cong \tau(s, \lambda + \eta)$ for some $s \in W(\un{G})$ and $\lambda \in X^*(\un{T})$ which is $m$-deep in alcove $\un{C}_0$. 

\item Define $\rhobar:G_K \ra \GL_n(\F)$ to be $m$-generic if $\rhobar^{\mathrm{ss}}|_{I_K} \cong \ovl{\tau}(s, \lambda + \eta)$ for $\lambda \in X^*(\un{T})$ which is $m$-deep in alcove $\un{C}_0$. 

\item \label{defi:ngeneric} We say that $\tau$ (resp. $\rhobar$) is \emph{generic} if it is $n$-generic (resp.~ $3n$-generic). 
\item A \emph{lowest alcove presentation} of $\tau$ is a pair $(s, \mu) \in W(\un{G}) \times X^*(\un{T})$ where $\mu \in \bun{C}_0$ such that $\tau \cong \tau(s, \mu + \eta)$ (which by definition exists exactly when $\tau$ is 0-generic).
\end{enumerate}
\end{defn}
\begin{rmk} The results in \S 3.2 hold for $\tau$ which are $n$-generic and $\rhobar$ which are $2n$-generic. For \S 3.4, we will use $2n-1$-generic to control monodromy condition. 
In most of \S4, $\rhobar$ will be $6n-2$-generic because of representation theoretic input and a reliance on \cite{Enns} to eliminate weights near the alcove walls. For example, Proposition \ref{prop:JH} requires that the type to be $2n$-generic which combined with \cite{Enns} forces $\rhobar$ to be $6n-2$-generic.  A more careful analysis would likely improve this bound.
\end{rmk}

\begin{rmk}
Since the bounds for genericity do not depend on $p$, as $p$ gets large, most semisimple representations will be generic. 
\end{rmk}

Concretely, $\tau$ is $m$-generic, if there exists an isomorphism $\tau \cong \tau(s, \lambda + \eta)$ with $\lambda_j + \eta_j=(a_{1,j}, a_{2,j}, \ldots, a_{n, j})$ 
\[
m < |a_{i, j} - a_{k, j}| < p - m, \quad \text{for all } 1 \leq i\neq k \leq n. 
\]

\begin{rmk} \label{rmk:comparison} The notion of generic here is slightly stronger than that of \cite{LLLM} when $n =3$ and \cite{Enns} in general.  Comparing with Definition 2.1 in \cite{LLLM} and Definition 2 in \cite{Enns}, we see that if $\tau$ is $m$-generic as in Definition \ref{defi:generic} then it is $m$-generic in the sense of \cite{LLLM, Enns}.  If it is $m$-generic in the sense of \cite{LLLM, Enns}, then it is $m-1$ generic in the sense of Definition \ref{defi:generic}. The difference being that the first inequality above is strict. Both \cite{LLLM, Enns} state genericity as a condition on a presentation as in (\ref{pres1}), that is, as a condition on the $\bm{\alpha}_j$'s.  To compare the two definitions, note that if $(s, \mu)$ is a lowest alcove presentation then $\bm{\alpha}_{j}$ is in the Weyl group orbit of $\mu_j + \eta_j$.   
\end{rmk}

\begin{defn} \label{defn:regular} We say that a tame inertial type $\tau$ is \emph{regular} if the characters appearing in $\tau$ as in (\ref{pres1}) are pairwise distinct.  
\end{defn}
Note that 1-generic implies regular but regular is a weaker condition. 

\begin{defn} \label{defn:SW} A \emph{Serre weight} is an irreducible representation of $\GL_n(\F_{p^f})$.  If $\lambda \in X_1(\un{T})$ is a $p$-restricted dominant weight, then $F(\lambda)$ denotes the associated Serre weight which is the socle of be the (dual) Weyl module, $W(\lambda)$, with highest weight $\lambda$. A Serre weight $V$ is \emph{regular} if $V \cong F(\lambda)$ for a regular $p$-restricted weight $\lambda$ (see discussion after Lemma \ref{bound}). 
\end{defn} 

We recall that the map $\lambda \mapsto F(\lambda)$ induces a bijection between $X_1(\bun{T})/(p-\pi)X^0(\bun{T})$ and the set of isomorphism classes of Serre weights (see \cite[Lemma 9.2.4]{GHS}).

Let $w_0$ denote the longest elements of $W(\un{G})$.  Recall the self-bijection on regular $p$-restricted weights defined by $\lambda \mapsto w_0 \cdot (\lambda - p \eta)$ which induces a map $\cR$ on regular Serre weights (pg. 54 of \cite{GHS}).  If we let $\tld{w}_h := w_0 t_{-\eta}$,  then 
\[
\cR(F(\lambda)) = F(w_0 \cdot (\lambda - p \eta)) = F(\tld{w}_h \cdot \lambda).
\]
Note that $\tld{w}_h \cdot \bun{C}_0$ is the highest $p$-restricted $p$-alcove.  

We are now ready to state the Serre weight recipe following \cite{herzig-duke, GHS}.
\begin{defn} \label{defn:SWC} (cf.~ \cite{GHS}, Definition 9.2.5)  Let $\rhobar$ be a generic semisimple $n$-dimensional representation of $\Gal(\overline{K}/K)$.   Then, 
\[
W^{?}(\rhobar) := \{ \cR (F(\lambda)) : F(\lambda) \text{ is a Jordan--H\"older constituent of } \ovl{V}(\rhobar|_{I_K}) \}
\]
where $\ovl{V}(\rhobar|_{I_K})$ denotes the semisimplification of a reduction modulo $\varpi$ of $V(\rhobar|_{I_K})$.
Define $W_{\mathrm{obv}}(\rhobar) \subset W^{?}(\rhobar)$ as in \cite{GHS}, Definition 7.1.3.  
\end{defn}
We give another characterization of the obvious weights:
\begin{prop} \label{prop:obvwt}
Let $\rhobar$ be generic and semisimple.   Then for $\lambda \in X_1(\un{T})$,  $F(\lambda) \in W_{\mathrm{obv}}(\rhobar)$ if and only $\rhobar|_{I_K} \cong \ovl{\tau}(w, \lambda + \eta)$ for some $w\in W(\un{G})$.   
\end{prop} 
\begin{proof} See Proposition 9.3.7 in \cite{GHS}. 
\end{proof}

\begin{cor} \label{cor:obvweight} Let $\rhobar$ be $n$-generic and semisimple.   Then $|W_{\mathrm{obv}}(\rhobar)| = (n!)^f$.  
\end{cor}
\begin{proof} Since $\rhobar$ is $n$-generic, $\rhobar|_{I_K} \cong \ovl{\tau}(s, \mu + \eta)$ for some $\mu$ which is $n$-deep in $\bun{C}_0$.
We define a map 
\begin{align*}
W(\bun{G}) &\ra W_{\mathrm{obv}}(\rhobar)\\
\sigma &\mapsto F(\mu')
\end{align*}
where $(\sigma s \pi(\sigma)^{-1},\mu' + \eta) = ^{(\nu, \sigma)} (s, \mu + \eta)$ and $\nu\in X^*(\bun{T})$ is the unique element up to $X^0(\bun{T})$ such that $(t_{\nu} \sigma)\cdot \bun{C}_0$ is $p$-restricted. 
Explicitly,
\begin{equation} \label{eqn:change}
\mu' = \sigma(\mu + \eta) + p \nu - \sigma s \pi(\sigma)^{-1} \pi(\nu) - \eta = (t_{\nu} \sigma) \cdot \mu - \sigma s \pi(\sigma)^{-1} \pi(\nu).
\end{equation}
The inequality $|\langle \nu,\alpha^\vee\rangle|\leq |\langle \eta, \alpha^\vee\rangle| < n$ from Lemma \ref{bound} coupled with the fact that $\mu$ is $n$-deep in $\bun{C}_0$ implies that $\mu'$ is in the alcove $(t_{\nu} \sigma)\cdot \bun{C}_0$.  By Proposition \ref{prop:obvwt}, $F(\mu')$ is in $W_{\mathrm{obv}}(\rhobar)$, so that our map lands in the claimed set.

We now prove surjectivity.   Consider any pair $(\nu, \sigma)$ and let $\mu'$ be such that $(\sigma s \pi(\sigma)^{-1},\mu' + \eta) = ^{(\nu, \sigma)} (s, \mu + \eta)$.   Assume that $\mu'$ is $p$-restricted. Any obvious weight is of this form by Propositions \ref{prop:relabeltype} and \ref{prop:obvwt}.   We claim then that $(t_{\nu} \sigma) \cdot \bun{C}_0$ is $p$-restricted and so $(\nu, \sigma)$ is one of the pairs above.  Since $\mu'$ is $p$-restricted, for any coroot $\alpha^{\vee},$
\[
0 \leq |\langle \mu' + \eta, \alpha^{\vee} \rangle| = | \langle \sigma(\mu + \eta) + p \nu - \sigma s \pi(\sigma)^{-1} \pi(\nu), \alpha^{\vee} \rangle | \leq p(n-1).
\]
Since $\mu$ is $n$-deep in alcove $\bun{C}_0$, we have $n < | \langle \sigma(\mu + \eta), \alpha^{\vee} \rangle | < p- n$.   We conclude that 
\[
| \langle p \nu - \sigma s \pi(\sigma)^{-1} \pi(\nu), \alpha^{\vee} \rangle | < n(p-1).
\] 
Set $M = \max_{\alpha^{\vee}} \{ |\langle \nu, \alpha^{\vee}\rangle| \}$.   We deduce that $(p-1)M < n(p-1)$ and so $M < n$.   
Since $M < n$ and $(t_{\nu} \sigma) \cdot \mu$ is $n$-deep in its alcove, we deduce that $\mu'$ lies in the alcove $(t_{\nu} \sigma) \cdot \bun{C}_0$.   Since $\mu'$ is $p$-restricted so is $(t_{\nu} \sigma) \cdot \bun{C}_0$.  

For injectivity,
suppose that $^{(\nu, \sigma)} (s, \mu + \eta) \equiv ^{(\nu', \sigma')} (s, \mu + \eta) \mod (p-\pi)X^0(\bun{T})$.
Then restricting the part in $X^*(\bun{T})$ to $\bun{Z}$, we see that $(p-\pi)\nu|_{\bun{Z}} \equiv (p-\pi)\nu'|_{\bun{Z}} \mod (p-\pi)nX^*(\bun{Z})$.
Since $p-\pi$ acts injectively on $X^*(\bun{Z})$, we deduce that $\nu|_{\bun{Z}} \equiv \nu'|_{\bun{Z}} \mod nX^*(\bun{Z})$, so after modifying $\nu'$ by an element in $X^0(\bun{T})$ we can now assume $\nu-\nu'\in \un{\Lambda}_R$.

As observed at the beginning of the proof, the fact that $\mu$ is $n$-deep in $\bun{C}_0$ implies that the part in $X^*(\bun{T})$ of $ ^{(\nu, \sigma)} (s, \mu + \eta)$ lies in alcove $t_\nu\sigma\cdot \bun{C}_0$. Thus the above equality implies an equality of alcoves $(t_\nu\sigma)\cdot \bun{C}_0=(t_{\nu'}\sigma')\cdot \bun{C}_0$.
Combining this with the decomposition $\tld{\bun{W}} = \bun{W}_a \rtimes \un{\Omega}$ shows that $\nu=\nu'$ and $\sigma=\sigma'$, thus giving what we want. 
\end{proof}

\begin{defn} \label{defn:obvweight} We say that $F(\lambda) \in W_{\mathrm{obv}}(\rhobar)$ is the obvious weight \emph{associated} to $w \in W(\un{G})$ if $w$ maps to $F(\lambda)$ in the bijection from the proof of Corollary \ref{cor:obvweight} (note that this depends on the choice of $(s,\mu)$ such that $\rhobar|_{I_K} \cong \ovl{\tau}(s, \mu + \eta)$ and $\mu$ is $n$-deep in $C_0$).
\end{defn}

\begin{prop}\label{prop:alcoveC0}  Assume $\tau$ is $d$-generic with $d \geq 1$. Then for any lowest alcove presentation $(s', \mu')$ of $\tau$, $\mu'$ lies $d-1$-deep in  alcove $\bun{C}_0$.   Furthermore,  if $(s, \mu)$ is a fixed lowest alcove presentation of $\tau$, the collection of lowest alcove presentations are given by 
\[
\{ (s', \mu') \in W(\un{G}) \times X^*(\bun{T}) \mid (s', \mu' + \eta) = ^{(\nu, \sigma)} (s, \mu + \eta),  t_{\nu} \sigma \in \un{\Omega} \}.   
\]
\end{prop}
\begin{proof}  
The proof is similar to that of Corollary \ref{cor:obvweight}.  Fix a presentation $(s, \mu)$ with $\mu$ $d$-deep in $\bun{C}_0$.   For any other presentation $(s', \mu')$,  $(s', \mu' + \eta) = ^{(\nu, \sigma)} (s, \mu + \eta)$ by Proposition \ref{prop:relabeltype} so 
\[
\mu' = (t_{\nu} \sigma) \cdot \mu - \sigma s \pi(\sigma)^{-1} \pi(\nu)
\]
as in \eqref{eqn:change}.  Since $\mu' \in \bun{C}_0$, for all $\alpha^{\vee}$, 
\[
0 < | \langle \sigma(\mu + \eta) + p \nu - \sigma s \pi(\sigma)^{-1} \pi(\nu), \alpha^{\vee} \rangle | \leq p-1.
\]
The same argument in Corollary \ref{cor:obvweight} shows that if $M = \max_{\alpha^{\vee}} \{ |\langle \nu, \alpha^{\vee}\rangle| \}$, then $M < 2$.   

Thus, if $\mu$ is $d$-deep in $\bun{C}_0$ then $\mu' = (t_{\nu} \sigma) \cdot \mu - \sigma s \pi(\sigma)^{-1} \pi(\nu)$ is $d-1$-deep in $(t_{\nu} \sigma) \cdot \bun{C}_0$.   Since $\mu'$ is in $\bun{C}_0$, we see that $t_{\nu} \sigma \in \un{\Omega}$ and $\mu'$ is $d-1$-deep in $\bun{C}_0$. We have thus shown that all lowest alcove presentations occur in the set described in the statement of the Proposition.

Finally, we observe that any pair $(s',\mu')$ in the set described in the Proposition indeed gives a lowest alcove presentation. This is because if  $t_{\nu} \sigma \in \un{\Omega}$, then $|\langle \nu,\alpha^\vee \rangle|\leq 1$. Using that $\mu$ is at least $d$-deep in $\bun{C}_0$ and $d\geq 1$, we conclude that $\mu'$ is also in $\bun{C}_0$.

\end{proof}

\begin{prop} \label{prop:lambdagen} Let $\lambda$ be a weight which is $d$-deep in a $p$-restricted alcove $\bun{C}$ with $d \geq n$.   Then for any $s \in W(\un{G})$ and any lowest alcove presentation $(s',\mu')$ of $\tau(s, \lambda + \eta)$, $\mu'$ is $d-n+1$ deep in $\bun{C}_0$. Moreover, $\tau(s, \lambda + \eta)$ is at least $d-n+1$-generic.
\end{prop}
\begin{proof} The same argument as in Corollary \ref{cor:obvweight} shows that $(s',\mu' + \eta) = ^{(\nu, \sigma)} (s, \lambda + \eta)$ with $(t_{\nu} \sigma) \cdot \bun{C} = \bun{C}_0$. The result now follows because $|\langle \nu, \alpha^{\vee}\rangle| \leq n-1$ for all coroots $\alpha^{\vee}$.  
\end{proof} 


\subsection{Inertial local Langlands} \label{sec:ill}

In this section, we establish some simple instances of the inertial local Langlands following \S 2.4 of \cite{EGH}. We let $B_n\subset \GL_n$ be the Borel subgroup consisting of upper triangular matrices. Fix an isomorphism $i:\overline{\mathbb{Q}}_p \cong \mathbb{C}$.  As above let $K \subset \overline{\Q}_p$ be the unramified extension of degree $f$.  Let $\rec_{K, \mathbb{C}}$ denote the local Langlands correspondence for $\GL_n(K)$ of \cite{harris-taylor}.   Using $i$, define a local Langlands correspondence $\rec_K$ over $\overline{\Q}_p$ such  $i \circ \rec_K = \rec_{K, \mathbb{C}} \circ i$. We first recall the existence statement:
\begin{thm} \label{thm:ill}  Let $\tau$ be an inertial type, then there is a finite-dimensional smooth irreducible $\overline{\Q}_p$-representation $\sigma(\tau)$ of $\GL_n(\cO_K)$ such that if $\pi$ is any irreducible smooth $\overline{\Q}_p$-representation of $\GL_n(K)$ then $\pi|_{\GL_n(\cO_K)}$ contains a unique copy of $\sigma(\tau)$ as a subrepresentation if and only if $\rec_K(\pi)|_{I_K} \cong \tau$ and $N = 0$ on $\rec_K(\pi)$. 

If $\pi$ is an irreducible smooth $\ovl{\Q}_p$-representation of $\GL_n(K)$ such that 
\[\Hom_{\GL_n(\cO_K)}(\Ind_{B_n(k)}^{\GL_n(k)} \boxtimes_{i=1}^n \chi_i,\pi) \neq 0,\]
then $\rec_K(\pi)|_{I_K} \cong \oplus_{i=1}^n \chi_i \circ \mathrm{Art}^{-1}_K|_{I_K}$. 
\end{thm}
\begin{proof}
The first part is \cite[Theorem 3.7]{ceggps}.
We now prove the second part.
One proves as in the proof of \cite[Proposition 2.4.1(ii)]{EGH} that $\pi$ is a subrepresentation of $\mathrm{n-}\Ind_{B_n(K)}^{\GL_n(K)} \delta_{B_n}^{-1/2}\boxtimes_{i=1}^n \tld{\chi}_i$, for some $\tld{\chi}_i$ extending $\chi_i$ as in \emph{loc.~cit.}  
Then the result follows as in \emph{loc.~cit.}~ (though $N(\pi)$ is not necessarily zero in our context).
\end{proof}

What we will need is an explicit $\sigma(\tau)$ in the case when $\tau$ is a tame inertial type.   Let $K'/K$ denote an unramified extension of degree $r$ with residue fields $k'/k$.   A character $\theta:k'^{\times} \ra \overline{\Q}_p^{\times}$ is \emph{primitive} if all its $\Gal(k'/k)$-conjugates are distinct.   Following \S 4, \cite{herzig-duke}, let $\kappa(\theta) = (-1)^{r-1} R^{\theta}_{T_w}$ denote the cuspidal representation of $\GL_r(k)$ parametrized by $\theta$.   

\begin{prop} \label{illcusp} Let $\theta:k'^{\times} \ra \overline{\Q}_p^{\times}$ be a primitive character. Let $\tau = \oplus_{i=0}^{r-1} (\theta \circ \mathrm{Art}_{K'}^{-1})^{(i)}|_{I_K}$ where $\chi^{(i)}$ denote the $i$th Frobenius twist.  Then $\sigma(\tau)$ can be taken to be $\kappa(\theta)$ interpreted as a $\GL_r(\cO_K)$-representation via the reduction map $\GL_r(\cO_K) \ra \GL_r(k)$.   
\end{prop} 
\begin{proof} See \cite[Proposition 2.4.1(i)]{EGH}.  
\end{proof}

\begin{defn} \label{pind} Let $n = \sum_{j=1}^{\ell} r_j$ be a partition of $n$.   For each $j$, let $\theta_j$ be a primitive character for the extension of degree $r_j$ of $k$.  Define $\mathrm{PInd}(\kappa(\theta_1), \ldots, \kappa(\theta_{\ell}))$ to be the parabolic induction to $\GL_n(k)$ of $\otimes_j \kappa(\theta_j)$ as a representation of the rational points of a parabolic $P \supset B$ with Levi subgroup $\prod_j \GL_{r_j}$.  
\end{defn}

\begin{prop} \label{illtame} Let $\tau = \oplus_{j=1}^{\ell} \tau_j$ where $\tau_j$ is a cuspidal inertial type associated to primitive characters $\theta_j$ of degree $r_j$ as in Proposition \ref{illcusp}. Assume that the cuspidal types $\kappa(\theta_j)$ are pairwise distinct.  Then $\sigma(\tau)$ can be taken to be $\mathrm{PInd}(\kappa(\theta_1), \ldots, \kappa(\theta_{\ell})).$
\end{prop} 
\begin{proof}
This follows from \S 6 of \cite{PZ} where $\sigma(\tau)$ is constructed as $\sigma_{\cP}(\lambda)$ for maximal $\cP$ (see also \S 3.6 of \cite{ceggps}).  In the case of principal series, see also Proposition 2.4.1(ii) in \cite{EGH}.  
\end{proof}

\begin{cor} \label{explicitill}  Let $\tau = \tau(w, \mu)$ be a regular tame inertial type $($Definition $\ref{defn:regular}).$   Then $\sigma(\tau)$ can be taken to be $R_w(\mu)$.    
\end{cor}
\begin{proof}  As in the Definition \ref{defn:tau}, we immediately reduce to the case that $w = (s_{\tau}, \id,\ldots, \id)$.   The condition of being regular corresponds to the pairwise distinctness condition in Proposition \ref{illtame}.   Finally, we use Lemma 4.7 of \cite{herzig-duke} to relate the parabolic induction in Proposition \ref{illtame} to the Deligne-Lusztig representation $R_w(\mu)$.
\end{proof}

From now on, for any regular tame inertial type, we let $\sigma(\tau)$ be as in Corollary \ref{explicitill}.  

\begin{rmk}  Note that for a regular tame inertial type $\tau$,  by Corollary  \ref{explicitill} and \eqref{Vmap}
\[
V(\ovl{\tau}) \cong \sigma(\tau). 
\] 
\end{rmk}

\section{Local results}\label{sec:local}

In this section, we prove the main results on reductions of potentially crystalline representations which will be used for weight elimination in \S \ref{sec:we}.  

\subsection{\'Etale $\varphi$-modules}\label{sec:phi}
In this section, we consider \'etale $\varphi$-modules associated to affine Weyl group elements and determine their corresponding semisimple Galois representations.  The key result is Proposition \ref{Galoistotype} which relates the Galois representation to the tame types defined in \S \ref{sec:ttypes}.

Let $\cO_{\cE}$ denote the $p$-adic completion of $\fS[\frac{1}{v}]$, where $\fS := \cO_K[\![v]\!]$ is endowed with a Frobenius morphism $\phz$ extending the Frobenius on $\cO_K$ such that $\phz(v) = v^p$. Let $R$ be a local, complete Noetherian $\cO$-algebra with finite residue field. By base change, the ring $\cO_{\cE}\widehat{\otimes}_{\Zp}R$ is naturally endowed with a Frobenius endomorphism $\phz$ and we write $\Phi\text{-}\Mod^{\text{\'et}}(R)$ for the category of \'etale $(\phz,\cO_{\cE}\widehat{\otimes}_{\Zp}R)$-modules. Its objects are finite type projective modules $\cM$ over $\cO_{\cE}\widehat{\otimes}_{\Zp}R$, endowed with a Frobenius semilinear endomorphism $\phi_{\cM}:\cM\ra\cM$ inducing an isomorphism on the pull-back: $\Id\otimes_{\phz}\phi_{\cM}:\phz^*(\cM)\stackrel{\sim}{\longrightarrow}\cM$. 

Since $K_{\infty}/K$ is totally wildly ramified, the subgroup $G_{K_{\infty}}$ of $G_K$ projects surjectively to the tame quotient of $G_K$. Hence the restriction map 
\[
\Rep_{\F}^t(G_K) \rightarrow \Rep_{\F}(G_{K_{\infty}})
\]   
is fully faithful where $\Rep^t$ denotes subcategory of tame representation. We use $\Rep^t_{\F}(G_{K_{\infty}})$ to denote the essential image of this map and will often implicitly identify these representations of $G_{K_{\infty}}$ with their canonical extensions to $G_K$. Note that this essential image contains exactly representations of $G_{K_\infty}$ which are trivial on $G_{K_\infty}\cap G_{K^t}$, where $K^t$ is the maximal tamely ramified extension of $K$.

For any complete local Noetherian $\cO$-algebra $R$ with finite residue field, by theory of norm fields, there is anti-equivalence of categories 
\[
\bV^*:\Phi\text{-}\Mod^{\text{\'et}}(R) \ra \Rep_R(G_{K_{\infty}})
\]
(cf.~ Lemma 1.2.7 \cite{MFFGS} for version with coefficients). If $K'$ is finite unramified extension of $K$, let $K'_{\infty} = K_{\infty} \otimes_K K'$, and we can similarly consider the category of \'etale $\varphi$-modules over $K'$ denoted $\Phi\text{-}\Mod^{\text{\'et}}_{K'}(R)$ together with the anti-equivalence $\bV^*_{K'}:\Phi\text{-}\Mod^{\text{\'et}}_{K'}(R) \ra \Rep_R(G_{K'_{\infty}})$.

For any $(\cM, \phi_{\cM}) \in \Phi\text{-}\Mod^{\text{\'et}}(R)$, we decompose $\cM = \oplus_{j} \cM^{(j)}$ over embeddings $\sigma_j:W(k) \ra \cO$ with the induced maps $\phi^{(j)}_{\cM}:\cM^{(j)} \ra \cM^{(j+1)}$.   We are now ready to define ``semisimple'' \'etale $\varphi$-modules. We fix an embedding $\widetilde{W}^{\vee} =  X^*(T) \rtimes W(\GL_n) \iarrow N_{\GL_n}(T)(\F(\!(v)\!))$ given by $\mu \mapsto v^{\mu}$ and identifying $W(\GL_n)$ with the subgroup of permutation matrices.
Here, for $\mu = (a_1,\ldots,a_n)$ we define $v^\mu$ to be the diagonal matrix with entries $v^{a_i}$ (we interpret $X^*(T)$ as the group of cocharacters of the dual torus and $v^\mu$ is the associated cocharacter evaluated at $v$).

\begin{defn} \label{defn:Wphi}  For any $\tld{w} = (\tld{w}_j) \in \bun{\widetilde{W}}^{\vee}$ and $D = (D_j) \in \un{T}(\F)$, define $\cM(\tld{w}, D) \in \Phi\text{-}\Mod^{\text{\'et}}(\F)$ to be the free module over $\cO_{\cE} \otimes_{\Zp} \F$ of rank $n$ such that the Frobenius $\phi^{(j)}_{\cM}$ is given by the matrix $D_j \tld{w}_j \in  N_{\GL_n}(T)(\F(\!(v)\!))$ (with respect to the standard basis). Set $\cM(\tld{w}) := \cM(\tld{w}, \mathrm{Id})$.  
\end{defn}
\begin{prop} \label{Galoistotype} Let $\cM(\tld{w}, D)$ be as in Definition \ref{defn:Wphi} with $\tld{w} = st_{\mu} \in \bun{\widetilde{W}}^{\vee}$.  Then,  $\bV^*(\cM(\tld{w}, D))$ lies in $\Rep^t_{\F}(G_{K_{\infty}})$ and  
\[
\bV^*(\cM(\tld{w}, D))|_{I_K} \cong \ovl{\tau}(s^*, \mu^*).
\]
In particular, the restriction to inertia doesn't depend on $D$. 
\end{prop} 
\begin{proof}
Assume that $(s_1^*, \mu^*_1) \sim (s_2^{*}, \mu^*_2)$ via conjugation by $(\nu^*, \sigma^*)$. Then for any $D_1 \in \un{T}(\F)$, there exists a $D_2 \in \un{T}(\F)$ such that $\phz$-conjugation by $t_{\pi^{-1}(\nu)} \sigma$ induces an isomorphism $\cM(s_1 t_{\mu_1}, D_1) \cong \cM(s_2 t_{\mu_2}, D_2)$ by \eqref{sigmaconj} . For any $\tld{w}$, by an appropriate conjugation, we can assume $\tld{w}_j = 1$ for $j \neq f-1$.   Let $\tld{w}_{f-1} = s_0^{-1} t_{\mu_0}$ with $\mu_0 = (a_1, \ldots, a_n) \in \Z^n$, and let $r$ denote order of $s_0$. Consider the base change 
\[
\cM' := \cM(\tld{w}, D) \otimes_{W(k)} W(k')
\]
where $k'/k$ is finite extension of degree $r$. Let $\phi_{\cM'}$ denote the Frobenius on $\cM'$.  A straightforward computation as in \cite[Lemma 3.2.3]{LLLM2} shows that there exists a basis $(e_i)$ for $(\cM')^{(0)}$ (the piece corresponding to fixed embedding $\sigma_0':W(k') \ra \cO$) such that  
\[
\phi_{\cM'}^{fr}(e_i) =  d_i v^{\sum_{m=0}^{r-1} a_{s_0^{m+1}(i)} p^{fm}} e_i
\]
for some scalars $d_i$ determined by $D$. 
Following Proposition 2.1.7 in \cite{CDM} and using our choice of embedding $\sigma_0':W(k') \ra \cO$, one can determine the $\bV^*_{K'}(\cM')$ from the $\phi_{\cM'}^{fr}$.  If $K'_{\infty} = K_{\infty} \otimes_{W(k)} W(k')$, then there are unramified characters $\xi_i$ such that
\[
\bV^*_{K'}(\cM') \cong \left( \bigoplus_{i=1}^n \xi_i \, \omega_{rf}^{\sum_{m=0}^{r-1} p^{fm} a_{s_0^{m+1}(i)}} \right)|_{G_{K'_{\infty}}}.
\]
Hence $\bV^*(\cM(\tld{w}, D))$ is tame and by comparison with \eqref{pres1}, 
\[
\bV^*(\cM(\tld{w}, D))|_{I_K} \cong \ovl{\tau}((s_0, 1, \ldots, 1), (\mu_0, 0, \ldots, 0)).
\]
\end{proof} 

\subsection{Semisimple Kisin modules}

In this subsection, we generalize \cite{LLLM} from $\GL_3$ to $\GL_n$ and study reductions of Kisin modules with descent. For the convenience of the reader, we first state the main theorem which is used for weight elimination in \S 4. The theorem will be a consequence of Theorem \ref{fixedpoints} about reductions of Kisin modules with descent data.  The proof appears at the end of the subsection.

\begin{thm} \label{thm:admcrit} Let $\tau \cong \tau(s,\mu+ \eta)$ be 1-generic with lowest alcove presentation $(s, \mu)$  and let $\lambda \in X^*(\un{T})$. Let $\rhobar$ be the reduction of a potentially crystalline representation of type $(\lambda, \tau)$.  Assume either $(1)$ $\tau$ is a principal series type and $ \rhobar^{\mathrm{ss}}$ is a direct sum of characters or $(2)$ $\lambda = \eta$ and $\tau$ is generic.   Then there exists $(w, \nu) \in W(\un{G}) \times X^*(\un{T})$ such that
\[
 \rhobar^{\mathrm{ss}}|_{I_K} \cong \ovl{\tau}(w, \nu + \eta) \text{ and }  s^{-1} t_{\nu - \mu} w \in \Adm(\lambda).
\] 
\end{thm}
\begin{rmk}
The element $s^{-1} t_{\nu - \mu} w$ in case (2) of Theorem \ref{thm:admcrit} is later defined to be $\tld{w}^*(\rhobar,\tau)$ in Definition \ref{defn:relshape}.
\end{rmk}
We continue to use the notation of \S \ref{sec:ttypes}.
Let $\tau:I_K \ra \GL_n(E)$ be a tame inertial type.
We will assume throughout that $\tau$ is 1-generic and fix a lowest alcove presentation $(s, \mu)$  where $\mu \in \bun{C}_0$ (i.e., $\tau \cong \tau(s, \mu + \eta)$). 

If $s = (s_0, \ldots, s_{f-1})$ and $\mu = (\mu_j)_{0 \leq j \leq f-1} \in X^*(\un{T})$, we take $s_\tau \defeq s_0 s_{f-1} s_{f-2} \cdots s_1 \in W(\GL_n)$ and  $\bm{\alpha}_{(s, \mu)} \in X^*(\un{T})$ such that $\bm{\alpha}_{(s,\mu), j} = s_1^{-1} s_2^{-1} \ldots s_j^{-1}(\mu_j + \eta_j)$ for $1 \leq j \leq f-1$ and $\bm{\alpha}_{(s, \mu), 0} = \mu_0 + \eta_0$.  Let $r$ denote the order of $s_{\tau}$ and $f' = rf$.  As in Definition \ref{defn:tau}, $\tau(s, \mu+ \eta) \cong \tau((s_{\tau}, 1, \ldots, 1), \bm{\alpha}_{(s, \mu)})$ and concretely, 
\begin{equation} \label{pres}
\tau \cong \oplus_{i=1}^{n} \chi_i \text{ with } \chi_i :=  \omega_{f'}^{\sum_{0 \leq k \leq r-1} \bf{a}^{(0)}_{(s, \mu),s_{\tau}^{k}(i)} p^{fk}}
\end{equation}
where $\bf{a}^{(0)}_{(s, \mu)} = \sum_{j=0}^{f-1} \bm{\alpha}_{(s, \mu), j}  p^j  \in \Z^n$ (compare with \eqref{pres1}). By fixing a choice of lowest alcove presentation, we also fix the order of the characters $\chi_i$ as above.  
 
\begin{rmk} \label{rmk:compare}  In \cite{LLLM}, the notion of lowest alcove presentation does not appear. Everything is written for presentations of the form $\tau((s_{\tau}, 1, \ldots, 1), \bm{\alpha}_{(s, \mu)})$ (see, for example, the beginning of \S 2.1, \S 6.1 of \cite{LLLM}). In the notation of the \emph{loc. cits.},  $\bm{\alpha}_{(s, \mu), j} = (a_{1, j}, a_{2,j}, a_{3,j})$.  If \[s_{\mathrm{or}} := (s_1^{-1} s_2^{-1} \ldots s_{f-1}^{-1}, s_1^{-1} s_2^{-1} \ldots s_{f-2}^{-1},   \ldots, s_1^{-1},  1) \in W(\un{G}),\] then $s^*_{\mathrm{or}}(\bm{\alpha}_{(s, \mu)}) = \mu + \eta$, and conjugation by $(\un{0}, s^*_{\mathrm{or}})$ changes one presentation to the other.  The element $s_{\mathrm{or}}$ is called the \emph{orientation} of $\bm{\alpha}_{(s, \mu)}$ (Definition 2.6 and equation (2.2) in \cite{LLLM}).  
\end{rmk}

\begin{rmk} Comparing \eqref{pres} with equation at beginning of \S 2.1 \cite{LLLM}, the exponents differ by a minus sign.  This is because of a dual with appears in Definition \ref{defn:Kisin} which makes everything consistent.   See Remark \ref{tauvee} for more details.  
\end{rmk}


We continue to write $K$ (resp. $K'$) for the unramified extension of $\Q_p$ of degree $f$ (resp. $f'\defeq fr$).
If $r =1$, we say that $\tau$ is a \emph{principal series type}. Otherwise, we write $\tau'$ for the base change of $\tau$ to $K'/K$ (which is just $\tau$ considered as a principal series type for $G_{K'}$). We record the relevant data for $\tau'$.  Define $\bm{\alpha}'_{(s,\mu)} \in  X^*(T)^{\Hom(k', \F)} \cong X^*(T)^{f'}$ (using the fixed choice of embedding $\sigma_0'$) by 
\[
\bm{\alpha}'_{(s,\mu), j + kf} :=  s_{\tau}^{-k} (\bm{\alpha}_{(s, \mu), j}) \text{ for } 0 \leq j \leq f-1, 0 \leq k \leq r-1.
\]   If $\tau_{K'}(w', \mu')$ is the analogous construction of tame types over $K'$ for $(w', \mu') \in (W(\GL_n) \times X^*(T))^{\Hom(k', \F)}$, then $\tau' \cong \tau_{K'}(1, \bm{\alpha}'_{(s,\mu)})$ by direct comparison using \eqref{pres}.  The \emph{orientation} $s'_{\mathrm{or}} \in W(\GL_n)^{f'}$ of $\bm{\alpha}'_{(s,\mu)}$ in the sense of Definition 2.6 in \cite{LLLM} is given by 
\begin{equation} \label{primeorient}
s'_{\mathrm{or}, j + kf} := s_{\tau}^{k+1} s_{\mathrm{or}, j}\text{ for } 0 \leq j \leq f-1, 0 \leq k \leq r-1 
\end{equation} 
 (compare with Proposition 6.1 \cite{LLLM}).  

Note that $\bm{\alpha}'_{(s, \mu), j+ kf} $ is in the $W(\GL_n)$ orbit of $\mu_j + \eta_j$. If $(s, \mu)$ is a lowest alcove presentation with $\mu$ $m$-deep in alcove $\bun{C}_0$, then for all coroots $\alpha^{\vee}$ of $\GL_n$ and $0 \leq j' \leq f'-1$,
\begin{equation} \label{primegen}
m < |\langle \bm{\alpha}_{(s, \mu), j'}, \alpha^{\vee} \rangle| < p- m \text{ and } (s'_{\mathrm{or}, j'})^{-1}(\bm{\alpha}'_{(s, \mu), f' - 1 - j'}) = \mu_{f-1 - j} + \eta_{f-1-j} \text{ is dominant}
\end{equation}
where $j \equiv j'$ mod $f$. 





Define $L' := K'(\varpi_r) = K'((-p)^{\frac{1}{p^{rf}-1}})$, and let $\Delta' := \Gal(L'/K') \subset \Delta := \Gal(L'/K)$. Note that $\tau$ defines a  $\cO$-valued representation of $\Delta'$. For any complete local Noetherian $\cO$-algebra with residue field $\F'$ finite over $\F$, let $\fS_{L', R} := (W(k') \otimes_{\Zp} R)[\![u']\!]$. We endow $\fS_{L', R}$ with an action of $\Delta$ as follows: for any $\tau$ in $\Delta'$, $\tau(u') = \frac{\tau(\varpi_r)}{\varpi_r} u'$ and $\tau$ acts trivially on the coefficients; if $\sigma \in\Gal(L'/\Qp)$ is the lift of Frobenius on $W(k')$ which fixes $\varpi_r$, then $\sigma^f$ generates $\Gal(K'/K)$ acting in natural way on $W(k')$ and trivially on both $u'$ and $R$. Set $v = (u')^{p^{rf}-1}$, and note that 
\[
(\fS_{L', R})^{\Delta = 1} = (W(k) \otimes_{\Zp} R)[\![v]\!].
\]
As usual, $\varphi:\fS_{L', R} \ra \fS_{L', R}$ acts as $\sigma$ on $W(k')$, trivially on $R$, and sends $u'$ to $(u')^{p}$.  

For any positive integer $h$, let $Y^{[0, h], \tau}(R)$ be the category of Kisin modules over $L'$ with tame descent of type $\tau$ and height in $[0,h]$ as defined in  \S 3 of \cite{CL} if $\tau$ is principal series types. For other types, we refer to \S 6 of \cite{LLLM} for further background.   

\begin{defn} \label{defn:Kisin} An element $(\fM, \phi_{\fM}, \{ \widehat{g} \}) \in Y^{[0, h], \tau}(R)$ is a Kisin module $(\fM, \phi_{\fM})$ over $\fS_{L', R}$ (Definition 2.3 \cite{LLLM}) with height less than $h$ together with a semilinear action of $\Delta$ which commutes with $\phi_{\fM}$ such that for each $0 \leq j \leq f' - 1$ 
\[
\fM^{(j)} \mod u' \cong \tau^{\vee} \otimes_{\cO} R 
\]  
as $\Delta'$-representations. In particular, the semilinear action induces an isomorphism $(\sigma^f)^*(\fM) \cong \fM$ (see \cite[\S 6.1]{LLLM}) as elements of $Y^{[0,h], \tau'}(R)$.
\end{defn}   
\begin{rmk} As explained in \cite[\S 6.1]{LLLM}, the data of an extension of the action of $\Delta'$ to an action of $\Delta$ is equivalent to the choice of an isomorphism $(\sigma^f)^*(\fM) \cong \fM$ satisfying an appropriate cocycle condition. We will use both point of view interchangeably.
\end{rmk}
\begin{rmk} \label{tauvee} The appearance of $\tau^{\vee}$ in the definition is due to the fact that we are using the contravariant functors to Galois representations to be consistent with \cite{LLLM} as opposed to the covariants versions which appear in \cite{CL, EGH}.  In \cite{LLLM}, we didn't use the notation $\tau^{\vee}$. Instead, we included it in our description of descent data by having a minus sign in the equation before Definition 2.1 of \emph{loc. cit.}   The notion of Kisin module with tame descent data of type $\tau$ here is consistent with what appears in \emph{loc. cit.} 
\end{rmk}

Recall that we have fixed a lowest alcove presentation $(s, \mu)$ with $\mu \in \bun{C}_0$.  Definitions \ref{defn:shape} and \ref{defn:eigenbasis} as well as the matrix of partial Frobenius $A^{(j)}$ below depend on the choice of presentation (see Remark \ref{rmk:preschoice}). 

Recall the following definition:
\begin{defn} \label{defn:eigenbasis} For any complete local Noetherian $\cO$-algebra $R$, an eigenbasis $\beta$ for $\fM \in Y^{[0, h], \tau}(R)$ is a (ordered) basis $\beta^{(j')} = (f_1^{(j')}, f_2^{(j')}, \ldots, f_n^{(j')})$ of $\fM^{(j')}$ for each $0 \leq j' < f'$ such that $\Delta'$ acts on $f_i^{(j')}$ via the character $\chi_i^{-1}$ from \eqref{pres} and such that $(\sigma^f)^*(\beta) = \beta$ (see \cite[Definition 2.8]{LLLM} and \cite[Definition 3.1.6]{LLLM2} for details). 
\end{defn}

Note that since the order of $\Delta'$ is prime to $p$ and $\cO$ is assumed to be sufficiently large eigenbases exist for any $\fM \in Y^{[0, h], \tau}(R)$ when $R$ is a complete local Noetherian $\cO$-algebra as above.

Given $\fM \in Y^{[0, h], \tau}(R)$ together with an eigenbasis $\beta$, the matrix of the partial Frobenius $\phi^{(j')}_{\fM, s'_{\mathrm{or}, j'+1}(n)}$  with respect to $\beta$ is defined as in Definition 2.11 of \cite{LLLM}.  Namely, let $\phi^{(j')}_{\fM, s'_{\mathrm{or}, j'+1}(n)}:\phz^*(\fM)^{(j')}_{\chi_{s'_{\mathrm{or}, j'+1}(n)}} \ra  \fM^{(j'+1)}_{\chi_{s'_{\mathrm{or}, j'+1}(n)}}$ be the Frobenius map on the $\chi_{s'_{\mathrm{or}, j'+1}(n)}$ isotypic pieces of $\fM^{(j'+1)}$ and $\phz^*(\fM)^{(j')}$ respectively.  For any $0 \leq j' \leq f'-1$, set 
\begin{equation} \label{def:atot}
\bf{a}_{(s, \mu)}^{\prime \, (j')} = \sum_{i=0}^{f'-1} \bm{\alpha}'_{(s, \mu),- j' + i} p^i
\end{equation}
where $-j' + i$ is taken modulo $f'$.   

If $\beta^{(j' + 1)} = (f_1^{(j' + 1)}, \ldots, f_n^{(j' + 1)})$, then as in Lemma 2.9 in \cite{LLLM}, $$\left\{ (u')^{\bf{a}_{(s, \mu), i}^{\prime \, (j' +1)} - \bf{a}_{(s, \mu), s'_{\mathrm{or}, j'+1}(n) }^{\prime \, (j'+1)} } f_i^{(j' +1)} \right\}_{i=1}^n$$ is a basis of $\fM^{(j'+1)}_{\chi_{s'_{\mathrm{or}, j'+1}(n)}}$.  Similarly, if $\beta^{(j')} = (f_1^{(j')}, \ldots, f_n^{(j')})$, then $$\left\{ (u')^{\bf{a}_{(s, \mu), i}^{\prime \, (j' +1)} - \bf{a}_{(s, \mu), s'_{\mathrm{or}, j'+1}(n) }^{\prime \, (j'+1)}} \otimes f_i^{(j')} \right\}_{i=1}^n$$ is a basis for $\phz^*(\fM)^{(j')}_{\chi_{s'_{\mathrm{or}, j'+1}(n)}}$.  We order these bases such that the $u'$-multiple of $f_{s'_{\mathrm{or}, j'+1}(i)}$ is the $i$th basis vector. Note that the orientation $s'_{\mathrm{or}}$ is chosen such that for all $i < k$,
\begin{equation} \label{eq:domorient}
p^{f'} - 1 > \bf{a}_{(s, \mu), s'_{\mathrm{or}, j'+1}(i)}^{\prime \, (j'+1)} - \bf{a}_{(s, \mu), s'_{\mathrm{or}, j'+1}(k)}^{\prime \, (j'+1)}  > 0
\end{equation}
 so that all the exponents which appear are positive. The inequalities are strict because $\tau$ is regular (since $\tau$ is $1$-generic).   

The matrix of the $j'$-th partial Frobenius with respect to these bases ordered as above will be denoted by $A^{(j')} = \Mat_{\beta}\big(\phi^{(j')}_{\fM, s'_{\mathrm{or}, j'+1}(n)} \big)$. We stress that the notion of eigenbasis and the definition of $A^{(j')}$ depends on the chosen presentation of $\tau$, as well as our choice of ordering of the characters in $\tau$. By our requirement that $\beta$ is $\sigma^f$ invariant, $A^{(j')}$ only depends on $j'$ mod $f$. We also observe that the height condition implies $v^{h}(A^{(j')})^{-1}\in M_n(R[\![v]\!])$.  

For any $\cO$-algebra $R$, define
\begin{itemize}
\item $\cI(R) := \{ M \in \GL_n(R[\![v]\!]) \mid M \mod v \text{ is upper triangular} \}$;
\item $\cI_1(R) := \{ M \in \GL_n(R[\![v]\!]) \mid M \mod v \text{ is upper triangular unipotent} \}$;
\item For any $m \geq 1$, $\cD_m(R) := \{ M \in \GL_n(R[\![v]\!]) \mid M \mod v^m \text{ is diagonal} \}$. 
\end{itemize}

For any $M \in \Mat_n(R(\!(u')\!))$ and $g \in \GL_n(R(\!(u')\!) )$, define 
\begin{equation} \label{def:adj}
\Ad(g)(M) := g M g^{-1}. 
\end{equation}

We can now record the effect of changing the eigenbasis $\beta$ on the matrices $A^{(j)}$, which is the generalization of Propositions 2.15 and 2.16 in \cite{LLLM}:
\begin{prop} \label{prop:changeofbasis}
Let $R$ be a complete local Noetherian $\cO$-algebra.  Let $\fM \in Y^{[0,h], \tau}(R)$ together with two eigenbases $\beta^{(j)}:=  \Big( f^{(j)}_{1},  f^{(j)}_{2},\ldots ,f^{(j)}_{n} \Big) $ and $\beta'^{(j)}:= \Big ( f'^{(j)}_{1}, f'^{(j)}_{2},\ldots, f'^{(j)}_{ n} \Big)$ related by
$$
\Big (f'^{(j)}_{1}, f'^{(j)}_{2},\ldots , f'^{(j)}_{n} \Big) D^{(j)} = \Big (f^{(j)}_{1}, f^{(j)}_{2},\ldots, f^{(j)}_{n} \Big)
$$
with $D^{(j)} \in \GL_n(R[\![u']\!])$. 
Let us write $A^{(j)} \defeq \Mat_{\beta}\big(\phi^{(j)}_{\fM, s'_{\mathrm{or}, j+1}(n)}\big)$ and  $A'^{(j)} \defeq \Mat_{\beta'}\big(\phi^{(j)}_{\fM, s'_{\mathrm{or}, j+1}(n)}\big)$ as above.

Set $I^{(j)} \defeq \Ad \big((s'_{\mathrm{or}, j})^{-1} (u')^{-\mathbf{a}_{(s, \mu)}^{\prime \, (j)}} \big) (D^{(j)}) \in \Iw(R)$,which only depends on $j \mod f$. 
 
Then for all $0 \leq j \leq f'-1$, 
\[
A'^{(j)}= I^{(j+1)}A^{(j)} \big(s^*_j \big(I^{(j), \phz}\big) (s^*_j)^{-1}\big) 
\]
where $$I^{(j), \phz}\defeq  v^{\mu_j^*+\eta_j^*} (\phz(I^{(j)})^{-1}) v^{-\mu_j^*-\eta_j^*}.$$ 

Furthermore, if $\mu$ is $m$-deep in alcove $\bun{C}_0$, then $I^{(j), \phz} \in \cD_{m+1}(R)$.
\end{prop} 
\begin{proof} The formula for change of basis only depends on $\fM$ as Kisin module over $L'$ for the principal series type $\tau'$. The fact that $I^{(j)}$ only depends on $j \mod f$ follows from the fact that $\beta$ is fixed by $\sigma^f$.  The rest of the proof is the same as in Proposition 2.15 of \cite{LLLM} where we note that $s_{j}$ which appear in \emph{loc. cit.} are called $s'_{\mathrm{or}, j}$ here.  We use that $(s'_{\mathrm{or}, j+1})^{-1} s'_{\mathrm{or}, j} = s^*_j$ by \eqref{primeorient} and Remark \ref{rmk:compare}. 
Also, we use that for $0 \leq j \leq f-1$, $(s'_{\mathrm{or}, j})^{-1}(\bm{\alpha}'_{(s, \mu), f-1-j}) = \mu_{f-1-j} + \eta_{f-1-j}$ (cf. \eqref{primegen}). 
That $I^{(j)} \in \cI(R)$ follows from Equation \ref{eq:domorient} which follows from the fact that the characters appearing in $\tau$ are distinct (see Proposition 4.6 in \cite{CL}).
The fact that $I^{(j), \phz} \in \cD_{m+1}(R)$ is straightforward (compare with Proposition 2.16 in \cite{LLLM}).    

\end{proof}


\begin{rmk}\label{rmk:changeofbasistuple} In the situation of Proposition \ref{prop:changeofbasis}, we call the tuple $(I^{(j)})\in \cI(R)^{f'}$ the change of basis tuple from $\beta$ to $\beta'$. It satisfies $I^{(j)}=I^{(k)}$ if $j\equiv k$ mod $f$. Conversely, any tuple in $\cI(R)^{f'}$ with this property is the change of basis tuple from $\beta$ to another eigenbasis $\beta'$ (this uses our running assumption that $\tau$ is $1$-generic). In other words, given $\beta$, the data of an eigenbasis $\beta'$ is the same as the data of the tuple $(I^{(j)})$.
\end{rmk}
 Recall (from \cite[Definition 5.5]{CL}) the notion of shape:
\begin{defn}
\label{defn:shape} If $\tau$ is a principal series type, the \emph{shape} of a Kisin module $\fM \in Y^{[0,h], \tau}(\F')$ is the element $\tld{w} = (\widetilde{w}_0, \widetilde{w}_1, \ldots, \widetilde{w}_{f-1}) \in \bun{\widetilde{W}}^{\vee}=(\widetilde{W}^{\vee})^{\mathrm{Hom}(k,\F)}$ such that for any eigenbasis $\beta$ and any $0 \leq j \leq f-1$, the matrix $A^{(j)} = \Mat_{\beta}\big(\phi_{\fM, s'_{\mathrm{or}, j+1}(n)}^{(j)} \big) $ lies in $\Iw(\F') \widetilde{w}_j \Iw(\F')$. (Recall the fixed inclusion $\widetilde{W}^{\vee} \iarrow N_{\GL_n}(T)(\F(\!(v)\!))$ before Definition \ref{defn:Wphi}.) 

For a non-principal series type $\tau$, we define the shape via base change as in \cite[Definition 6.10]{LLLM}. By definition, an element $\fM \in Y^{[0,h], \tau}(\F')$ consists of an element $\fM' \in Y^{[0,h], \tau'}(\F')$ together with an isomorphism $(\sigma^f)^*(\fM') \cong \fM'$ satisfying the cocycle condition as in \cite[Definition 6.3]{LLLM}. By the principal series type case, we have the shape of $\fM'$, which is an element  $\widetilde{w}' = (\widetilde{w}'_0, \widetilde{w}'_1, \ldots, \widetilde{w}'_{f'-1}) \in (\widetilde{W}^{\vee})^{\mathrm{Hom}(k',\F)}$. By the isomorphism $(\sigma^f)^*(\fM') \cong \fM'$ and our requirement that eigenbases are compatible with it, the components of $\widetilde{w}'$ corresponding to two embeddings $k'\into \F$ are equal if they restrict to the same embedding $k\into \F$. In our numbering, this gives $\widetilde{w}'_j=\widetilde{w}'_{j+f}$. We then \emph{define} the shape $\widetilde{w}$ of $\fM$ as the element $\widetilde{w}= (\widetilde{w}'_0, \widetilde{w}'_1, \ldots, \widetilde{w}'_{f-1}) \in \bun{\widetilde{W}}^{\vee}=(\widetilde{W}^{\vee})^{\mathrm{Hom}(k,\F)}$.
\end{defn}

\begin{rmk} \label{rmk:preschoice} Proposition \ref{prop:changeofbasis} shows that the shape of a Kisin module is well-defined. The shape of a Kisin module depends mildly on the choice of lowest alcove presentation of $\tau$ (and the associated ordered characters $\chi_i$ in \eqref{pres}).   For a different choice of presentation, the shape changes by an outer automorphism of $\un{W}_a$ coming from the action of fundamental group (cf. \cite{LLLM}, Corollary 2.24).  
Everything we do depends on a choice of a lowest alcove presentation, and we will always fix it at the outset before talking about objects such as $A^{(j)} = \Mat_{\beta}\big(\phi_{\fM, s'_{\mathrm{or}, j+1}(n)}^{(j)} \big)$, etc.
\end{rmk}

\begin{rmk} The shape (Definition \ref{defn:shaperhobar}) is a kind of relative position between the two tame representations $\rhobar^{ss}$ and $\ovl{\tau}$. The shape is closely related to the geometry of the potentially crystalline deformation ring of type $(\eta, \tau)$ as studied in \cite{LLLM}.   
\end{rmk}

Recall the functor from \S 6.1 of \cite{LLLM}:
\[
T^*_{dd}:Y^{[0, h], \tau}(R) \ra \Rep_{R} (G_{K_{\infty}}).
\]
Let $\lambda \in X^*(\un{T})$ effective, i.e., $\lambda_j = (a_{i,j})$ with $a_{i,j} \geq 0$. 
We will need finer control of the shape of the Kisin module in the case when $\rhobar$ is semisimple.  For any $\lambda$ effective, we have a closed substack $Y^{\lambda, \tau} \subset Y^{[0,h], \tau}$ constructed in \cite[Proposition 5.2]{CL} (see also \cite[Section 3.1]{LLLM}). Then for any finite extension $\F'/\F$, $Y^{\lambda, \tau}(\F')$ is the full subgroupoid of $Y^{[0,h], \tau}(\F')$ (for any sufficiently large $h$) consisting of Kisin modules whose shapes lies in $\Adm^{\vee}(\lambda)$ by \cite[Proposition 5.4]{CL}.

\begin{defn} \label{Kissemisimple}  Let $\overline{\fM} \in Y^{[0,h], \tau}(\F')$ where $\F'/\F$ is a finite extension.  We say that $\overline{\fM}$ is \emph{semisimple} of shape $\tld{w} = (\widetilde{w}_j) \in \tld{\bun{W}}^{\vee}$ if there exists an eigenbasis $\beta$ of $\overline{\fM}$ such that 
$$
A^{(j)} =  \Mat_{\beta}\big(\phi_{\ovl{\fM}, s'_{\mathrm{or}, j+1}(n)}^{(j)} \big)\in T(\F'[\![v]\!]) \widetilde{w}_j
$$ 
for $0 \leq j \leq f'-1$.
 \end{defn} 
\begin{rmk}\label{monomial 0} Since the set of monomial matrices (i.e. matrices that have at most one non-zero entry in each row and column) in $\cI(\F')\tld{w}_j\cI(\F')$ is exactly $T(\F'[\![v]\!])\tld{w}_j$, the above condition is equivalent to $A^{(j)}$ being a monomial matrix. 
\end{rmk}
\begin{prop} \label{simpleform} If $\overline{\fM}$ is \emph{semisimple} of shape $\tld{w} = (\widetilde{w}_j) \in \bun{\tld{W}}^{\vee}$, then there exists an eigenbasis $\beta$ such that 
\begin{equation*}
A^{(j)} \in T(\F') \widetilde{w}_j \text{ for } 0 \leq j \leq f-1. 
\end{equation*}
  \end{prop}
\begin{proof} By definition, there exists an eigenbasis $\beta$ such that $A^{(j)} = D_j \tld{w}_j$ for $D_j \in T(\F'[\![v]\!])$.  Let $\ovl{D}_j = D_j \mod v$.   For $0 \leq j \leq f'-1$, set $I^{(j+1)} = \ovl{D}_{j} D_{j}^{-1} \in \cI(\F')$ with $j$ considered mod $f$.    Then $(I^{(j)}) \in \cI(\F')^{f'}$ defines an $f'$-tuple as in Remark \ref{rmk:changeofbasistuple}. 
By Proposition \ref{prop:changeofbasis}, there is an eigenbasis $\beta_1$ for $\ovl{\fM}$ such that the matrix of partial Frobenius with respect to $\beta_1$ is 
\[
A^{(j)}_1 = I^{(j+1)} A^{(j)} s_j^*(I^{(j), \phz})^{-1} (s_j^*)^{-1} = \ovl{D}_j \tld{w}_j  s_j^*(I^{(j), \phz})^{-1} (s_j^*)^{-1}.
\]
Since $I^{(j)} \equiv 1 \mod v$ and is an element of $T(\F'[\![v]\!])$, $s_j^*(I^{(j), \phz})^{-1} (s_j^*)^{-1} \equiv 1 \mod v^p$ and is an element of $T(\F'[\![v]\!])$.   We conclude that 
$
A^{(j)}_1 = D_{1, j} \tld{w}_j
$
where $D_{1,j} \equiv \ovl{D}_j \mod v^p$.  Repeating this process, we can inductively construct a sequence of eigenbases $\beta_m$ such that the matrix of partial Frobenius with respect to $\beta_m$ has the form $D_{m, j} \tld{w}_j$ where $D_{m,j} \equiv  \ovl{D}_j\mod v^{p^m}$.  The sequence $\beta_m$ converges to an eigenbasis with the desired property.   
\end{proof}
	
\begin{cor} \label{cor:phiandKisin}  Let $(s, \mu)$ be a lowest alcove presentation for $\tau$.  If $\overline{\fM} \in Y^{[0,h], \tau}(\F')$ is semisimple of shape $\tld{w}  \in \bun{\tld{W}}^{\vee}$, then $T_{dd}^*(\ovl{\fM})$ is semisimple and (after extending coefficients)
\[
T_{dd}^*(\overline{\fM})|_{I_{K}} \cong \bV^*(\cM(\widetilde{w} s^* t_{\mu^* + \eta^*}))|_{I_{K}} \cong \ovl{\tau}(w, \nu + \eta)
\] 
where $\widetilde{w} s^* t_{\mu^*} = w^* t_{\nu^*}$. 
\end{cor}
\begin{proof}

The second isomorphism is from Proposition \ref{Galoistotype}. 
This first isomorphism  follows from a direct computation of the \'etale $\varphi$-module $\overline{\cM} = (\overline{\fM}[1/u'])^{\Delta = 1}$ as in \cite[Proposition 3.2.1]{LLLM2}. We briefly go through the main points.  

Let $\beta = (\beta^{(j)})$ be an eigenbasis for $\ovl{\fM}$ as in Proposition \ref{simpleform}.   We start by considering a basis $\tld{\beta}'$ for $\ovl{\cM}' :=(\overline{\fM}[1/u'])^{\Delta' = 1}$ as follows:  for $0 \leq j' \leq f'-1$, if $\beta^{(j')} =(f_1^{(j')}, \ldots, f_n^{(j')})$, define $\tld{\beta}^{\prime, (j')} =  ( (u')^{\bf{a}^{\prime \, (j')}_{(s, \mu). 1}}f_1^{(j')}, \ldots,  (u')^{\bf{a}^{\prime \, (j')}_{(s, \mu), n}} f_n^{(j')})$ which is a basis for $\ovl{\cM}^{\prime \, (j')}$.  The matrix for $\phi_{\ovl{\cM}'}^{(j')}:\ovl{\cM}^{\prime \, (j')} \ra \ovl{\cM}^{\prime \, (j'+1)}$ with respect to $\tld{\beta}'$ is given by 
\[
s'_{\mathrm{or}, j'+1} A^{(j')} (s'_{\mathrm{or}, j'+1})^{-1} (u')^{p \bf{a}^{\prime \, (j')}_{(s,\mu)} - \bf{a}^{\prime \, (j'+1)}_{(s,\mu)}}.
\]
Since $p \bf{a}^{\prime \, (j')}_{(s,\mu)} - \bf{a}^{\prime \, (j'+1)}_{(s, \mu)} = (p^{f'} -1) \bm{\alpha}'_{(s, \mu), f'-1-j'}$, this is same as 
\[
s'_{\mathrm{or}, j'+1} A^{(j')} (s'_{\mathrm{or}, j'+1})^{-1}  v^{\bm{\alpha}'_{(s, \mu), f'-1-j'}}.  
\]
Define $\tld{\beta}$ by $\tld{\beta}^{(j')} = \tld{\beta}^{\prime \, (j')} s'_{\mathrm{or}, j'}$. Let $j' = j + if$ for $0 \leq j \leq f-1$.   Then the matrix for $\phi_{\ovl{\cM}'}^{(j')}$ with respect to $\tld{\beta}$ is given by 
\[
 A^{(j')} (s'_{\mathrm{or}, j'+1})^{-1}  s'_{\mathrm{or},j'} v^{(s'_{\mathrm{or},j'})^{-1} (\bm{\alpha}'_{(s, \mu), f'-1-j'})} =  A^{(j')} s^*_j v^{\mu^*_j + \eta^*_j}
\] 
using \eqref{primeorient} and \eqref{primegen}.   Since $(\sigma^f)^*(\tld{\beta}^{(j')}) = \tld{\beta}^{(j' -f)}$, this descends to a basis of $\ovl{\cM} := (\overline{\fM}[1/u'])^{\Delta = 1} = (\ovl{\cM}')^{\sigma^f = 1}$ such that Frobenius $\phi^{(j)}_{\ovl{\cM}}:\ovl{\cM}^{(j)} \ra \ovl{\cM}^{(j+1)}$ is given by 
\[
A^{(j)} s^*_j v^{\mu^*_j + \eta^*_j} = D_j \widetilde{w}_j s^*_j v^{\mu^*_j + \eta^*_j}
\] 
for $D = (D_j) \in T(\F')^f$ using Proposition \ref{simpleform}.   Thus, $\ovl{\cM} \cong \cM(\widetilde{w} s^* t_{\mu^* + \eta^*}, D)$ (Definition \ref{defn:Wphi}). 
\end{proof}

Before proving the main theorems of the section, we show that in generic situations, the Kisin module which gives rise to $\rhobar$ is necessarily unique (if one exists), i.e., the Kisin variety is trivial. 
	
\begin{prop} \label{Kisinvariety}  Assume that $\tau$ is generic.
\begin{enumerate}
\item
Let $\overline{\fM}, \overline{\fM}' \in  Y^{\eta, \tau}(\F')$ for some finite extension $\F'/\F$.  If $T^*_{dd}(\overline{\fM}) \cong T^*_{dd}(\overline{\fM'})$, then $\overline{\fM} \cong \overline{\fM'}$.  
\item Let $\ovl{\fM}  \in  Y^{\eta, \tau}(\F')$ and $\rhobar := T^*_{dd}(\ovl{\fM})$. Define a groupoid
$$
\mathfrak{t}_{\overline{\fM}} = \Big\{ (\fM, \delta_0) \mid \fM \in Y^{\eta, \tau}(\F'[\eps]/\eps^2),\ \delta_0:\fM/\eps\fM \stackrel{\sim}{\longrightarrow}\overline{\fM} \Big\}.
$$

Then the functor $T_{dd}^*$ induces a fully faithful functor
$$
T_{\mathrm{tan}}^*: \mathfrak{t}_{\overline{\fM}} \ra \Rep_{\F'[\eps]/\eps^2}(G_{K_{\infty}})_{\rhobar} = \Big\{ (\rho, \gamma_0) \mid \rho \in \Rep_{\F'[\eps]/\eps^2}(G_{K_{\infty}}),\ \gamma_0:\rho \! \mod \eps \stackrel{\sim}{\longrightarrow}\rhobar \Big\}.
$$   
\end{enumerate}
\end{prop}
\begin{proof} We fix a lowest alcove presentation $(s,\mu)$ of $\tau$ such that $\mu$ is $n$-deep in alcove $\bun{C}_0$ to perform all calculations.

For part (1), since the $G_{K_{\infty}}$-representations are isomorphic, we have $\overline{\fM}[1/u'] \cong \overline{\fM}'[1/u']$ as \'etale $\varphi$-modules with descent datum. We pick two eigenbases $\beta$, $\beta'$ of $\overline{\fM}$, $\overline{\fM}'$, and let $(D^{(j)})\in (\GL_n(\F'((u'))))^{f'}$ be the $f'$-tuple which expresses the basis $\beta'$ in terms of $\beta$ as in Proposition \ref{prop:changeofbasis}. Note that $D^{(j)}=D^{(j+f)}$ since our eigenbases are compatible with the action of $\sigma^f$. The same computation in the proof of \cite[Theorem 3.2]{LLLM} with $2$ replaced everywhere by $n-1$ (cf. Remark \ref{rmk:comparison} for the comparison of the genericity in the present paper with that of \cite{LLLM}) now shows that $D^{(j)}\in \GL_n(\F' [\![u']\!])$ for all $j$.

Part (2) is similar to part (1).  The argument in the proof of \cite[Proposition 3.4]{LLLM} adapts to our situation.

\end{proof}

\begin{defn} \label{defn:shaperhobar} Assume that $\tau$ is generic with a chosen lowest alcove presentation $(s,\mu)$. If there exists $\ovl{\fM} \in Y^{\eta, \tau}(\F)
$ such that $T^*_{dd}(\overline{\fM}) \cong \rhobar|_{G_{K_{\infty}}}$, then define $\tld{w}(\rhobar, \tau) \in \Adm^{\vee}(\eta)$ to be the shape of $\overline{\fM}$.  This is well-defined by Proposition \ref{Kisinvariety}. 
\end{defn}

A key input for weight elimination is the following: 

\begin{thm} \label{thm:potcris} Assume either $(1)$ $\tau$ is a regular principal series type or $(2)$ $\lambda = \eta$ and $\tau$ is generic.  If $\rhobar$ has a potentially crystalline lift with type $(\lambda, \tau)$ where $\lambda$ is effective, there is a Kisin module $\overline{\fM} \in  Y^{\lambda, \tau}(\F) $ such that $T^*_{dd}(\overline{\fM}) \cong \rhobar|_{G_{K_{\infty}}}$. 
\end{thm}
\begin{proof}
If $\tau$ is a principal series type, this is direct consequence of \cite[Proposition 5.4 and Corollary 5.18]{CL} (also see Remark 5.6 in \emph{loc. cit.} which compares the stratification of the moduli of Kisin module with Definition \ref{defn:shape}).  Proposition 5.4 in \emph{loc. cit.} crucially uses results about local models from \cite{PZ}.

For case (2), by the first case, there exists a Kisin module $\ovl{\fM}' \in  Y^{\eta, \tau'}(\F)$ with $T^*_{dd, K'}(\overline{\fM}') = \rhobar|_{G_{K'_{\infty}}}$ where $T^*_{dd, K'}$ is the analogous functor over $K'$.  Since $\rhobar$ extends to $G_{K_{\infty}}$, we have an isomorphism $\iota:(\sigma^f)^*(\ovl{\fM}'[1/u']) \cong \ovl{\fM}'[1/u']$ (satisfying an appropriate cocycle condition, cf. \cite[\S 6.1]{LLLM} ).  Since $\tau$ is generic, by the proof of Proposition \ref{Kisinvariety}, $\iota((\sigma^f)^*(\ovl{\fM}')) = \ovl{\fM}'$ inside $\ovl{\fM}'[1/u']$. Thus, $\ovl{\fM}'$ defines an element of $Y^{\eta, \tau}(\F)$ and hence has shape in $\Adm^{\vee}(\eta)$.     
\end{proof}

\begin{rmk} Our definition of $\Adm^{\vee}(\lambda)$ (Definition \ref{defn:adm}) is in terms of the Bruhat order on the (extended) affine Weyl group for $\GL_n$ using our choice of antidominant base alcove for the standard apartment.  The Bruhat order which appears in \cite[\S 5]{CL} comes from the closure relation on the affine flag variety with respect to the standard Iwahori subgroup, the subgroup $\cI(\F)$ of matrices which are upper triangular mod $v$. The subgroup $\cI(\F)$ is the stabilizer of the antidominant base alcove, hence the Bruhat order (and hence the admissible set as well) in \emph{loc. cit.} is the same as ours.  
\end{rmk}


The following Proposition gives us control on $Y^{\eta, \tau}(\F')$:
\begin{prop} \label{gauge basis} Assume $\tau$ is generic. Let $\overline{\fM}\in   Y^{[0,n-1], \tau}(\F')$, $\beta$ an eigenbasis of $\overline{\fM}$ and $A^{(j)} = \Mat_{\beta}\big(\phi_{\overline{\fM}, s'_{\mathrm{or}, j+1}(n)}^{(j)} \big)$. Then the assignment $\beta' \mapsto ( \Mat_{\beta'}\big(\phi_{\overline{\fM},s'_{\mathrm{or}, j+1}(n)}^{(j)} \big))_{0\leq j<f'}$ defines a bijection between the set of eigenbases $\beta'$ such that $\beta'=\beta$ mod $u'$ and the set tuples of the form $(X_jA^{(j)})_{0\leq j<f'}$ such that:
\begin{itemize}
\item $X_j\in \cI_1(\F')$ for all $j$.
\item $X_j=X_k$ if $j \equiv k$ mod $f$.
\end{itemize}
\end{prop}

\begin{proof}
Throughout this proof, we adopt the same notation as in Proposition \ref{prop:changeofbasis}: 
We let $(D^{(j)})$ be the tuple of matrices expressing $\beta$ in terms of $\beta'$, from which we get the matrices $I^{(j)}$, $I^{(j), \phz}$. Observe that the condition $\beta' \equiv \beta$ mod $u'$ is equivalent to $D^{(j)} \equiv 1$ mod $u'$ and also equivalent to $I^{(j)}\in\cI_1(\F')$. If these equivalent conditions hold then $I^{(j), \phz} \equiv 1$ mod $v^{n+1}$, since $\tau$ is assumed to be generic. We also observe that if $j \equiv k$ mod $f$, then the $j$-th and $k$-th component of any tuple of matrices that we work with are equal.

Since $\overline{\fM}\in   Y^{[0,n-1], \tau}(\F')$ we have $v^{n-1}(A^{(j)})^{-1}\in M_n(\F'[\![v]\!])$. This gives us a bound on the denominators of $(A^{(j)})^{-1}$.

First we check that our assignment is actually defined, in the sense that the collection $A'^{(j)}= \Mat_{\beta'}\big(\phi_{\overline{\fM}, s'_{\mathrm{or}, j+1}(n)}^{(j)} \big)$ associated to $\beta'$ is really of the  prescribed form.

By Proposition \ref{prop:changeofbasis} we have
\[
X_j=A'^{(j)}(A^{(j)})^{-1}= I^{(j+1)}A^{(j)} \big(s^*_j\big(I^{(j), \phz}\big) (s^*_j)^{-1}\big)(A^{(j)})^{-1}
\]
It is clear that $X_j=X_k$ if $j \equiv k$ mod $f$.
As observed above, $I^{(j), \phz} \equiv 1$ mod $v^{n+1}$ hence we can write $ s^*_j\big(I^{(j), \phz}\big) (s^*_j)^{-1}=1+v^{n+1}Y_j$ with $Y_j\in M_n(\F'[\![v]\!])$.  Since  $v^{n-1}(A^{(j)})^{-1}\in M_n(\F'[\![v]\!])$, we get
\[X_j=A'^{(j)}(A^{(j)})^{-1}= I^{(j+1)}(1+v^{n+1}A^{(j)} Y_j(A^{(j)})^{-1})\in \cI_1(\F')\]
as desired.

Next, we show that our assignment is surjective. Thus we are given $(X_j)$ with $X_j\in \cI_1(\F')$, and we need to solve the system of equations 
\[X_j= I^{(j+1)}A^{(j)}  \big(s^*_j\big(I^{(j), \phz}\big) (s^*_j)^{-1}\big)(A^{(j)})^{-1}\]
with $I^{(j)}\in\cI_1(\F')$. To do this, we carry out the same limiting procedure as in the proof of Lemma 2.20 of \cite{LLLM}, using $I^{(j), \phz} \equiv 1$ mod $v^{n+1}$ and $v^{n-1}(A^{(j)})^{-1}\in M_n(\F'[\![v]\!])$ to establish convergence to a solution.

Finally, we show that our assignment is injective. This means that we have to show that if there are two collections $(I^{(j)}), (I'^{(j)})\in \cI_1(\F')^{f'}$ corresponding to eigenbases $\beta', \beta''$ such that
\[ I^{(j+1)}A^{(j)} \big(s^*_j\big(I^{(j), \phz}\big) (s^*_j)^{-1}\big)(A^{(j)})^{-1}= I'^{(j+1)}A^{(j)} \big(s^*_j\big(I'^{(j), \phz}\big) (s^*_j)^{-1}\big)(A^{(j)})^{-1}\]
then $I^{(j)}=I'^{(j)}$ for all $j$. By replacing $\beta$ with $\beta''$ and $A^{(j)}$ with $A''^{(j)}$, we reduce to the case when $I'^{(j)}=1$. Thus it suffices to show that if 
\[ I^{(j+1)}A^{(j)} \big(s^*_j\big(I^{(j), \phz}\big) (s^*_j)^{-1}\big)(A^{(j)})^{-1}=1\]
then $I^{(j)}=1$ for all $j$.

Indeed, by the observations at the beginning of the proof, $A^{(j)} \big(s^*_j\big(I^{(j), \phz}\big) (s^*_j)^{-1}\big)(A^{(j)})^{-1} \equiv 1$ mod $v^2$ for all $j$, thus $I^{(j)} \equiv 1$ mod $v^2$ for all $j$. Suppose we already have $I^{(j)} \equiv 1$ mod $v^\delta$ for some $\delta\geq 2$ and for all $j$, then $I^{(j),\phz} \equiv 1$ mod $v^{p\delta-n-1}$, and hence $A^{(j)}  \big(s^*_j\big(I^{(j), \phz}\big) (s^*_j)^{-1}\big)(A^{(j)})^{-1} \equiv 1$ mod $v^{p\delta-2n}$ for all $j$, hence also $I^{(j)} \equiv 1$ mod $v^{p\delta-2n}$ for all $j$. Since $p\delta-2n>\delta$ (the existence of a generic $\tau$ implies $p>2n+1$), this shows that $I^{(j)} \equiv 1$ mod arbitrary high powers of $v$ for all $j$. This shows $I^{(j)}=1$ for all $j$.

\end{proof}

We now discuss the notion of gauge basis, which provide certain normal forms for Kisin modules. For a root $\alpha=\epsilon_i-\epsilon_j$ of $\GL_n$, recall that the $(i,j)$-th entry of an $n\times n$ matrix $A$ is also called the $\alpha$-th entry, and denoted by $A_{\alpha}$.  For any statement $\star$, define $\delta_{\star}$ to be 1 if $\star$ is true and 0 if $\star$ is false. 

For any $\tld{w}=wt_{\nu} \in \tld{W}^{\vee}$ and any ring $R$, we define the subset $U_{\widetilde{w}}(R)\subset \GL_n(R(\!(v)\!))$ to be the set of $X\in \GL_n(R(\!(v)\!))$ satisfying the following conditions:
\begin{itemize}
\item The diagonal entries of $X$ are in $R^{\times}$.
\item For any root $\alpha$ of $\GL_n$, the $\alpha$-th entry $X_\alpha$ is of the form $\sum_i a_iv^i\in R(\!(v)\!)$, where $a_i=0$ unless $\delta_{\alpha<0}\leq i <-\langle \nu, \alpha^{\vee} \rangle +\delta_{w(\alpha)<0}$. In particular, $X_{\alpha}=0$ unless $\delta_{\alpha<0} <-\langle \nu, \alpha^{\vee} \rangle +\delta_{w(\alpha)<0}$. 
\end{itemize}
By a standard computation with affine root groups, the natural map 
\[
\widetilde{w} U_{\tld{w}}(\F') \ra \cI_1(\F') \backslash  \cI(\F') \tld{w} \cI(\F')
\]
is a bijection, for any extension $\F'$ of $\F$.

The following definition generalizes  Definition 2.22 of \cite{LLLM}:
\begin{defn} \label{gauge basis defn} Let $\overline{\fM}\in   Y^{\eta, \tau}(\F')$ with shape $(\widetilde{w}_j)$. A \emph{gauge basis} $\beta$ for $\overline{\fM}$ is an eigenbasis for $\overline{\fM}$ such that $A^{(j)} = \Mat_{\beta}\big(\phi_{\fM, s'_{\mathrm{or}, j+1}(n)}^{(j)} \big)$ belongs to $\tld{w}_j U_{\tld{w}_j}(\F')$ for all $0\leq j <f$.
\end{defn}

\begin{exam} When $n =3$, the list of $\eta$-admissible elements $\Adm^{\vee}(\eta)$ is given in Table 1 of \cite{LLLM}.   For each $\tld{w} \in \Adm^{\vee}(\eta)$ up to outer automorphism, Table 4 of \emph{loc. cit.} lists the set $\tld{w}_j U_{\tld{w}_j}(\F)$ which are the possible matrices of partial Frobenii for pairs $(\overline{\fM}, \overline{\beta})$ with shape $\tld{w}$ where $\overline{\beta}$ is a gauge basis. 
\end{exam}

For $\tau$ generic, Proposition \ref{gauge basis} shows that gauge bases exist and are unique up to scaling by the subgroup of $T(\F')^{f'}$ consisting of tuples whose $j$-th and $k$-th entries are the same for $j \equiv k$ mod $f$.
\begin{rmk}\label{monomial} Assume $\tau$ is generic. If $\overline{\fM}\in Y^{\eta,\tau}(\F')$ is semisimple of shape $\tld{w}=(\tld{w}_j)$, Proposition \ref{simpleform} shows that there is an eigenbasis $\beta$ with the property that $A^{(j)}= \Mat_{\beta}\big(\phi_{\fM, s'_{\mathrm{or}, j+1}(n)}^{(j)} \big)$ belongs to $T(\F')\tld{w}_j=\tld{w}_jT(\F')\subset \tld{w}_jU_{\tld{w}_j}(\F')$. Such an eigenbasis is therefore a gauge basis, and we deduce that the matrices of partial Frobenii with respect to any gauge basis have this form. In particular they are all monomial matrices. Conversely, if there is a gauge basis for which the matrices of partial Frobenii is monomial, then $\overline{\fM}\in Y^{\eta,\tau}(\F')$ is semisimple, by Remark \ref{monomial 0}.

\end{rmk}

\begin{thm} \label{fixedpoints} Let $\rhobar:G_K \ra \GL_n(\F)$ be a semisimple representation and let $\tau$ be a 1-generic type. Assume that either $(1)$ $\rhobar$ is a direct sum of characters, or $(2)$ $\lambda = \eta$ and $\tau$ is generic. Then if there exists a Kisin module $\overline{\fM} \in  Y^{\lambda, \tau}(\F)$ such that $T^*_{dd}(\overline{\fM}) \cong \rhobar|_{G_{K_{\infty}}}$, then there exists a finite extension $\F'/\F$ and a semisimple Kisin module $\overline{\fM}' \in Y^{\lambda, \tau}(\F')$  with shape in $\Adm^{\vee}(\lambda)$ (cf. Definition $\ref{Kissemisimple}$) such that (after extending scalars) $T^*_{dd}(\overline{\fM}') \cong \rhobar|_{G_{K_{\infty}}}$. Furthermore, we can take $\F'=\F$ in case (2).
\end{thm}
\begin{proof}
Let $\cM_{dd} = \overline{\fM}[1/u']$. It is an \'etale $\varphi$-module over $L'$ with descent datum to $K$. 

We first treat the case where $\rhobar$ is a direct sum of characters. 
Since $\bV^*_{dd}$ is an equivalence of categories and $\rhobar$ is a direct sum of characters, $\cM_{dd} = \bigoplus_{i=1}^n \cM_i$ where each $\cM_i$ has rank $1$, is stable under $\phi_{\cM_{dd}}$ and the descent datum.  

Let $Y^{\lambda, \tau}_{\cM_{dd}}$ be the Kisin variety parametrizing lattices in $\cM_{dd}$ which lie in $Y^{\lambda, \tau}$, i.e., which have shape in $\Adm^{\vee}(\lambda)$. When $\tau$ is principal series type, $Y^{\lambda, \tau}_{\cM_{dd}}$ is defined as in  Definition 3.1 in \cite{LLLM} and is shown there to be a projective scheme over $\F$.   In general, we define $Y^{\lambda, \tau}_{\cM_{dd}}$ to be the closed subscheme of fixed points on $Y^{\lambda, \tau'}_{\cM_{dd}}$ for the natural action of $\sigma^f$.    Note that by assumption $Y^{\lambda, \tau}_{\cM_{dd}}$ is non-empty.

The torus $T = \Gm^n$ acts on $\cM_{dd}$ by scaling individually in each factor of the above decomposition.  As a consequence, we get an algebraic action of $T$ on the projective variety $Y^{\lambda, \tau}_{\cM_{dd}}$. 
Any such action has a fixed point (possibly after passing to a finite extension $\F'$ of $\F$).  Let $\overline{\fM}' \subset Y^{\lambda, \tau}_{\cM_{dd}}(\F')$ be a $T$-fixed point.   Let $\psi_i:T \ra \Gm$ denote projection onto the $i$-th coordinate, and set $\overline{\fM}'_i := (\overline{\fM}')^{\psi_i}$. Then
\begin{equation} \label{a1}
\overline{\fM}' = \bigoplus_{i=1}^n \overline{\fM}'_i
\end{equation}
Since the $T$-action commutes with $\phi_{\cM_{dd}}$ and $\Delta$, each $\overline{\fM}'_i$ is stable under both, hence $\overline{\fM}'_i$ is in fact a rank one Kisin module with descent datum.  Any choice of eigenbasis which respects this decomposition shows that $\overline{\fM}'$ is semisimple. Because $\overline{\fM}'$ is in $Y^{\lambda,\tau}(\F')$, it is semisimple with an admissible shape $\tld{w} \in \Adm^{\vee}(\lambda)$.

Now suppose that $\tau$ is generic, but $\rhobar$ is not necessarily a direct sum of characters. In this case, $\overline{\fM}$ is the unique element of $Y^{\eta,\tau}(\F)$ such that $T^*_{dd}(\overline{\fM}) \cong \rhobar|_{G_{K_{\infty}}}$ by Proposition \ref{Kisinvariety}. We pick an unramified extension $\breve{K}/K'$ such that $\rhobar|_{G_{\breve{K}}}$ is a direct sum of characters.  Let $\breve{\tau}$ be the base change of $\tau$ to $\breve{K}$, and let $\breve{\ovl{\fM}}=\ovl{\fM}\otimes _{W(k')}W(\breve{k})$ be the base change of the Kisin module $\ovl{\fM}$ to $\breve{K}L'$. Since $\breve{\tau}$ is still generic by Lemma \ref{lemma:bc}, $\breve{\ovl{\fM}}$ is the unique lattice in $\breve{\ovl{\fM}}[1/u']$ which belongs to $Y^{\eta,\breve{\tau}}(\F)$. By the above argument, the set of such Kisin lattices must have a semisimple element (possibly after extending $\F$), thus $\breve{\ovl{\fM}}$ is semisimple. Fix a gauge basis $\beta$ of $\overline{\fM}$, and let $\breve{\beta}$ be the induced eigenbasis of $\breve{\ovl{\fM}}$. It is easy to check that it is a gauge basis of $\breve{\ovl{\fM}}$. By Remark \ref{monomial}, the matrices of partial Frobenii of $\breve{\ovl{\fM}}$ with respect to $\breve{\beta}$ is monomial, hence the same is true for the matrices of partial Frobenii of $\ovl{\fM}$ with respect to $\beta$. This shows that $\ovl{\fM}\in Y^{\eta,\tau}(\F)$ is semisimple.

\end{proof}
\begin{rmk} \label{orderwarning}
In the proof of Theorem \ref{fixedpoints}, while $\Delta'$ acts on $(\overline{\fM}'_1)^{(j)} \mod u'$ through one of the characters of $\tau$, it need not be the same character for each $j$.  This is why even though $\phi_{\overline{\fM}}$ is ``diagonal,'' the individual $A^{(j)}$ need not be diagonal, but only monomial.
\end{rmk}

\begin{proof}[Proof of Theorem \ref{thm:admcrit}]  If $\rhobar$ is the reduction of a potentially crystalline representation of type $(\lambda, \tau)$, by a standard argument (cf., for example, \cite[Lemma 5]{Enns}), $\rhobar^{\mathrm{ss}}$ is also the reduction (after possible extending scalars) of a potentially crystalline representation of type $(\lambda, \tau)$.  So, without loss of generality, we assume $\rhobar$ is semisimple. After twisting, we may assume $\lambda$ is effective. Then, by Theorem \ref{thm:potcris} combined with Theorem \ref{fixedpoints}, after possibly replacing $\F$ by a finite extension, there exists a semisimple Kisin module $\overline{\fM} \in  Y^{\lambda, \tau}(\F)$ of shape $\tld{w} \in \Adm^{\vee}(\lambda)$ such that $\rhobar|_{G_{K_{\infty}}} \cong T^*_{dd}(\overline{\fM})$.   

If we write
\[
\widetilde{w} s^* t_{\mu^* + \eta^*} = w^* t_{\nu^* + \eta^*}
\] 
for $(w, \nu) \in W(\un{G}) \times X^*(\un{T})$. 
Then, by Corollary \ref{cor:phiandKisin}, 
\[
\rhobar^{\mathrm{ss}}|_{I_K} \cong \ovl{\tau}(w, \nu + \eta). 
\]
Furthermore, by Lemma \ref{lem:adjadm},   
\[
\tld{w}^* = s^{-1}  t_{\nu - \mu} w  \in \Adm(\lambda).
\]
\end{proof}

\subsection{Genericity conditions} \label{sec:gen}

\begin{lemma} \label{lemma:bc} Let $K'/K$ be an unramified extension.  If $\tau$ is an $m$-generic tame inertial type, then $\tau|_{I_{K'}}$ is also $m$-generic.   Similarly, if $\rhobar:G_K \ra \GL_n(\F)$ is an $m$-generic representation, then $\rhobar|_{G_{K'}}$ is $m$-generic. 

\end{lemma}
\begin{proof}
 Let $r$ be the degree of $K'/K$.  Let $\tau_{K'}$ denote the analogous construction as in Definition \ref{defn:tau} with $K$ replaced by $K'$ (using the choice of compatible system of $p^{r'f} -1$-system of roots of $-p$ from \S \ref{sec:not} for all $r'$ and an embedding $\sigma'_0:K' \ra E$ extending $\sigma_0$).   Then, for $(w', \mu') \in W(\un{G})^r \times X^*(\un{T})^r$,  there is a tame inertial type $\tau_{K'}(w', \mu'):I_{K'} \ra \GL_n(\cO)$.  For $(w, \mu) \in  W(\un{G})\times X^*(\un{T})$, let $\iota(w, \mu) = ((w,w, \ldots, w), (\mu, \mu, \ldots, \mu)) \in  W(\un{G})^r \times X^*(\un{T})^r$.  Then,
\begin{equation} \label{c1}
\tau(w, \mu)|_{I_{K'}} \cong \tau_{K'}(\iota(w, \mu)).
\end{equation}
This can be checked by direct computation.  Alternatively, one can appeal to Proposition \ref{Galoistotype} which says in particular that $\bV^*(\cM(w^* t_{\mu^*}, D))|_{I_K} \cong \ovl{\tau}(w, \mu)$ for any $D \in \un{T}(\F)$.  If $k'$ is the residue field of $k$, then restricting to $G_{K'}$ corresponds to tensoring with $k'$ on \'etale $\varphi$-modules side. We see immediately that 
\[
\cM(w^* t_{\mu^*}, D) \otimes_{k} k' = \cM((w^* t_{\mu^*}, w^* t_{\mu^*}, \ldots, w^* t_{\mu^*}), D') \text{ for } D' = (D, \ldots, D) \in \un{T}(\F)^r
\]
and so $\ovl{\tau}(w, \mu)|_{I_K} \cong \ovl{\tau}_{K'}(\iota(w, \mu))$.  Applying Teichm\"uller lifts yields \eqref{c1}.  

Let $\un{G}' = \Res_{k'/k} \un{G}$.  Let $\bun{C}_0'$ the analogous lowest alcove  with $k$ replaced by $k'$ viewed as a subset of $X^*(\un{T})^r \otimes \mathbb{R}$.   
Let $\tau \cong \tau(s, \lambda + \eta)$ with $\lambda$ $m$-deep in $\bun{C}_0$.  Then clearly $\iota(s, \lambda)$ is a lowest alcove presentation with $\iota(\lambda)$ $m$-deep in $\bun{C}_0'$.  The argument for $\rhobar$ is the same. 

\end{proof}

\begin{prop} \label{telim2}  Let $\tau$ be a 1-generic tame inertial type.  Assume $\rhobar$ is semisimple and $m$-generic where $m > n$. Then if $\rhobar$ is the reduction of a potentially crystalline representation of type $(\eta, \tau)$ then $\tau$ is  $m -n$-generic.  In particular, if $\rhobar$ is $2n$-generic, then $\tau$ is generic. 
\end{prop}
\begin{proof}
Let $(s, \mu)$ be a lowest alcove presentation for $\tau$. First, assume that that $\rhobar$ is a direct sum of characters and $\tau$ is a principal series.  By Theorem \ref{thm:admcrit}, we have
\[
\rhobar|_{I_K} \cong \overline{\tau}(w, \mu' + \eta)
\]
with $s^{-1} t_{\mu' - \mu} w \in \Adm(\eta)$. Define $\nu=\mu'-\mu$. By Lemma \ref{bound}, for any $\alpha^\vee\in \un{R}^{\vee}$ we have
\begin{equation}\label{ineq 1}
 |\langle \nu, \alpha^{\vee} \rangle|  \leq n-1.
\end{equation}
   Since $\rhobar$ is $m$-generic, there exists $(\nu', \sigma) \in X^*(\un{T}) \times W(\un{G})$ such that 
\begin{equation*}
\sigma(\mu' + \eta) + p \nu' - \sigma w \pi(\sigma)^{-1} \pi(\nu') - \eta
\end{equation*}
is $m$-deep in alcove $\bun{C}_0$.  
Hence, for any $\alpha^{\vee}\in \un{R}^{\vee}$,
\begin{equation}
\label{ineq 2}
m < |\langle \sigma(\mu'+\eta)+p\nu'-\sigma w\pi(\sigma^{-1})\pi(\nu'),\alpha^{\vee} \rangle| < p-m.
\end{equation}
Since $\mu$ is in $\bun{C}_0$, 
\[ |\langle \mu'+\eta,\sigma^{-1}\alpha^{\vee}\rangle|=|\langle \mu+\eta+ \nu,\sigma^{-1}\alpha^{\vee} \rangle |<  p+n-1.\]
It follows that
\[|\langle p\nu'-\sigma w\pi(\sigma^{-1})\pi(\nu'),\alpha^{\vee} \rangle|<2p+ n - m -1 \]
for all $\alpha^{\vee} \in \un{R}^{\vee}$. 
Let $M=\mathrm{max} _{\alpha^{\vee}\in \un{R}^{\vee}}\{ |\langle \nu',\alpha^{\vee} \rangle| \}$. By choosing $\alpha^{\vee}$ so that $ |\langle \nu',\alpha^{\vee} \rangle|=M$ in the above inequality, we get $(p-1)M<2p+n-m-1$, hence $M<2$. Thus we have shown $| \langle \nu',\alpha^{\vee} \rangle|\leq 1$ for all $\alpha^{\vee}\in \un{R}^{\vee}$. Hence for all $\alpha^{\vee}\in \un{R}^{\vee}$
\begin{equation}\label{ineq 3}
|\langle \sigma w\pi(\sigma^{-1})\pi(\nu'),\alpha^{\vee} \rangle|\leq 1.
\end{equation}

Since 
\[\sigma(\mu'+\eta)+p\nu'-\sigma w\pi(\sigma^{-1})\pi(\nu') =  \sigma(\mu+\eta) + \sigma(\nu) +p \nu' -  \sigma w\pi(\sigma^{-1})\pi (\nu') , \]
inequalities \eqref{ineq 1} (with $\alpha^{\vee}$ replaced by $\sigma^{-1}\alpha^{\vee}$), \eqref{ineq 2},\eqref{ineq 3} and the equality $\mu'=\mu+\nu$ together imply 
\[
m - n < |\langle  \sigma(\mu+\eta)+ p \nu' ,\alpha^{\vee} \rangle|= |\langle \mu+\eta,\sigma^{-1}\alpha^{\vee} \rangle+p\langle \nu',\alpha^{\vee} \rangle| < p - m + n
\]
for all coroots $\alpha^{\vee}$. 
It follows that for any coroot $\alpha^{\vee}$, the integer $\langle \mu+\eta, \alpha^{\vee} \rangle$ is not congruent modulo $p$ to any integer between $n-m$ and $m-n$. 
But since $\mu \in \bun{C}_0$, for any \emph{positive} coroot $\alpha^{\vee}$ we also have that $0<\langle \mu+\eta, \alpha^{\vee} \rangle<p$, and thus we in fact have
\[m-n<\langle \mu+\eta, \alpha^{\vee} \rangle<p-m+n\] 
This shows that $\tau$ is $(m-n)$-generic.

Finally, if $\rhobar$ is not the direct sum of characters or $\tau$ is not a principal series, let $K'/K$ be unramified extension over which both become true. By Lemma \ref{lemma:bc}, $\rhobar|_{G_{K'}}$ is $m$-generic.  We will use notation from the proof of Lemma \ref{lemma:bc}.  By \eqref{c1},  $\iota(s, \mu)$ is a lowest alcove presentation for $\tau|_{I_{K'}}$.   The same argument as above with $K$ replaced by $K'$ shows that $\iota(\mu)$ is $m-n$-deep in $\bun{C}_0'$.   Hence $\mu$ is $m-n$-deep in $\bun{C}_0$. 
\end{proof}

\subsection{Potential diagonalizability}\label{sec:pd}


Let $\rhobar$ be a tame representation. The goal of this section is to show that for certain sufficiently generic tame types $\tau$, all potentially crystalline lifts of $\rhobar$ of type $(\eta,\tau)$ are potentially diagonalizable in the sense of \cite{BLGGT}.

The main theorem is the following:

\begin{thm} \label{thm:smoothdef}  Let $\tau\cong \tau(s,\mu)$ be a tame type with a lowest alcove presentation $(s,\mu-\eta)$ such that $\mu-\eta$ is $(2n-1)$-deep in $\bun{C}_0$. Let $\rhobar:G_K\to \GL_n(\F)$ be a semisimple representation. Assume that there exists $\ovl{\fM}\in Y^{\eta,\tau}(\F)$ of shape $(t_{w_j(\eta_0)})_j$, where $w_j\in W(\GL_n)$ and such that $T^*_{dd}(\ovl{\fM})\cong \rhobar|_{G_{K_\infty}}$.
Then the framed potentially crystalline deformation ring $R_{\rhobar}^{\eta, \tau}$ is non-zero and formally smooth. 
\end{thm}

\begin{rmk} The proof of Theorem \ref{thm:smoothdef} uses the techniques introduced in \cite{LLLM} for $n =3$.   When $n =3$, the shapes $t_{w_j(\eta_0)}$ as in the Theorem correspond to the shapes $\alpha \beta \alpha \gamma$ and $\beta \gamma \alpha \gamma$ (and their cyclic permutations) in the Tables at the end of \cite{LLLM}. In Table 6 of \emph{loc. cit.}, the reader can see that the deformation ring for these shapes is formally smooth.  
\end{rmk}

For the rest of the section, we will be in the setting of the theorem. 
By Theorem \ref{fixedpoints} and the uniqueness of $\ovl{\fM}$ (cf.~ Theorem \ref{Kisinvariety}), $\ovl{\fM}$ is semisimple. We fix a gauge basis $\ovl{\beta}$ of $\ovl{\fM}$. By Remark \ref{monomial}, for all $j$, the matrices of the $j$-th partial Frobenius with respect to $\ovl{\beta}$ has the form
\begin{equation}\label{eqn:mod p frob}
\ovl{A}^{(j)}=\ovl{D}_j v^{w_j(\eta_0)}
\end{equation}
with $\ovl{D}_j \in T(\F).$  

We will need the following result, which gives a normal form for deformations of the pair $(\ovl{\fM},\ovl{\beta})$

\begin{prop} \label{prop:univlift} Let $R$ be complete local Noetherian $\cO$-algebra with maximal ideal $\mathfrak{m}$, residue field $\F$ and let $\tau\cong \tau(s,\mu)$ be a type with $\mu -\eta$ $n$-deep in $\bun{C}_0$.  Let $\fM \in Y^{[0,n-1], \tau}(R)$ such that $\fM \otimes \F \cong \overline{\fM}$.  Then there exists an eigenbasis $\beta$ lifting $\overline{\beta}$ such that the matrices of partial Frobenii $(A^{(j)})_{0 \leq j \leq f'-1}$ with respect to $\beta$ satisfy the following degree bounds:
\begin{enumerate}
\item[] $A^{(j)}_{ik} \in v^{\delta_{i>k}}R[v]$  and has degree $<{\delta_{i\geq k}}+ \langle w_j(\eta_0),\eps_k^\vee \rangle$.
\end{enumerate}
$($Note that automatically $A^{(j)}_{ik} \in v^{\delta_{i>k}}R[\![v]\!])$. 

Furthermore, such $\beta$ is uniquely determined up to scaling by the group $\{ (t_j)\in \ker({T(R)}\ra T(\F))^{f'} \mid t_j=t_{k} \textrm{ for }j\equiv k \textrm{ mod } f\}$.
\end{prop}

\begin{exam}\label{Example:fh}  Let $n = 3$, $f =1$, and $\tau$ be a generic principal series type. Let $\overline{\fM}$ be a Kisin module with shape $t_{(1,2,0)} = \beta \gamma \alpha \gamma$ in the notation of \cite{LLLM} and choose a gauge basis $\overline{\beta}$.  Proposition \ref{prop:univlift} says that any lift $\fM$ with height in $[0,n-1]$ has an eigenbasis $\beta$ lifting $\overline{\beta}$ such that the matrix $A^{(0)}$ has polynomial entries with degrees 
\[
\begin{pmatrix}  1 & \leq 1 & <0 \\
\leq 1 & 2 & <0 \\
 \leq 1 & \leq 2 & 0 \\
\end{pmatrix}
\]  
where the entries below diagonal are divisible by $v$ (compare with Table 5 in \emph{loc. cit.} where degree bounds are given for all admissible shapes).
\end{exam}
\begin{rmk} Our method of proof for Proposition \ref{prop:univlift} can be adapted easily (with more burdensome notations) to treat \emph{semisimple} Kisin modules $\overline{\fM}$ of more general shapes. On the other hand, the generalization to the situation where $\overline{\fM}$ is not semisimple requires more work. We only treat the case covered in Proposition \ref{prop:univlift} in this paper as this is all that we need, and leave these generalizations to future work.
\end{rmk}

\begin{proof}[Proof of Proposition \ref{prop:univlift}] The proof is a straightforward generalization of the arguments in \cite[\S 4]{LLLM}.
As in \emph{loc.~cit.}, we introduce a semi-valuation on $R$ given by $v_R(r) = \max \{ k \in \N \mid k \geq 0, r \in \mathfrak{m}^k\}$ and $v_R(0)=\infty$.
For $P = \sum_i r_i v^i \in R[\![v]\!]$, define $d(P) = \min_i ((n+3) v_R(r_i) + i)$. For a matrix $X$ with entries in $R[\![v]\!]$, define $d(X)$ to be the minimum of $d(X_{ij})$ where $X_{ij}$ are the entries of $X$, and for a tuple of matrices $(X_j)_j\in M_n(R[\![v]\!])^{f'}$ we define $d((X_j)_j)=\min_j d(X_j)$. Note that in all cases $d$ takes values in $\Z_{\geq 0} \cup \{\infty\}$. We have:
\begin{itemize}
\item $d(a+b)\geq \min \{d(a),d(b)\}$ for $a$, $b$ both in either $R[\![v]\!]$, $M_n(R[\![v]\!])$ or $M_n(R[\![v]\!])^{f'}$.
\item $d(ab)\geq d(a)+d(b)$ for $a$, $b$ both in either $R[\![v]\!]$ or $M_n(R[\![v]\!])$.
\end{itemize}
On any of the spaces $R[\![v]\!]$, $M_n(R[\![v]\!])$ or $M_n(R[\![v]\!])^{f'}$, the function $||a||=2^{-d(a)}$ defines a norm, which is furthermore submultiplicative in the first two cases. Thus each of these spaces are endowed with a metric topology, which is easily checked to be complete.

For each $0\leq j\leq f'-1$, we define the truncation operator $\ftr_j: M_n(R[\![v]\!])\to  M_n(R[\![v]\!])$ as follows: For $X\in M_n(R[\![v]\!])$,
\begin{itemize}
\item If $i<k$ then $\ftr_j(X)_{ik}$ is the sum of the terms in $X_{ik}\in R[\![v]\!]$ of degree $\geq   \langle w_j(\eta_0),\eps_k^\vee \rangle$;
\item If $i\geq k$ then $\ftr_j(X)_{ik}$ is the sum of the terms in $X_{ik}\in R[\![v]\!]$ of degree $>  \langle w_j(\eta_0),\eps_k^\vee \rangle$. 
\end{itemize}
In other words $\ftr_j(X)$ kills off precisely the part of $X$ that satisfies the degree bounds on $A^{(j)}$ in the conclusion. 
It is clearly an idempotent additive map.
We observe that our degree bounds are chosen precisely so that the image $\ftr_j$ is the subspace of $X$ such that $X v^{-w_j(\eta_0)}\in  M_n(R[\![v]\!])$ is integral and is furthermore upper triangular nilpotent mod $v$.

Note also since $d(vP)=1+d(P)$ for $P\in R[\![v]\!]$, we have 
\begin{equation} \label{ineq truncation}
d(\ftr_j(X) v^{-w_j(\eta_0)})\geq d(\ftr_j(X))-n+1\geq d(X)-n+1.
\end{equation}
We also define $\ftr:M_n(R[\![v]\!])^{f'}\to M_n(R[\![v]\!])^{f'}$ by $\ftr((X_j)_j)=(\ftr_j(X_j))_j$.

We will show that for any given eigenbasis $\beta$ of $\fM$ lifting $\ovl{\beta}$, there is a unique $\beta'$ lifting $\ovl{\beta}$ such that $\beta'\equiv\beta$ mod $u'$, and $\beta'$ satisfies the conclusion of the Proposition. This proves the Proposition, since the set of all possible $\beta$ mod $u'$ forms a torsor for the group $\{ (t_j)\in \ker({T(R)}\ra T(\F))^{f'} \mid t_j=t_{k} \textrm{ for }j\equiv k \textrm{ mod } f\}$.

We now fix an eigenbasis $\beta$ lifting $\ovl{\beta}$. Our strategy will be to interpret the problem of finding $\beta'$ as finding a fixed point for certain mapping on a complete subspace of $M_n(R[\![v]\!])^{f'}$. We then show that this mapping is contracting on this subspace, and the Proposition follows by the Contraction Mapping Theorem.

By Remark \ref{rmk:changeofbasistuple}, prescribing any other eigenbasis $\beta'$ of $\fM$ is the same as prescribing a change of basis tuple $(I^{(j)})_j\in \cI(R)^{f'}$ such that $I^{(j)}$ depends only on $j$ mod $f$. The condition that $\beta'$ also lifts $\ovl{\beta}$ is equivalent to $I^{(j)} \equiv 1$ mod $\mathfrak{m}$, and the condition that $\beta' \equiv \beta$ mod $u'$ is equivalent to $I^{(j)}\in \cI_1(R)$. Thus the tuple $(X_j)_j=(I^{(j)}-1)_j$ satisfies
\begin{itemize}
\item $X_j$ depends only on $j$ mod $f$;
\item $X_j \equiv 0$ mod $\mathfrak{m}$;
\item $X_j$ is upper triangular nilpotent mod $v$.
\end{itemize}
This leads us to define the subspace $V\subset M_n(R[\![v]\!])^{f'}$ consisting of tuples satisfying all these conditions. Clearly $V$ is stable under component-wise addition and is easily seen to be a closed subspace, hence is also complete.

Let $(A^{(j)})_j$ be the tuple of matrices of partial Frobenii with respect to $\beta$. Since $\beta$ lifts $\ovl{\beta}$, $\ovl{A}^{(j)}$ has the form given in (\ref{eqn:mod p frob}), thus we can decompose
\[  A^{(j)}=D_jv^{w_j(\eta_0)}+M_j \]
with $D_j\in T(R)$ and $M_j\in M_n(\mathfrak{m}[\![v]\!])$. We can and will assume that this decomposition has been chosen so that $D_j$ and $M_j$ only depend on $j$ mod $f$.

By Proposition \ref{prop:changeofbasis}, our problem of finding $\beta'$ now reduces to finding $(X_j)_j\in V$ such that for all $j$
\begin{equation}\label{main eqn}
\ftr_j( (1+X_{j+1}))A^{(j)} \Ad(s_j^* v^{\mu^*_j})(\phz((1+X_j)^{-1})) =0
\end{equation}
by Proposition \ref{prop:changeofbasis} (Recall that $\Ad(g)(M) := gMg^{-1}$.) 
To lighten the notation, we put $Y_j=Y_j(X_j)= \Ad(s_j^* v^{\mu^*_j})(\phz((1+X_j)^{-1}))$, and think of it as a function in $X_j$. We now rewrite the left-hand side of the above equation as
\begin{align*}
 \qquad &\ftr_j ((1+X_{j+1})(D_j v^{w_j(\eta_0)}+M_j))+\ftr_j ((1+X_{j+1})A^{(j)}(Y_j-1))\\
=\quad&\ftr_j(D_j v^{w_j(\eta_0)})+\ftr_j(X_{j+1}D_j v^{w_j(\eta_0)})+\ftr_j((1+X_{j+1})M_j)+\ftr_j ((1+X_{j+1})A^{(j)}(Y_j-1))\\
=\quad&X_{j+1}D_j v^{w_j(\eta_0)}+\ftr_j((1+X_{j+1})M_j)+\ftr_j ((1+X_{j+1})A^{(j)}(Y_j-1)),\\
\end{align*}
where the last equality is due to the fact that $X_{j+1}\in M_n(R[\![v]\!])$ is upper triangular nilpotent mod $v$.

Thus equation (\ref{main eqn}) is equivalent to $(X_j)_j$ being a fixed point of the map $H: V \to  M_n(R[\![v]\!])^{f'}$ given by
\[H((X_j)_j)= \big( -(\ftr_{j-1}((1+X_{j})M_{j-1})+\ftr_{j-1} ((1+X_{j})A^{(j-1)}(Y_{j-1}(X_{j-1})-1)))v^{-w_{j-1}(\eta_0)}(D_{j-1})^{-1}   \big)_j\]
Note that the assumption $\mu - \eta$ is $n$-deep in $\bun{C}_0$ implies that $(Y_j-1)_j\in V$, so that each expression that gets truncated is indeed in the domain of definition of the appropriate truncation operator. 
  
Clearly $H((X_j)_j)$ satisfies the first property defining $V$. Now $H((X_j)_j)$ satisfies the second property defining $V$, since truncation operators preserve the property of being 0 mod $\mathfrak{m}$, and $M_j=0$ mod $\mathfrak{m}$ since $\overline{A}^{(j)}$ satisfied the correct degree bounds.
Finally, the description of the image of $\ftr_j$ implies that $H((X_j)_j)$ satisfies the third property defining $V$. Thus $H$ maps $V$ to $V$.

The proof of the Proposition is complete once we have the following:
\begin{lemma} We have
\[d(H(a)-H(b))\geq d(a-b)+1\]
for $a,b\in V$. 
\end{lemma}
\begin{proof}
Put $a=(X_j)_j$ and $b=(X_j+\Delta_j)_j$. Put $\delta=d(a-b)=\min_j d(\Delta_j)$.

On the $(j+1)$-th component, we have
\begin{align*}
 &(H(a)_{j+1}-H(b)_{j+1})D_jv^{w_j(\eta_0)}\\
=\quad &\ftr_j(\Delta_{j+1}M_j)+\ftr_j ((1+X_{j+1}+\Delta_{j+1})A^{(j)}(Y_j(X_j+\Delta_j)-1)) -\ftr_j ((1+X_{j+1})A^{(j)}(Y_j(X_j)-1))\\
=\quad &\ftr_j(\Delta_{j+1}M_j)+\ftr_j (\Delta_{j+1}A^{(j)}(Y_j(X_j+\Delta_j)-1))+\ftr_j ((1+X_{j+1})A^{(j)}(Y_j(X_j+\Delta_j)-Y_j(X_j)))
\end{align*}
For the first term, since $M_j\in  M_n(\mathfrak{m}[\![v]\!])$, $d(M_j)\geq n+3$. Thus
\[d(\ftr_j(\Delta_{j+1}M_j))\geq d(\Delta_{j+1})+d(M_j)\geq \delta+ n+3.\]
For the second term, as observed before $Y_j(X_j+\Delta_j)-1\in  M_n(\mathfrak{m}[\![v]\!])$, and we similarly have
\[d(\ftr_j (\Delta_{j+1}A^{(j)}(Y_j(X_j+\Delta_j)-1)))\geq \delta+ n+3.\]
For the third term, we observe
\begin{align*}
&d(\ftr_j ((1+X_{j+1})A^{(j)}(Y_j(X_j+\Delta_j)-Y_j(X_j))))\geq d(Y_j(X_j+\Delta_j)-Y_j(X_j))\\
=\quad&d(-\Ad(s_j^*v^{\mu^*_j})(\phz((1+X_j+\Delta_j)^{-1})\phz(\Delta_j)\phz((1+X_j)^{-1}))))\\
\geq \quad &d(\Ad(s_j^* v^{\mu^*_j})(\phz(\Delta_j))=d(\Ad(v^{\mu^*_j})(\phz(\Delta_j)),
\end{align*}
where the second inequality is due to the fact $\Ad(s_j^*v^{\mu^*_j})(\phz(X))\in M_n(R[\![v]\!])$ for $X\in \cI_1(R)$ (this uses that $\mu-\eta$ is $0$-deep in alcove $\un{C}_0$). For the diagonal entries, we have
\[d((\Ad(v^{\mu^*_j})(\phz(\Delta_j)))_{ii})=d(\phz(\Delta_j)_{ii})\geq p-1+d((\Delta_j)_{ii})\geq p-1+\delta\geq \delta+n\]
by the observation that $d(\phz(P))=p+d(\phz(P/v))\geq p+d(P/v)=p-1+d(P)$ for $P\in vR[\![v]\!]$.
For the $\alpha$-th entry where $\alpha$ is a root of $\GL_n$, we have
\[d((\Ad(v^{\mu^*_j})(\phz(\Delta_j)))_{\alpha})=d(\phz(\Delta_j)_{\alpha})+\langle \mu^*_j, \alpha^{\vee} \rangle\]
By the above observation, the fact that $(\Delta_j)_{\alpha} \in  vR[\![v]\!]$ for $\alpha<0$ and the fact that $\mu-\eta$ is $n$-deep in alcove $\un{C}_0$, we conclude that for all roots $\alpha$
\[d((\Ad(v^{\mu^*_j})(\phz(\Delta_j)))_{\alpha})\geq \delta+n.\]
Thus 
\[d(\ftr_j ((1+X_{j+1})A^{(j)}(Y_j(X_j+\Delta_j)-Y_j(X_j))))\geq \delta+n.\]
Putting everything together, we obtain
\[d(H(a)_{j+1}-H(b)_{j+1})\geq d((H(a)_{j+1}-H(b)_{j+1})D_j v^{w_j(\eta_0)})-n+1\geq \delta+n-n+1=\delta+1\]
as desired.

\end{proof}

 Thus we deduce that the map $H:V \to V$ is a contraction mapping. In particular, $H$ has a unique fixed point in $V$, which is what we wanted.


\end{proof}

We call a basis as above a \emph{gauge basis} (lifting $\ovl{\beta}$) of the deformation $\fM$ of $\ovl{\fM}$. This is consistent with \cite[Definition 4.15]{LLLM}. Since we have fixed the data ($\ovl{\fM}$,$\ovl{\beta}$), we will suppress the dependence on $\ovl{\beta}$. 

For each $0\leq j<f'$, the deformation problem that assigns to each Artinian $\cO$-algebra $A$ with residue field $\F$ the set of matrices $A_j \in M_n(A[\![v]\!])$ lifting $\ovl{A}^{(j)}$ and satisfies the degree bounds in Proposition \ref{prop:univlift} is clearly representable by a complete local Noetherian $\cO$-algebra $R_j$, which is a formal power series ring over $\cO$. It carries the universal matrix $A_j^{\univ}$.

 For any Artinian $\cO$-algebra $A$ with residue field $\F$, let $D^{\tau, \overline{\beta}}_{\overline{\fM}}(A)$ be the category of pairs $(\fM_A, \beta_A)$ deforming $(\overline{\fM}, \overline{\beta})$ where $\fM_A \in Y^{\eta, \tau}(A)$ and $\beta_A$ is a gauge basis of $\fM_A$.  We would like to give an ``explicit'' presentation for $D^{\tau, \overline{\beta}}_{\overline{\fM}}$ as in \cite[Theorem 4.17]{LLLM}. Define $R_{{w_j(\eta_0)}}^{\expl}$ to be the quotient of the formal power series ring $R_j$ above by the height $\leq \eta$ relations:
\begin{itemize} 
\item $\det A_j^{\univ}=x_j^*(v+p)^{n(n-1)/2}$ where $x_j^*\in (R_j[\![v]\!])^\times$
\item $(v+p)^{k(k-1)/2}$ divides each $k\times k$ minor of $A_j^{univ}$. Note that the condition that $(v+p)^l$ divides $P\in R_j[\![v]\!]$ can be expressed as 
\[P|_{v=-p}=0,\qquad \frac{d}{dv}P|_{v=-p}=0,\qquad \ldots, \qquad (\frac{d}{dv})^{l-1}P|_{v=-p}=0,\]
which are algebraic conditions in the coefficients of $P$.
\end{itemize}

\begin{prop} The functor $D^{\tau, \overline{\beta}}_{\overline{\fM}}$ is representable by the complete local Noetherian $\cO$-algebra 
\begin{equation}
R^{\tau, \overline{\beta}}_{\overline{\fM}} = \widehat{\otimes}_{0\leq j<f} (R_{{w_j(\eta_0)}}^{\expl})^{p\text{-flat, red}}.
\end{equation}
\end{prop}
\begin{proof}
By Proposition \ref{prop:univlift}, there exists a closed immersion $D^{\tau, \overline{\beta}}_{\overline{\fM}} \iarrow \Spf \, \widehat{\otimes}_{0\leq j<f} R_j$ (note that $Y^{\eta, \tau}$ is closed in $Y^{[0, n-1], \tau}$) so $D^{\tau, \overline{\beta}}_{\overline{\fM}}$ is representable by a quotient $R^{\tau, \overline{\beta}}_{\overline{\fM}}$ of $\widehat{\otimes}_{0\leq j<f} R_j$. 

By \cite[Theorem 5.3]{CL}, $Y^{\eta, \tau}$ is equisingular to the local model $M(\eta)$. Since the addition of a gauge basis is formally smooth, $R^{\tau, \overline{\beta}}_{\overline{\fM}}$ is $p$-flat and reduced.   It suffices then to compare $\overline{\Q}_p$-points of $R^{\tau, \overline{\beta}}_{\overline{\fM}}$ and $\widehat{\otimes}_{0\leq j<f} (R_{{w_j(\eta_0)}}^{\expl})^{p\text{-flat, red}}$.  

Let $F/E$ be a finite extension with ring of integers $\cO_F$. Let $x: \widehat{\otimes}_{0\leq j<f}  R_j \ra \cO_F$ with associated Kisin module $\fM_x$ and the matrix of partial Frobenii given by $A_{j, x}$.   By Theorem 5.13 in \cite{CL}, $\fM_x$ lies in $Y^{\eta, \tau}(\cO_F)$ and hence $D^{\tau, \overline{\beta}}_{\overline{\fM}}(\cO_F)$ if and only if $\fM_x[1/p]$ has $p$-adic Hodge type $\leq \eta$. In our notation, $\fM_x[1/p]$ has $p$-adic Hodge type $\leq \eta$ if for each $j$, the filtration on $\fM^{(j)}_{x, \chi_{s'_{\mathrm{or}}(n)}}[1/p]$ is of type $\mu_j$ with $\mu_j \leq \eta_0$.   The filtration is induced by the partial Frobenius with matrix $A_{j,x}$ and so this is equivalent to the condition that the elementary divisors of $A_{j, x}$ as a matrix over $F[\![v+p]\!]$ are $(v+p)^{\mu_j}$ for each $0 \leq j < f$. Thus $\fM_x$ lies in $D^{\tau, \overline{\beta}}_{\overline{\fM}}(\cO_F)$ if and only if the elementary divisors of $A_{j, x}$ as a matrix over $F[\![v+p]\!]$ are bounded by $(v+p)^{\eta_0}$ for each $0 \leq j < f$. But this condition is exactly the divisibility condition on the minors and the determinant condition on $A_{j, x}$ imposed by the relations defining $R_{{w_j(\eta_0)}}^{\expl}$.
\end{proof}

Let ($\fM^{\mathrm{univ}},\beta^{\mathrm{univ}})$ be the universal pair living over $R^{\tau, \overline{\beta}}_{\overline{\fM}}$.
\begin{prop} \label{prop:univfhlift} \begin{enumerate}
\item
Over $R^{\tau, \overline{\beta}}_{\overline{\fM}}$, the universal matrices of partial Frobenii of $\fM^{\mathrm{univ}}$ with respect to $\beta^{\mathrm{univ}}$ have the form
\[A^{(j), \univ}=D_j^{\univ}(v+p)^{w_j(\eta_0)}U^{(j),\univ}\]
for $0\leq j<f$, where
\begin{itemize}
\item $D_j^{\univ}\in T(R^{\tau, \overline{\beta}}_{\overline{\fM}})$ lifts $\ovl{D}_j$.
\item $w_j^{-1} U^{(j),\univ} w_j$ is lower triangular unipotent, and for any root $\alpha=\eps_i-\eps_k$ of $\GL_n$, its $\alpha$-th entry is a polynomial with topologically nilpotent coefficients of the form $v^{\delta_{w_j(\alpha)<0}} f^{(j)}_{\alpha}(v)$ where $\deg f^{(j)}_{\alpha}(v)<-\langle \alpha^\vee,\eta_0\rangle= i-k$.
\end{itemize}   
\item $R^{\tau, \overline{\beta}}_{\overline{\fM}}$ is the formal power series ring over $\cO$ generated by the coefficients $X_{\alpha}^{(j),l}$ of $f^{(j)}_{\alpha}$ $($where $0\leq j<f$, $\alpha<0$ a negative root of $\GL_n$, $0\leq l<-\langle \alpha^\vee,\eta_0\rangle$) and the variables $c_{ii}^{(j)}=(D_j^{\univ})_{ii}-[(\ovl{D}_j)_{ii}]$ (where $1\leq i\leq n$, $0\leq j<f$ and $[\cdot]$ denotes the Teichmuller lift$)$. 
\end{enumerate}
\end{prop}


\begin{exam} In the situation of Example \ref{Example:fh}, that is in for $n = 3$, $f =1$, $\tau$ a generic principal series type and $\overline{\fM}$ has shape $t_{(1,2,0)} = \beta \gamma \alpha \gamma$, the Proposition asserts that the universal deformation of $(\overline{\fM},\overline{\beta})$ living over $R^{\tau, \overline{\beta}}_{\overline{\fM}}$ has matrix of Frobenius of the form
\[\begin{pmatrix}  (v+p)c^*_{11} & (v+p)c_{12} & 0 \\
0 &  (v+p)^2c^*_{22} & 0 \\
 vc_{31} & v(c_{32}+(v+p)c'_{32}) & c^*_{33} \\
\end{pmatrix}
\]  
where the starred coefficients are units and the non-starred coefficients are topologically nilpotent. This is exactly what is given in Table 5 of \cite{LLLM}.
The ring $R^{\tau, \overline{\beta}}_{\overline{\fM}}$ is a power series ring over $c_{12}, c_{31},c_{32},c'_{32}, c^*_{11}-[\overline{c}^*_{11}], c^*_{22}-[\overline{c}^*_{22}], c^*_{33}-[\overline{c}^*_{33}]$.
\end{exam}

\begin{proof} 
\begin{enumerate}
\item
We work with fixed $j$, and set $B=w_j^{-1}A^{(j),\univ}w_j$. The fact that $A^{(j),\univ}$ obeys the degree bounds in Proposition \ref{prop:univlift} implies:
\begin{itemize}
\item $B_{ii}\in R^{\tau, \overline{\beta}}_{\overline{\fM}}[v]$ and has degree $\leq n-i$.
\item For any root $\alpha=\eps_i-\eps_k$ of $\GL_n$, $v^{-\delta_{w_j(\alpha)<0}}B_{ik}$ is in $R^{\tau, \overline{\beta}}_{\overline{\fM}}[v]$ and has degree $<n-k$.
\end{itemize} 
We claim that these conditions together with the height $\leq$ relations force $B$ to be lower triangular, $B_{ii}$ to be $u_i(v+p)^{n-i}$ with $u_i \in (R^{\tau, \overline{\beta}}_{\overline{\fM}})^\times$, and $B_{ik}$ to be divisible by $v^{\delta_{w_j(\eps_i-\eps_k)<0}}(v+p)^{n-i}$. This finishes the proof, since we can then uniquely factorize 
\[B=D_j(v+p)^{\eta_0}U\]
with $D_j\in T(R^{\tau, \overline{\beta}}_{\overline{\fM}})$ and $U$ lower triangular unipotent (whose entries obey degree bounds deduced from the bounds for $B$), and conjugating by $w_j$ yields the desired factorization of $A^{(j),\univ}$. Note that the non-diagonal entries of $U$ are necessarily topologically nilpotent since $B$ is diagonal modulo the maximal ideal.

We now prove the claim by downward induction on the index of the rows and columns. We start by showing the claim for entries in the $n$th column and $n$-th row of $B$. The degree bounds imply $B_{in}=0$ for $i<n$, while $B_{nn}\in  R^{\tau, \overline{\beta}}_{\overline{\fM}}$. Furthermore $B_{nn}$ is a unit since it lifts a unit in the residue field. The claim is empty for all other entries of the $n$-th row of $B$. Suppose we claim holds for all entries in the last $k-1$ rows and columns. Then the induction hypothesis and the condition that each $k\times k$ minor of $A^{(j),\univ}$ (and hence also each $k\times k$ minor of of $B$) is divisible by $(v+p)^{k(k-1)/2}$ implies:
\begin{itemize}
\item Looking at the minor formed by the last $k$ columns, the last $(k-1)$ rows and the $i$-th row of $B$ for $i\leq n-k+1$, we get  
\[(v+p)^{k(k-1)/2} \mid B_{i(n-k+1)} \prod_{l=0}^{k-2} u_{n-l}(v+p)^{l},\]
thus $(v+p)^{k-1} \mid B_{i(n-k+1)}$. For $i=n-k+1$, since $B_{(n-k+1)(n-k+1)}$ has degree $\leq k-1$, we must have $B_{(n-k+1)(n-k+1)}=u_{(n-k+1)}(v+p)^{k-1}$, and $u_{(n-k+1)}$ is a unit since it lifts a unit in the residue field. On the other hand, for $i<n-k+1$, the degree bounds imply that $B_{i(n-k+1)}$ is of the form $v^{\delta}$ times a polynomial of degree $<k-1$, for $\delta\in \{0,1\}$. However if $p$ is regular in a ring $R$, the condition that $vP$ is divisible by $(v+p)^l$ for $P\in R[v]$ is equivalent to $P $ divisible by $(v+p)^l$ (this can be seen by using the interpretation of this condition in terms of vanishing up to $(l-1)$-th order derivatives of $P$ evaluated at $v=-p$). Since $p$ is regular in $R^{\tau, \overline{\beta}}_{\overline{\fM}}$, we conclude that $B_{i(n-k+1)}=0$.
\item Looking at the minor formed by the last $k$ rows, the last $k-1$ columns, and the $i$-th column of $B$ for $i<n-k+1$, we get  
\[(v+p)^{k(k-1)/2} \mid B_{(n-k+1)i} \prod_{l=0}^{k-2} u_{n-l}(v+p)^{l},\]
thus $(v+p)^{k-1} \mid B_{(n-k+1)i}$. We get the claim about divisibility by $v^{\delta_{w_j(\eps_{n-k+1}-\eps_i)<0}}(v+p)^{k-1}$of $B_{(n-k+1)i}$ by the same reasoning as above.
\end{itemize}
\item
We observe that for a polynomial $P\in R[v]$ with given degree, the condition that $P$ is divisible by $(v+p)^l$ is equivalent to solving the first $l$ coefficients of $P$ in terms of the remaining ones.

Thus, for $j$, there is a quotient $\tld{R}_j$ of the ring $R_j$ over which the universal matrix with degree bounds $A_j^{\univ}$ has the form in the first part of the Proposition (as this is equivalent to asking that each entry is either 0 or is divisible by certain powers of $v+p$ and $v$), and $\tld{R}_j$ is exactly the power series ring in the variables described in the second part of the Proposition. Furthermore, as the specialization of the universal matrix $A_j^{\univ}$ to $\tld{R}_j$ clearly satisfies the determinant and the height conditions defining $R_{{w_j(\eta_0)}}^{\expl}$, we conclude that there is a factorization $R_j\onto R_{{w_j(\eta_0)}}^{\expl}\onto \tld{R}_j$.
Part $(1)$ then shows that there is a factorization $\widehat{\otimes}_{0\leq j<f}  R_{{w_j(\eta_0)}}^{\expl}\onto\widehat{\otimes}_{0\leq j<f} \tld{R}_{j} \onto R^{\tau, \overline{\beta}}_{\overline{\fM}}$. But since $\widehat{\otimes}_{0\leq j<f} \tld{R}_{j}$ is a power series ring, hence reduced and $p$-flat, and $R^{\tau, \overline{\beta}}_{\overline{\fM}}$ is the maximal reduced and $p$-flat quotient of $\widehat{\otimes}_{0\leq j<f}  R_{{w_j(\eta_0)}}^{\expl}$, the last quotient map is an isomorphism.
\end{enumerate}

\end{proof}

We now recall the monodromy condition on the universal Kisin module $\fM^{\univ}$ over $R^{\tau, \overline{\beta}}_{\overline{\fM}}$, as in \cite[\S 5.1]{LLLM}, \cite{KisinFcrys}. We refer to \emph{loc.~cit.}~ for undefined symbols. On $\fM^{\univ}\otimes \cO^{\mathrm{rig}}_{R^{\tau, \overline{\beta}}_{\overline{\fM}}}$, there is a canonical derivation over the differential operator $-\lambda u'\frac{d}{d(u')}$, the \emph{monodromy operator}, which is meromorphic along $\lambda$ (in fact it has poles of order $\leq n-2$ due to the finite height conditions we imposed). The \emph{monodromy condition} is the condition that this operator has no poles. On the closed points of $\Spec R^{\tau, \overline{\beta}}_{\overline{\fM}}[\frac{1}{p}]$, this condition precisely cuts out the (Zariski closed) locus where the induced Kisin module comes from a potentially crystalline representation (which is necessarily of inertial type $\tau$ and Hodge-Tate weight $\leq \eta$).

We recall some more deformation problems attached to the current situation, similar to \cite[Definition 5.10]{LLLM} (when $\tau$ is principal series) and \cite[\S 6.2]{LLLM} (for general $\tau$). All data below are understood to be compatible with the given data $\rhobar$, $\ovl{\fM}$, $\ovl{\beta}$, etc.
\begin{enumerate}
\item $R^{\eta,\tau}_{\rhobar}$ is the framed potentially crystalline deformation ring of type $(\eta,\tau)$ as in \cite{KisinPSS}. By \cite[Theorem 3.3.4]{KisinPSS}, if this ring is not zero it has Krull dimension $\dim R^{\eta,\tau}_{\rhobar}=\frac{n(n-1)f}{2}+n^2+1$.
We denote the deformation problem it represents $D_{\rhobar}^{\tau, \Box}$.

\item Let $R_{\overline{\fM}, \rhobar}^{\tau, \Box}$ denote the complete local Noetherian $\cO$-algebra which represents the deformation problem
$$
D_{\overline{\fM}, \rhobar}^{\tau, \Box}(A) := \left \{ (\fM_A, \rho_A, \delta_A) \mid \fM_A \in  Y^{\mu, \tau}(A), \rho_A \in D_{\rhobar}^{\tau, \Box}(A),  \delta_A:T_{dd}^*(\fM_A) \cong (\rho_A)|_{G_{K_{\infty}}} \right \}.
$$
Thanks to Proposition \ref{Kisinvariety} (and our running hypothesis that $\tau$ is generic), this deformation problem is representable, and in fact is representable by $R^{\eta,\tau}_{\rhobar}$ as explained in \cite[Section 5.2]{LLLM}.
\item Let $R^{\tau, \overline{\beta}, \Box}_{\overline{\fM}, \rhobar}$ denote the complete local Noetherian $\cO$-algebra which represents the deformation problem
$$
D^{\tau, \overline{\beta}, \Box}_{\overline{\fM}, \rhobar}(A) = \left\{ (\fM_A, \rho_A, \delta_A,\beta_A) \mid (\fM_A, \rho_A,\delta_A) \in D_{\overline{\fM}, \rhobar}^{\tau,\Box}(A), \beta_A \text{ a gauge basis for } \fM_A \right \}.
$$

\item Let $R^{\tau, \overline{\beta}, \Box}_{\overline{\fM}}$ denote the complete local Noetherian $\cO$-algebra which represents the deformation problem of triples $(\fM_A, \beta_A, \un{e}_A)$ where $(\fM_A, \beta_A) \in D^{\tau, \overline{\beta}}_{\overline{\fM}}(A)$ and $\un{e}_A$ is a basis of $T_{dd}^*(\fM_A)$ lifting the basis on $\rhobar|_{G_{K_{\infty}}}$ so that $(T_{dd}^*(\fM_A), \un{e}_A)$ is a framed deformation of $\rhobar|_{G_{K_{\infty}}}$.
\item Let $R^{\tau, \overline{\beta}, \nabla}_{\overline{\fM}}$ denote the $\cO$-flat and reduced quotient of $R^{\tau, \overline{\beta}}_{\overline{\fM}}$ such that $\Spec R^{\tau, \overline{\beta}, \nabla}_{\overline{\fM}}[1/p]$ is the locus where the monodromy condition holds on $\Spec R^{\tau, \overline{\beta}}_{\overline{\fM}}[1/p]$.  
 We define  $R^{\tau, \overline{\beta},\Box, \nabla}_{\overline{\fM}}$ from $R^{\tau, \overline{\beta},\Box}_{\overline{\fM}}$ in a similar way.
\end{enumerate}

We recall \cite[Diagram (5.9)]{LLLM}, which summarizes the relationship between the above deformation problems.  The square is Cartesian and f.s.~ stands for formally smooth. 
\begin{equation} \label{defdiagram}
\xymatrix{
& & \Spf R^{\tau, \overline{\beta}, \square, \nabla}_{\overline{\fM}} \ar[r]^{f.s.} \ar@{^{(}->}[d]  & \Spf R^{\tau, \overline{\beta}, \nabla}_{\overline{\fM}} \ar@{^{(}->}[d] \\
 &\Spf R^{\tau, \overline{\beta}, \Box}_{\overline{\fM}, \rhobar} \ar[d]^{f. s.}  \ar@{^{(}.>}[ru]^{\xi} \ar@{^{(}->}[r]  & \Spf R_{\overline{\fM}}^{\tau, \overline{\beta}, \Box} \ar[r]^{f.s.} & \Spf R_{\overline{\fM}}^{\tau, \overline{\beta}}  \\
 \Spf R^{\eta, \tau}_{\rhobar} & \Spf R_{\overline{\fM},\rhobar}^{\tau, \Box}  \ar[l]_{\sim}  
}
\end{equation}

The maps which are formally smooth correspond to forgetting either a framing on the Galois representation or a gauge basis on the Kisin modules (the fact that adding gauge basis is a formally smooth operation is due to Proposition \ref{prop:univlift}). The fact that the horizontal arrow below the dotted arrow is a closed immersion is due to our assumption that $\tau$ is (at least) generic, which implies $\ad(\rhobar)$ is cyclotomic free (e.g.~ by looking at the inertial weights, which are easily read off by applying Corollary \ref{cor:phiandKisin}), and hence the argument of \cite[Proposition 5.11]{LLLM} applies. This shows that if the dotted arrow exists, it must be a closed immersion. 
We show below that the dotted arrow exists, and is furthermore an isomorphism. The fact that it is an isomorphism rather than just a closed immersion is because in the present situation, the elementary divisors of the matrices of partial Frobenii of $\overline{\fM}$ with respect to $\overline{\beta}$ are \emph{exactly} $(v+p)^{\eta_0}$, and thus no lift $\fM$ of $\overline{\fM}$ can satisfy a height $\leq\lambda$ relation for $\lambda<\eta$.

\begin{prop}\label{prop:factors} The natural map  $ R^{\tau, \overline{\beta}}_{\overline{\fM}}\to  R_{\overline{\fM},\rhobar}^{\tau, \overline{\beta}, \Box}$ factors through the quotient $R^{\tau, \overline{\beta}, \nabla}_{\overline{\fM}}$. The induced map $\xi$ is an isomorphism.
\end{prop}
\begin{proof} The proof is completely analogous to the proof \cite[Theorem 5.12]{LLLM}. As both target rings are reduced and $p$-flat, we only need to check the factorization exists on closed points of the generic fibers. However, this is just the statement that a Kisin module coming from a potentially crystalline Galois representation satisfies the monodromy condition.

To see that the map $\xi$ is an isomorphism, we note that the only closed points in $\Spec R^{\tau, \overline{\beta}, \Box,\nabla}_{\overline{\fM}}[\frac{1}{p}]$ that do not come from 
$\Spec R_{\overline{\fM},\rhobar}^{\tau, \overline{\beta}, \Box}[\frac{1}{p}]$ are those for which the $j$-th component of the underlying Kisin module has elementary divisors strictly dominated by $(v+p)^{\eta_0}$, for some $0\leq j<f$ (this corresponds to the condition that the Hodge-Tate weight of the $j$-th embedding of the corresponding Galois representation is $<\eta_0$). However, this possibility is ruled out by the form of the universal Kisin module given in Proposition \ref{prop:univfhlift}
\end{proof}

\begin{cor} \label{cor:existencelift} For $\rhobar, \tau$ as in Theorem \ref{thm:smoothdef}, there exists a closed point $x\in \Spec R^{\eta,\tau}_{\rhobar}[\frac{1}{p}]$ such that the corresponding Galois representation $\rho_x$ becomes a direct sum of characters after restriction to a finite index subgroup. In particular  $R^{\eta,\tau}_{\rhobar}\neq 0$.
\end{cor}
\begin{proof}
Recall that we have fixed a gauge basis $\ovl{\beta}$ of $\ovl{\fM}$ such that for all $j$, the matrices of the $j$-th partial Frobenius with respect to $\ovl{\beta}$ has the form
\begin{equation}
\ovl{A}^{(j)}=\ovl{D}_j v^{w_j(\eta_0)}
\end{equation}
with $\ovl{D}_j \in T(\F).$  
Using Proposition \ref{prop:univfhlift}, we can produce an $\cO$ point of $R^{\tau,\ovl{\beta}}_{\ovl{\fM}}$ such that the matrices of partial Frobenii are monomial matrices of the form
\[A^{(j)}=D_j(v+p)^{w_j(\eta_0)}.\]
by choosing a diagonal matrices $D_j\in T(\cO)$ lifting $\ovl{D}_j$.
A Kisin module of this form becomes isomorphic to a direct sum of rank $1$ Kisin modules after passing to a finite unramified extension $\breve{K}$ of $K'$. Since the monodromy condition can be checked after base change, and always holds for rank $1$ Kisin modules (by an easy computation), we deduce that the above Kisin module satisfies the monodromy condition. Lifting this point along the formally smooth maps in the diagram \ref{defdiagram} yields a closed point $x\in \Spec R^{\eta,\tau}_{\rhobar}[\frac{1}{p}]$. As the underlying Kisin module of $x$ decomposes into direct sum of rank $1$ Kisin modules over $\breve{K}$, $\rho_x$ becomes a direct sum of characters over $\breve{K}$.

Alternatively, we can also directly produce $\rho_x$ as a direct sum of inductions of (potentially crystalline) characters for unramified extensions of $K$ and then check that it comes from a Kisin module with the above form, hence has type $(\eta,\tau)$.
\end{proof}
\begin{prop}\label{prop:upperbounddefring} $R^{\tau, \overline{\beta}, \nabla}_{\overline{\fM}}/\varpi$ is a quotient of a power series ring over $\F$ in $\frac{n(n+1)f}{2}$ variables.
\end{prop}
\begin{proof}
We work over $R=R^{\tau, \overline{\beta}}_{\overline{\fM}}$ with universal Kisin module $\fM=\fM^{\univ}$ and universal gauge basis $\beta$. This determines the matrices $A^{(j)}$ as before. It will suffice to analyze the monodromy condition on $\fM$ viewed as a Kisin module over $K'$ with descent data corresponding to the base changed type $\tau'$ (which is a principal series type), and we will do so by closely following the computations in \cite[\S 5.1]{LLLM} (especially \cite[Theorem 5.6]{LLLM}). As in \emph{loc.~cit.}, we have the matrices $C^{(j)}$ which are determined by the $A^{(j)}$ and our chosen presentation of $\tau$, and a ring $\cO^{\mathrm{rig}}_R$. We note that the variable $u$ in \emph{loc.~cit.}~ corresponds to our variable $u'$, the variable $v$ there is the same as our $v=(u')^{e'}$, and $E(u)$ there is $(u')^{e'} + p = v+p$. 
Exactly as in \cite[Lemma 5.2]{LLLM} we have a formula for the $j$-th component $N_\infty^{(j)}$ of the monodromy operator:
\begin{align*}
N_{\infty}^{(j)}=N_1^{(j)} + \sum^{\infty}_{i = 1} \left(\prod_{k=0}^{i-1}\phz^{k}(C^{(j-k-1)})\right) \phz^i(N_{1}^{(j-i)})\left(\prod_{k=i-1}^{0}\phz^{k} (C^{(j-k-1), *}) \right)
\end{align*}
where $C^{(j), *} :=  (v+p) (C^{(j)})^{-1}$,
\[N_{1}^{(j)} = \lambda u'\frac{d}{du'}(C^{(j-1)}) (C^{(j-1)})^{-1},\]
and the convergence happens inside $\lambda^{2-n}\Mat_n(\cO^{\mathrm{rig}}_R)$.

As in \cite[Theorem 5.6]{LLLM} we can write
\[p^{n-1}\lambda^{n-2}N_{\infty}^{(j)}= (p\lambda)^{n-1} u'\frac{d}{du'}(C^{(j-1)}) (C^{(j-1)})^{-1}+\sum_{i=1}^{\infty} X_i^{(j)}\]
where
\[X_i^{(j)} := \frac{\phz^{i+1}(\lambda)^{n-1}}{p^{i(n-2)}} \left(\prod_{k=0}^{i-1}\phz^{k}(C^{(j-k-1)})\right) \phz^i\left(u' \frac{d}{du'} C^{(j-i -1)} \right) \left(\prod_{k=i}^{0}\phz^{k} \left((v+p)^{n-2} C^{(j-k-1), *} \right) \right).\]
If $z \in \mathbb{Z}^n$, we use the shorthand $\mathrm{Diag}(z)$ to denote diagonal matrix with entries $z_1, z_2, \ldots, z_n$.   Also for $M, N \in M_n(R[\![v]\!]), [M, N] := MN - NM$.  By ``removing the descent data'' as in \emph{loc.~cit.} (see \eqref{def:adj} for notation), we obtain
\[
p^{n-1} \Ad \big((s'_{\mathrm{or}, j})^{-1} (u')^{-\mathbf{a}_{(s, \mu)}^{\prime \, (j)}} \big) (\lambda^{n-2} N_{\infty}^{(j)})=
-\phz(\lambda)^{n-1}P_N(A^{(j-1)})+\sum_{i= 1}^{\infty}\phz^{i+1}(\lambda)^{n-1}Z_i^{(j)},
\]
where (cf.~ \cite[Lemma 5.4]{LLLM}) 
\[P_{N}(A^{(j-1)}) \defeq   \left(- e'v \frac{d}{dv} A^{(j-1)} - [\mathrm{Diag}((s'_{\mathrm{or}, j})^{-1}(\mathbf{a}_{(s, \mu)}^{\prime \, (j)})), A^{(j-1)}] \right)  (v+p)^{n-1}(A^{(j-1)})^{-1}\]
and
\[Z^{(j)}_i = \Ad \big((s'_{\mathrm{or}, j})^{-1} (u')^{-\mathbf{a}_{(s, \mu)}^{\prime \, (j)}} \big) \left(\frac{1}{\phz^{i+1}(\lambda)^{n-1}} X_i^{(j)} \right).
\]
Now exactly as in the last part of the proof of \cite[Theorem 5.6]{LLLM}, using that $\tau'$ is $(2n-1)$-generic, we get $Z^{(j)}_i\in \frac{v^{(2n-2)p^{i-1}}}{p^{i(n-2)}}M_n(R[\![v]\!])$ for $i>1$ and $Z^{(j)}_1\in \frac{v^{2n-1}}{p^{n-2}}M_n(R[\![v]\!])$. This fact together with a simple computation with derivatives shows that $M^{(j)}:=\frac{1}{\phz(\lambda)^{n-1}}\sum_{i= 1}^{\infty}\phz^{i+1}(\lambda)^{n-1}Z_i^{(j)}$ satisfies $(\frac{d}{dv})^{t} M^{(j)}|_{v=-p}\in p^{2n-1-(n-2)-t}M_n(R)$ and $(\frac{d}{dv})^{t} v^{-1}M^{(j)}|_{v=-p}\in p^{2n-1-(n-2)-t-1}M_n(R)$ for $0\leq t\leq n-3$.

As in the proof of \cite[Proposition 5.3]{LLLM}, the monodromy condition is equivalent to $\lambda^{n-2}N^{(j)}_\infty$ vanishing to order $n-2$ at $u'=(-p)^{\frac{1}{e'}}$. Thus the upshot of the above discussion is that the monodromy condition is equivalent to $-P_N(A^{(j-1)})+M^{(j)}$ vanishing to order $n-2$ at $v=-p$ in $M_n(\cO^{\mathrm{rig}}_{R})$. Note that this condition is preserved under multiplication by any power of $v$ and can be expressed as the vanishing of all derivatives $(\frac{d}{dv})^t$ at $v=-p$ for $0\leq t\leq n-3$ .

We now recall Proposition \ref{prop:univfhlift} which gives the decomposition
\[A^{(j)}=D_j(v+p)^{w_j(\eta_0)}U^{(j)}\]
and $R=R^{\tau, \overline{\beta}}_{\overline{\fM}}$ is the formal power series ring over $\cO$ in the variables $X_{\alpha}^{(j),l}$ (where $0\leq j<f$, $\alpha<0$ is a negative root of $\GL_n$ and $0\leq l<-\langle \eta_0,\alpha^{\vee} \rangle)$ and the variables $c_{ii}^{(j)}$ (where $1\leq i\leq n$, $0\leq j<f$), which (up to a translation in the case of $c_{ii}^{(j)}$) give the coefficients of the entries of $U^{(j)}$ and $D_j$ respectively.

Substituting the above expression, we get 
\begin{align*}
&P_N(A^{(j-1)})= -(v+p)^{n-1}(e'v\frac{d}{dv}(v+p)^{w_{j-1}(\eta_0)})(v+p)^{-w_{j-1}(\eta_0)}\\
&-(v+p)^{n-1}\Ad(D_{j-1}(v+p)^{w_{j-1}(\eta_0)})\Big(\Big(e'v\frac{d}{dv}U^{(j-1)}+\left[\mathrm{Diag}((s'_{\mathrm{or}, j})^{-1}(\mathbf{a}_{(s, \mu)}^{\prime \, (j)})), U^{(j-1)}\right]\Big)(U^{(j-1)})^{-1}\Big).
\end{align*}

To understand the second term, recall that also from Proposition \ref{prop:univfhlift}, for each negative root $\alpha=\eps_i-\eps_k$ of $\GL_n$, the $\alpha$-th entry of $w_{j-1}^{-1}U^{(j-1)}w_{j-1}$ is given by 
\[v^{\delta_{w_j(\alpha)<0}} f^{(j-1)}_{\alpha}(v)=v^{\delta_{w_{j-1}(\alpha)<0}}\sum_{l=0}^{-\langle \eta_0,\alpha^{\vee} \rangle-1}X^{(j-1),l}_{\alpha}v^l.\]
Thus, the $\alpha$-th entry of $\Ad((D_{j-1}w_{j-1})^{-1})(P_N(A^{(j-1)}))$ is of the form
\begin{align*}
-&(v+p)^{n-1+\langle \eta_0,\alpha^{\vee} \rangle}\big( e'v\frac{d}{dv}+\langle (s'_{\mathrm{or}, j})^{-1}(\mathbf{a}_{(s, \mu)}^{\prime \, (j)}),w_{j-1}\alpha^{\vee}\rangle )(v^{\delta_{w_{j-1}(\alpha)<0}}f^{(j-1)}_{\alpha})+\ldots\big)\\
=-&(v+p)^{n-1+\langle \eta_0,\alpha^{\vee} \rangle}\big( \sum_{l=0}^{-\langle \eta_0,\alpha^{\vee} \rangle-1}(e'(l+\delta_{w_{j-1}(\alpha)<0})
+\langle (s'_{\mathrm{or}, j})^{-1}(\mathbf{a}_{(s, \mu)}^{\prime \, (j)}),w_{j-1}\alpha^{\vee}\rangle)  X^{(j-1),l}_{\alpha}v^{l+\delta_{w_{j-1}(\alpha)<0}}+\ldots\big)
\end{align*}  
where the ellipsis in the first expression is an $R$-linear combination of terms of the form 
\[e'\frac{d}{dv} \left(v^{\delta_{w_{j-1}(\alpha_0)<0}}f^{(j-1)}_{\alpha_0} \right)\prod_{i\neq 0} v^{\delta_{w_{j-1}(\alpha_i)<0}}f^{(j-1)}_{\alpha_i}\] 
and
\[\prod_i v^{\delta_{w_{j-1}(\alpha_i)<0}}f^{(j-1)}_{\alpha_i},\]
where $\alpha_i$ are negative roots of $\GL_n$ with $\alpha=\sum_i \alpha_i$ and the sum has at least two terms. 
In particular, the remaining terms are polynomials divisible by $v^{\delta_{w_{j-1}(\alpha)<0}}$ and whose coefficients only involve $X^{(j-1)}_{\alpha'}$ for roots $\alpha'$ strictly larger than $\alpha$.

Thus, the monodromy condition on the $\alpha$-th entry of $v^{-\delta_{w_{j-1}(\alpha)<0}}\Ad(w_{j-1}^{-1}(D_{j-1})^{-1})(-P_N(A^{(j-1)})+M^{(j)})$ has the form
\begin{align*}
 (&\frac{d}{dv})^t\big((v+p)^{n-1+\langle \eta_0,\alpha^{\vee} \rangle}\big( \sum_{l=0}^{-\langle \eta_0,\alpha^{\vee} \rangle-1}(e'(l+\delta_{w_{j-1}(\alpha)<0})
+\langle (s'_{\mathrm{or}, j})^{-1}(\mathbf{a}_{(s, \mu)}^{\prime \, (j)}),w_{j-1}\alpha^{\vee}\rangle)  X^{(j-1),l}_{\alpha}v^{l}\big)\big)|_{v=-p}\\
=&O((X_{\alpha'}^{(j-1),l})_{0>\alpha'>\alpha,l})+O(p^{2n-1-(n-2)-t-1})
\end{align*}
for $0\leq t\leq n-3$, and where the right-hand side only involves variables indexed by strictly larger roots and a term in $R$ that is divisible by $p^{2n-1-(n-2)-t-1}$, and hence is divisible by $p$.

Thus, the above system of equation holds in the quotient $R^{\tau, \overline{\beta}, \nabla}_{\overline{\fM}}/\varpi$ of $R$, where it implies (noting $e'=-1$ in $\F$)
\begin{align*}  
&\left(\frac{d}{dv} \right)^t\left( \sum_{l=0}^{-\langle \eta_0,\alpha^{\vee} \rangle-1}(-l-\delta_{w_{j-1}(\alpha)<0}+\langle (s'_{\mathrm{or}, j})^{-1}(\mathbf{a}_{(s, \mu)}^{\prime \, (j)}),w_{j-1} \alpha^{\vee}\rangle)  X^{(j-1),l}_{\alpha}v^{l}\big)\right) \mid_{v=0}\\
=\qquad&O((X_{\alpha'}^{(j-1),l})_{0>\alpha'>\alpha,l})
\end{align*}
for $0\leq t \leq (n-3)-(n-1+\langle \eta_0,\alpha^{\vee} \rangle)=-\langle \eta_0,\alpha^{\vee} \rangle -2$. Since $\tau$ is generic and $p>n$, all the coefficients $-l-\delta_{w_{j-1}^{-1}(\alpha)<0}+\langle (s'_{\mathrm{or}, j})^{-1}(\mathbf{a}_{(s, \mu)}^{\prime \, (j)}),w_{j-1}\alpha^{\vee}\rangle$ as well as the constants introduced by taking derivatives are non-zero in $\F$,
and hence this system of equations solves $X^{(j-1),l}_{\alpha}$ for $l<-\langle \eta_0,\alpha^{\vee} \rangle-1$ in terms of 
variables indexed by strictly larger roots. It follows that $R^{\tau, \overline{\beta}, \nabla}_{\overline{\fM}}/\varpi$ is topologically generated by the top degree coefficients of $f^{(j)}_{\alpha}$ and the $c^{(j)}_{ii}-[\overline{c}^{(j)}_{ii}]$ for $0\leq j<f$, $1\leq i\leq n$, and negative roots $\alpha$ of $\GL_n$. Hence it is topologically generated by $n(n-1)f/2+nf=n(n+1)f/2$ elements.
\end{proof}



\begin{proof}[Proof of Theorem \ref{thm:smoothdef}]
We already know $R^{\eta,\tau}_{\rhobar}\neq 0$ by Corollary \ref{cor:existencelift}. We look at diagram \ref{defdiagram}. By Proposition \ref{prop:univlift}, $\Spf R^{\tau, \overline{\beta}, \Box}_{\overline{\fM}, \rhobar}\to \Spf R^{\eta, \tau}_{\rhobar}$ is a torsor for $(\widehat{\mathbb{G}}_m)^{nf}$, hence  $d\defeq \dim R^{\tau, \overline{\beta}, \Box}_{\overline{\fM}, \rhobar}= \dim R^{\eta, \tau}_{\rhobar}+nf=n(n-1)f/2+n^2+nf+1=n(n+1)f/2+n^2+1$.
On the other hand, $\Spf R^{\tau, \overline{\beta},\Box, \nabla}_{\overline{\fM}} \to \Spf R^{\tau, \overline{\beta}, \nabla}_{\overline{\fM}}$ is formally smooth of relative dimension $n^2$, hence Proposition \ref{prop:upperbounddefring} shows that there is a surjection $\cO[\![x_1,\ldots,x_{d-1}]\!]\onto  R^{\tau, \overline{\beta},\Box, \nabla}_{\overline{\fM}}$. From Proposition \ref{prop:factors}, we obtain a surjection $\cO[\![x_1,\ldots,x_{d-1}]\!]\onto R^{\tau, \overline{\beta}, \Box}_{\overline{\fM}, \rhobar}$. Since the quotient ring has dimension $d=\dim \cO[\![x_1,\ldots,x_{d-1}]\!]$, the kernel of this surjection must be trivial, hence the surjection is an isomorphism. It follows that $R^{\eta, \tau}_{\rhobar}$ is formally smooth over $\cO$.
\end{proof}

\begin{cor}\label{cor:pd}
With $\rhobar, \tau$ as in Theorem \ref{thm:smoothdef}, any potentially crystalline lift of $\rhobar$ of type $(\eta,\tau)$ is potentially diagonalizable.  
\end{cor}
\begin{proof} This follows from Theorem \ref{thm:smoothdef} and Corollary \ref{cor:existencelift}.
\end{proof}

\section{Main results}\label{sec:main}

In this section, we deduce our main results using \S \ref{sec:local}.
In \S \ref{sec:global}, we deduce weight elimination in an axiomatic context and in the context of definite unitary groups.
In \S \ref{sec:obv}, we use weight elimination and the change of weight techniques of \cite{BLGGT} to deduce modularity of obvious weights.
In \S \ref{sec:red}, we use the above results to classify congruences between RACSDC $\GL_n$-automorphic representations of trivial weight and generic tame type $\tau$ in residually tame cases and solve the lifting problem for residually tame Galois representations to potentially crystalline representations of type $(\eta,\tau)$.
We introduce combinatorial results on Serre weights and affine Weyl groups as needed.
The key theme in these combinatorial results is the close relationship between certain reduced factorizations of admissible elements and Jantzen's description of the Jordan--H\"older factors of types.

Recall the (nonstandard) definition of the dot action in Definition \ref{defn:dot}.
We write $\bun{W}_a^+ \subset \tld{\bun{W}}^+$ for the subsets of $\bun{W}_a\subset\tld{\bun{W}}$, respectively, which map $\bun{C}_0$ to a dominant alcove under this dot action.

Let $\un{\Omega} \subset \tld{\bun{W}}$ be the stabilizer of $\bun{C}_0$.
Then we have the decomposition $\tld{\bun{W}} = \bun{W}_a \rtimes \un{\Omega}$.
We extend the Coxeter length function $\ell$ on $\bun{W}_a$ to $\tld{\bun{W}}$ by setting $\ell(\tld{w}\delta) = \ell(\tld{w})$ if $\tld{w} \in \bun{W}_a$ and $\delta\in \un{\Omega}$. 
Recall that one can calculate lengths from minimal galleries (cf. \cite[\S 2]{HN02}).
We will use galleries in a fixed direction (cf. \cite[Definition 5.2]{HN02}), which are necessarily minimal by \cite[Lemma 5.3]{HN02}.

Recall the upper arrow ($\uparrow$) ordering on $p$-alcoves (cf. \cite[\S II.6.5]{RAGS}), and extended to $\tld{\bun{W}}$ by writing $\tld{w}_1\uparrow \tld{w}_2$ if $\tld{w}_1 \cdot \bun{C}_0\uparrow \tld{w}_2 \cdot \bun{C}_0$ and $\bun{W}_a \tld{w}_1 = \bun{W}_a \tld{w}_2$ for $\tld{w}_1$ and $\tld{w}_2\in \tld{\bun{W}}$ (and elements of different right $\bun{W}_a$-cosets are incomparable).
We also use $\uparrow$ to denote ordering on $X^*(\bun{T})$ defined in \cite[\S II.6.4]{RAGS}.
Recall from \S \ref{sec:awg} the Bruhat ordering $\leq$ on $\bun{W}_a$ defined by the \emph{dominant} base alcove.
As with the upper arrow ordering, we extend this to a partial ordering on $\tld{\bun{W}}$ by setting $\tld{w}_1\delta \leq \tld{w}_2\delta$ if $\tld{w}_1$ and $\tld{w}_2\in \bun{W}_a$, $\tld{w}_1\leq \tld{w}_2$, and $\delta \in \un{\Omega}$ (and elements of different right $\bun{W}_a$-cosets are incomparable).

Let $\tld{w}_h = w_0 t_{-\eta} \in \tld{\bun{W}}$.
Note that $\tld{w}_h\cdot \bun{C}_0$ is the highest $p$-restricted alcove and $\tld{w}_h \cdot \lambda = \mathcal{R}(\lambda)$ from Definition \ref{defn:SWC}.

\subsection{Combinatorics of weights and types} \label{sec:we}

In this section, we deduce the key combinatorial results, especially Corollary \ref{cor:we}.
We will use the following theorem of Wang (see \cite[Theorem 4.3]{Wang}) without comment.

\begin{thm}\label{thm:pocomp}
If $\tld{w}_1$ and $\tld{w}_2 \in \tld{\bun{W}}^+$, then $\tld{w}_1\leq \tld{w}_2$ if and only if $\tld{w}_1\uparrow \tld{w}_2$.
\end{thm}

We also use the following proposition without comment.

\begin{prop}
If $\tld{w}_1$ and $\tld{w}_2 \in \tld{\bun{W}}$, then $\tld{w}_1\uparrow \tld{w}_2$ if and only if $\tld{w}_h \tld{w}_2 \uparrow \tld{w}_h \tld{w}_1$ if and only if $\tld{w}_h^{-1} \tld{w}_2\uparrow \tld{w}_h^{-1}\tld{w}_1$.
\end{prop}
\begin{proof}
From the definition of the up ordering, it is clear that $w_0$ reverses and translation preserves the ordering.
The proposition now follows from the definition of $\tld{w}_h$.
\end{proof}



\begin{prop}\label{prop:JH}
Suppose that $\mu$ is $2n$-deep in $\bun{C}_0$ and $\lambda$ is a $p$-restricted weight.
Then $F(\lambda) \in \JH(\ovl{R}_s(\mu + \eta))$ if and only if there exists $\tld{w} = w t_{\nu} \in \tld{\bun{W}}^+$ such that 
\begin{equation}\label{eqn:upord}
\tld{w}\cdot(\mu-s\pi\nu) \uparrow \tld{w}_h \cdot \lambda \quad\textrm{and}\quad \tld{w}\cdot \bun{C}_0 \uparrow \tld{w}_h\cdot \bun{C}_0.
\end{equation}
\end{prop}
\begin{proof}
We use \cite[Proposition 10.1.8]{GHS}.
(Note that by the depth assumption, the proof of \cite[Proposition 10.1.2]{GHS} based on \cite[Satz 4.3]{Jantzen} applies.) 
The Proposition cited shows that $F(\lambda) \in \JH(\ovl{R}_s(\mu + \eta))$ if and only if there is $\nu$ such that 
\begin{equation}\label{eqn:weylorbit}
\sigma t_{\nu}\cdot(\mu-s\pi\nu) \uparrow \tld{w}_h \cdot \lambda \qquad \text{for all } \sigma \in W(\bun{G}).
\end{equation}
It suffices to show that the existence of $\nu$ satisfying (\ref{eqn:weylorbit}) and $\tld{w}\in \tld{\bun{W}}^+$ satisfying (\ref{eqn:upord}) are equivalent.

We begin with the ``backwards" implication, for which the following remark is useful.
\begin{rmk} \label{rmk:bounderror} If $\tld{w} = w t_{\nu}$ satisfies $\tld{w}\cdot \bun{C}_0 \uparrow \tld{w}_h\cdot \bun{C}_0$ then $\nu$ lies in the convex hull of the Weyl orbit of $\eta$, and hence $\max_{\alpha^\vee} \{| \langle \nu,\alpha^\vee \rangle|\}\leq n-1$.
\end{rmk}
\noindent Suppose that $\tld{w} = wt_\nu$ satisfies (\ref{eqn:upord}).
With the depth assumption on $\mu$, the above remark implies that $\mu-s\pi\nu$ is in $\bun{C}_0$ so that $wt_\nu \cdot (\mu-s\pi\nu)$ is the unique dominant element of the $W(\bun{G})$-(dot) orbit of $\mu-s\pi\nu+p\nu$.
This implies that $\nu$ satisfies (\ref{eqn:weylorbit}).

For the ``forwards" implication, suppose that $\nu$ satisfies (\ref{eqn:weylorbit}) and take $\tld{w}$ to be the unique element $wt_\nu \in \tld{\bun{W}}^+$ with $w \in W(\bun{G})$.
Then $\tld{w}\cdot(\mu-s\pi\nu) \uparrow \tld{w}_h \cdot \lambda$ by assumption and it suffices to show that $\tld{w}\cdot \bun{C}_0 \uparrow \tld{w}_h\cdot \bun{C}_0$.
We claim that $\nu$ satisfying (\ref{eqn:weylorbit}) must automatically satisfy $|\langle \nu,\alpha^\vee \rangle|\leq n-1$ for all $\alpha^\vee\in \bun{R}^\vee$. 
Admitting this claim for the moment, we again have that $\tld{w}\cdot(\mu-s\pi\nu)$ is in alcove $\tld{w}\cdot \bun{C}_0$, so that $\tld{w}\cdot(\mu-s\pi\nu) \uparrow \tld{w}_h \cdot \lambda$ implies that $\tld{w}\cdot \bun{C}_0 \uparrow \tld{w}_h\cdot \bun{C}_0$.

Going back to our claim, (\ref{eqn:weylorbit}) implies that $\mu+p\nu-s\pi\nu+\eta$ is in the convex hull of the Weyl orbit of $\tld{w}_h \cdot \lambda+\eta$. The same argument as in the proof of Lemma \ref{bound} shows that  
\[\max\nolimits_{\alpha^{\vee}} \{ |\langle \mu+p\nu-s\pi\nu+\eta, \alpha^{\vee}\rangle|\}\leq  \max\nolimits_{\alpha^{\vee}} \{ |\langle \tld{w}_h \cdot \lambda+\eta, \alpha^{\vee}\rangle|\}\leq p(n-1).\]
The same argument as in the proof of Corollary \ref{cor:obvweight}, using $\mu$ is $2n$-deep in $\bun{C}_0$, shows that if $M= \max_{\alpha^{\vee}} \{ |\langle\nu, \alpha^{\vee}\rangle|\}$ then $(p-1)M<p(n-1)+p-2n=(p-2)n$, and thus $M\leq n-1$ as desired.
%
\end{proof}


We will often fix 
\begin{enumerate}[label={(\bfseries P\arabic*)}]
\item a generic semisimple Galois representation $\rhobar:G_K \ra \GL_n(\F)$; \label{item:rhobartau1}
\item a pair $(s_{\rhobar}, \mu_{\rhobar})$ such that $\rhobar|_{I_K} \cong \ovl{\tau}(s_{\rhobar},\mu_{\rhobar}+\eta)$ with $\mu_{\rhobar}$ 
in $\bun{C}_0$; and \label{item:rhobartau2}
\item A lowest alcove presentation $(s,\mu-\eta)$ of a tame inertial type $\tau\cong\tau(s,\mu):I_K \ra \GL_n(\cO)$ such that $\mu_{\rhobar}-\mu \in \un{\Lambda}_R$. Such a presentation is called compatible with $(s_{\rhobar}, \mu_{\rhobar})$. \label{item:rhobartau3}
\end{enumerate}
Note under \ref{item:rhobartau1} and \ref{item:rhobartau2}, $\mu_{\rhobar}$ is always $(3n-1)$-deep in $\bun{C}_0$ by Proposition \ref{prop:alcoveC0}.
\begin{lemma}\label{lemma:distincttype}
Suppose that $\tau(s,\mu) \cong \tau(s',\mu')$ is $1$-generic, and $\mu-\eta$ and $\mu'-\eta$ are both  in $\bun{C}_0$, and $\mu-\mu'\in \un{\Lambda}_R$. 
Then we have $(s,\mu) = (s',\mu')$.
\end{lemma}
\begin{proof}
By  Proposition \ref{prop:alcoveC0}, since $(s',\mu'-\eta)$ and $(s,\mu-\eta)$ are two lowest alcove presentation of $\tau(s,\mu)$, we have $(s,\mu) = ^{(\nu, \sigma)}(s',\mu')$ with $t_{\nu} \sigma \in \un{\Omega}$. 
Since $\mu - \mu' \in \un{\Lambda}_R$, we also must have that $(p-\pi)\nu|_{\bun{Z}} = 0$, or equivalently that $\nu|_{\bun{Z}} = 0$.
Combining these facts, we have that $t_{\nu} \sigma$ is the identity.
\end{proof}

\begin{defn} \label{defn:relshape} Fix \ref{item:rhobartau1}-\ref{item:rhobartau3} as above. 
Then, we define $w = s^{-1} s_{\rhobar}$, $\nu = s^{-1}(\mu_{\rhobar}+\eta-\mu)$, and 
$\tld{w}^*(\rhobar,\tau) \defeq t_\nu w$.  
\end{defn}

\begin{rmk}\label{rmk:changerelshape}
Definition \ref{defn:relshape} a priori depends on the choice of $(s_{\rhobar},\mu_{\rhobar})$ and the compatible presentation $(s,\mu-\eta)$ of the type $\tau$.
By Lemma \ref{lemma:distincttype}, if $\tau$ is $1$-generic, there is at most one compatible presentation, so that Definition \ref{defn:relshape} depends only on the choice of $(s_{\rhobar},\mu_{\rhobar})$. Furthermore if a compatible presentation $(s,\mu-\eta)$ exists for one choice of $(s_{\rhobar},\mu_{\rhobar})$, then a compatible presentation exists for all other choices, and changing the choice of $(s_{\rhobar},\mu_{\rhobar})$ conjugates $\tld{w}^*(\rhobar,\tau)$ by an element of $\un{\Omega}$.
\end{rmk}

\begin{prop}\label{rmk:equalshape} Let $\tau$ be a generic type with lowest alcove presentation $(s, \mu- \eta)$.  Assume there exists $\overline{\fM} \in Y^{\eta, \tau}(\F)$ as in Definition \ref{defn:shaperhobar} such that $T^*_{\mathrm{dd}}(\overline{\fM}) \cong \rhobar|_{G_{K_{\infty}}}$ where $\rhobar$ is moreover semisimple. Then there is a pair $(s_{\rhobar}, \mu_{\rhobar})$ as in \ref{item:rhobartau2} such that \ref{item:rhobartau3} holds and $\tld{w}(\rhobar,\tau)^* = \tld{w}^*(\rhobar,\tau)$.
\end{prop}
\begin{proof}
By Theorem \ref{fixedpoints} combined with Corollary \ref{cor:phiandKisin}, we have 
\[
\rhobar|_{I_K} \cong \ovl{\tau}(w, \nu + \eta)
\]
where $\tld{w}(\rhobar, \tau) s^* t_{\mu^* - \eta^*} = w^* t_{\nu^*}$.   Thus, $\tld{w}(\rhobar, \tau)^* = s^{-1} t_{\nu - \mu + \eta} w$.   Since $\tld{w}(\rhobar, \tau)^* \in \Adm(\eta)$, $\nu - \mu \in \un{\Lambda}_R$.
Moreover $\nu-\mu+\eta$ is in the convex hull of $\bun{W}\eta$ by Remark \ref{rmk:bounderror}, so that $\nu$ is $1$-deep in $\bun{C}_0$ since $\mu-\eta$ is $n$-deep in $\bun{C}_0$.
Then, letting $\mu_{\rhobar} = \nu$ and $s_{\rhobar} = w$, \ref{item:rhobartau2} and \ref{item:rhobartau3} are satisfied. Comparing, we see that $\tld{w}(\rhobar, \tau)^*$ agrees with Definition \ref{defn:relshape}.
\end{proof}


\begin{lemma}\label{lemma:gallery}
Suppose that $\tld{w}_1$ and $\tld{w}_2\in \tld{\bun{W}}^+$.
Then $\ell(\tld{w}_2^{-1}w_0\tld{w}_1) = \ell(\tld{w}_2^{-1})+\ell(w_0)+\ell(\tld{w}_1)$.
\end{lemma}
\begin{proof}
The length $\ell(\tld{w})$ is the length of a minimal gallery between  $\tld{w}_2^{-1}w_0\tld{w}_1 \cdot \bun{C}_0$ and $\bun{C}_0$, which is the length of a minimal gallery between $w_0\tld{w}_1 \cdot \bun{C}_0$ and $\tld{w}_2 \cdot \bun{C}_0$.
Such a gallery can be taken to start with a gallery from $w_0\tld{w}_1 \cdot \bun{C}_0$ to $w_0 \cdot \bun{C}_0$, then to $\bun{C}_0$, then to $\tld{w}_2 \cdot \bun{C}_0$, all in the dominant direction.
This decomposition of a minimal gallery in the dominant direction gives the desired equality. \end{proof}

\begin{lemma} \label{lemma:we} 
Fix \ref{item:rhobartau1}-\ref{item:rhobartau2} as above. 
Suppose that $\lambda \in X_1(\bun{T})$ is $3n$-deep in its alcove and such that
for all $s \in W(\bun{G})$, $\tau(s, \tld{w}_h\cdot \lambda+\eta)$ is a tame inertial type which admits a compatible presentation as in \ref{item:rhobartau3} and $\tld{w}^*(\rhobar,\tau(s, \tld{w}_h\cdot \lambda+\eta)) \in \mathrm{Adm}(\eta)$.
Then $F(\lambda) \in W^{?}(\rhobar)$.
\end{lemma}
\begin{proof}
Let $s\in W(\bun{G})$ and $\tau = \tau(s,\tld{w}_h\cdot \lambda+\eta)$.
If $\tld{w}_h \cdot \lambda \in \tld{\sigma}\cdot \bun{C}_0$ for $\tld{\sigma} = \sigma t_\nu$, then $\tau$ has a lowest alcove presentation $(\sigma^{-1}s\pi(\sigma),\tld{\sigma}^{-1}\tld{w}_h\cdot \lambda + \sigma^{-1}s\pi(\sigma\nu))$ with $\tld{\sigma}^{-1}\tld{w}_h\cdot \lambda + \sigma^{-1}s\pi(\sigma\nu)$ $2n$-deep in $\bun{C}_0$ by the first part of Proposition \ref{prop:relabeltype}.
By assumption, $\tau\cong \tau(s',\lambda')$ has a lowest alcove presentation $(s',\lambda'-\eta)$ as in \ref{item:rhobartau3}.
Combining these with the second part of Proposition \ref{prop:relabeltype}, we have that $s'=w_2^{-1}s\pi(w_2)$ and $\lambda' = \tld{w}_2^{-1}\tld{w}_h \cdot \lambda+s'\pi\nu_2+\eta$ for some $\tld{w}_2 = w_2t_{\nu_2}\in \tld{\bun{W}}^+$.
If we let $\lambda^{(0)}$ be $\tld{w}_2^{-1}\tld{w}_h \cdot \lambda$, then $\lambda^{(0)}+s'\pi\nu_2$ is in $\bun{C}_0$ since $\lambda$, and thus, $\lambda^{(0)}$ is $3n$-deep in its alcove and $\max_{\alpha^\vee} \{ |\langle \nu_2,\alpha^\vee \rangle|\}\leq n-1$ by Remark \ref{rmk:bounderror}.

Let $\tld{w} = \tld{w}^*(\rhobar,\tau)$.
Then by assumption, we have that $\mu_{\rhobar} - \lambda^{(0)}-\pi\nu_2  -\eta\in \un{\Lambda}_R$. Note that this condition and the condition that $\tld{w}_h \cdot \lambda\in \tld{w}_2\cdot \bun{C}_0$ uniquely determines $\tld{w}_2$, hence $\tld{w}_2$ does not change when $s$ changes. Thus as $s$ runs over $W(\bun{G})$, $s'$ runs over all of $W(\bun{G})$.

By definition, we have
\[
\tld{w} = s'^{-1} t_{\mu_{\rhobar} - \lambda^{(0)}-s' \pi \nu_2} s_{\rhobar} = s'^{-1} t_{\delta-s' \pi \nu_2} s_{\rhobar} = t_{- \pi \nu_2} s'^{-1} t_{\delta}s_{\rhobar} \in \mathrm{Adm}(\eta),
\]
where $\delta = \mu_{\rhobar} - \lambda^{(0)}$.

Let $\sigma \in W(\bun{G})$ (unrelated to the use of $\sigma$ in the first paragraph) be the unique element such that $\sigma t_{s_{\rhobar}^{-1}\delta} \in \tld{\bun{W}}^+$.
Note that the $\pi$-action on $\tld{\bun{W}}$ preserves $\tld{\bun{W}}^+$.
We now take $s' = s_{\rhobar}\sigma^{-1}w_0 \pi(w_2)$, so that $\tld{w} = \pi(t_{-\nu_2}w_2^{-1}) w_0\sigma t_{s_{\rhobar}^{-1}\delta} = \pi(\tld{w}_2^{-1}) w_0\sigma t_{s_{\rhobar}^{-1}\delta}$.
Since $\sigma t_{s_{\rhobar}^{-1}\delta}\in \tld{\bun{W}}^+$, we have that $w_0 \pi (\tld{w}_2) \tld{w}\in \tld{\bun{W}}^+$.
Thus, $w_0 \pi (w_2) \tld{w} = t_{-\pi (w_0 w_2 \nu_2)} w_0 \pi(\tld{w}_2) \tld{w}$ is in $\tld{\bun{W}}^+$ since $-\pi (w_0 w_2 \nu_2)$ is a dominant weight. 
Note that there is a gallery in the dominant direction from $w_0\pi (w_2)(\bun{C}_0)$ to $w_0\pi (w_2) \tld{w}(\bun{C}_0)$ passing through $\bun{C}_0$, and hence a gallery in the $(w_0 \pi (w_2))^{-1}$-direction from $\bun{C}_0$ to $\tld{w}(\bun{C}_0)$.
By \cite[Corollary 4.4]{HH}, $\tld{w} \leq t_{(w_0 \pi (w_2))^{-1}\eta}$. (The reference \cite{HH} uses the Bruhat order defined with respect to the \emph{anti-dominant} base alcove. However, by applying $w_0$-conjugation which interchanges the two Bruhat orders, the cited corollary holds for the Bruhat order defined with respect to the \emph{dominant} base alcove.)

Note that
\[t_{(w_0 \pi (w_2))^{-1}\eta} = (w_0 \pi (w_2))^{-1}t_\eta(w_0 \pi (w_2))=(\tld{w}_h^{-1} \pi (\tld{w}_2))^{-1}\tld{w}_h^{-1}w_0(\tld{w}_h^{-1} \pi (\tld{w}_2)) =  \pi (\tld{w}_2)^{-1} w_0(\tld{w}_h^{-1}\pi(\tld{w}_2)).\]
Since $\ell(\tld{w}) = \ell(\pi(\tld{w}_2)^{-1})+\ell(w_0)+\ell(\sigma t_{s_{\rhobar}^{-1}\delta})$ by Lemma \ref{lemma:gallery}, we conclude that
$\sigma t_{s_{\rhobar}^{-1}\delta}\leq \tld{w}_h^{-1}\pi(\tld{w}_2)$ by standard facts about Coxeter groups.
Since both sides of this inequality are elements of $\tld{\bun{W}}^+$, we conclude that $\sigma t_{s_{\rhobar}^{-1}\delta}\uparrow \tld{w}_h^{-1}\pi(\tld{w}_2)$.
In other words, we have 
\[\pi^{-1}(\sigma) t_{\pi^{-1}(s_{\rhobar}^{-1}\delta)} \cdot (\mu_{\rhobar}-\delta) \uparrow \lambda.\]
By Proposition \ref{prop:JH}, $F(\tld{w}_h^{-1}\lambda) \in \JH(\ovl{R}_{s_{\rhobar}}(\mu_{\rhobar}+\eta))$ or equivalently, $F(\lambda)\in W^?(\rhobar)$.
\end{proof}

\begin{rmk} Regarding the hypotheses in Lemma \ref{lemma:we}, if $\lambda$ is $n$-deep in a $p$-restricted alcove, then by Proposition \ref{prop:lambdagen}, $\tau = \tau(s, \tld{w}_h \cdot \lambda + \eta)$ is $1$-generic. By Proposition \ref{prop:alcoveC0}, all lowest alcove presentations $(s',\mu'-\eta)$ of $\tau$ are of the form
\[(s',\mu')=^{(\nu, \sigma)} ( \tld{w}_h \cdot \lambda + \eta)\]
with $(t_\nu\sigma \tld{w}_h) \cdot \lambda\in \bun{C}_0$. The condition that there is one compatible lowest alcove presentation as in \ref{item:rhobartau3} above is equivalent to $(\tld{w}_h \cdot \lambda+\eta-\mu_{\rhobar})|_{\bun{Z}}\in (p-\pi)X^*(\bun{Z})$. This condition is the condition that the central character of $F(\lambda)$ agrees with the central character of any element, or equivalently all elements, of $W^?(\rhobar)$.

\end{rmk}

\begin{cor}\label{cor:we}
Suppose that $\rhobar:G_K\ra\GL_n(\F)$ is a $5n$-generic Galois representation. 
If $\lambda$ is in $X_1(\bun{T})$ and $(2n-1)$-deep in its alcove and $F(\lambda)$ is not in $W^{?}(\rhobar^{\mathrm{ss}})$, then there exists an $n$-generic tame inertial type $\tau$ such that $F(\lambda) \in \JH(\overline{\sigma}(\tau))$ and $\rhobar$ does not have a potentially crystalline lift of type $(\eta, \tau)$.  
\end{cor}
\begin{proof}
Suppose that $\lambda$ is not $3n$-deep.
By Proposition \ref{prop:lambdagen} and its proof, the tame type $\tau\defeq \tau(1,\tld{w}_h\cdot \lambda+\eta)$ is $n$-generic, but not $4n$-generic (by comparing $(1,\tld{w}_h\cdot \lambda+\eta)$ with a lowest alcove presentation). By Corollary \ref{explicitill} and \cite[Lemma 2.5]{herzig-satake}, 
 $F(\lambda) \in \JH(\ovl{\sigma}(\tau(1,\tld{w}_h\cdot \lambda+\eta)))$. 
Then $\rhobar$ is not the reduction of a potentially crystalline representation of type $(\eta,\tau)$ by Proposition \ref{telim2}.

Now suppose that $\lambda$ is $3n$-deep in its alcove.
Suppose that $\rhobar$ is the reduction a potentially crystalline representation of type $(\eta,\tau)$ for all $\tau$ such that $F(\lambda) \in \JH(\overline{\sigma}(\tau))$.
We claim that $\tau$ can be taken to be $\tau(s, \tld{w}_h\cdot\lambda+\eta)$ for any $s\in \bun{W}$.
Suppose that $\tld{w} = wt_\nu$ and that $\tld{w}_h \cdot \lambda\in \tld{w}\cdot \bun{C}_0$.
Let $\mu$ be $\tld{w}^{-1} \tld{w}_h \cdot \lambda + w^{-1}s\pi(w\nu)$.
Then $\tau(w^{-1}s\pi(w),\mu+\eta) \cong \tau(s,\tld{w}_h\cdot \lambda+\eta)$ by Proposition \ref{prop:relabeltype}, and $\mu$ is $2n+1$-deep in $\bun{C}_0$ by Remark \ref{rmk:bounderror}.
Then $\ovl{\sigma}(\tau(s, \tld{w}_h\cdot\lambda+\eta)) = \ovl{R}_{w^{-1}s\pi(w)}(\mu+\eta)$ contains $F(\lambda)$ as a Jordan--H\"older factor by Proposition \ref{prop:JH} ($\tld{w}\cdot(\mu - w^{-1}s\pi(w\nu)) = \tld{w}_h\cdot \lambda$).
Thus for all $s \in W(\bun{G})$, $\rhobar$ is the reduction of a potentially crystalline representation of type $(\eta,\tau(s, \tld{w}_h\cdot\lambda+\eta))$.

By Theorem \ref{thm:admcrit} and Proposition \ref{rmk:equalshape}, for each $s\in W(\bun{G})$, there is a pair $(s_{\rhobar}, \mu_{\rhobar})$ for $\rhobar^{\mathrm{ss}}$ as in \ref{item:rhobartau2} such that $\tld{w}^*(\rhobar^{\mathrm{ss}},\tau(s, \tld{w}_h\cdot\lambda+\eta))\in \Adm(\eta)$ (which implies \ref{item:rhobartau3} holds). 
By Remark \ref{rmk:changerelshape}, these conditions hold for any choice of $(s_{\rhobar}, \mu_{\rhobar})$ satisfying \ref{item:rhobartau2}, since $\Adm(\eta)$ is stable under $\un{\Omega}$-conjugation. Thus we can use the same pair $(s_{\rhobar}, \mu_{\rhobar})$ for all choices of $s$.
Then by Lemma \ref{lemma:we}, we see that $F(\lambda)\in W^?(\rhobar^{\mathrm{ss}})$.
%
\end{proof}

\begin{rmk}
In the proof of Corollary \ref{cor:we}, we used Proposition \ref{prop:JH} to show that $F(\lambda) \in \JH(\ovl{\sigma}(\tau(s,\tld{w}_h\cdot \lambda+\eta))$.
This weight is the reflection of an obvious weight or a \emph{diagonal prediction} in \cite[Definition 12.8]{herzig-thesis}, so the above membership is likely true with weaker genericity hypotheses on the Deligne--Lusztig representation.
Corollary \ref{cor:we} also probably holds with weaker genericity hypotheses.
\end{rmk}

\subsection{Weight elimination}\label{sec:global}

In this section, we deduce our main weight elimination result (Corollary \ref{cor:globalwe}).
We begin with an axiomatic setup for our method, and then proceed to the case of modular forms for definite unitary groups.

\subsubsection{Axiomatic setup}

We begin with an axiomatic setup for modular Serre weights.
This is related to the axiomatic setup of \cite[\S 4.2]{EGH}.
Let $\rhobar:G_K \ra \GL_n(\F)$ be a Galois representation. (We no longer assume \ref{item:rhobartau1}-\ref{item:rhobartau3}.)
We write $\F[\GL_n(k)]\textrm{-mod}$ for the category of finite $\F[\GL_n(k)]$-modules and $\mathrm{Vect}_{/\F}$ for the category of $\F$-modules (i.e.~vector spaces).

\begin{defn}\label{defn:cfunct}
We say that a functor $S:\F[\GL_n(k)]\textrm{-mod} \ra \mathrm{Vect}_{/\F}$ is an {\it arithmetic cohomology functor for} $\rhobar$ if 
\begin{itemize}
\item $S \neq 0$;
\item whenever $S(V) \neq 0$ for a Serre weight $V$ and $V\in \JH(\taubar)$ for a regular tame inertial type $\tau$, $\rhobar$ has a potentially crystalline lift of type $(\eta, \tau)$; and
\item whenever $S(V) \neq 0$ for a Serre weight $V$ and $V\in \JH(\ovl{R}_1(\mu))$, $\rhobar$ has a potentially semistable lift of type $(\eta, \tau(1,\mu))$.
\end{itemize}
\end{defn}

We now fix an arithmetic cohomology functor $S$ for $\rhobar$.

\begin{defn}\label{defn:W}
We say that $\rhobar$ is {\it modular of weight} $V$ if $S(V) \neq 0$.
Let $W(\rhobar)$ be the set of isomorphism classes of Serre weights for which $\rhobar$ is modular.
Let $W_{\mathrm{elim}}(\rhobar)$ be the subset of $W(\rhobar)$ consisting of isomorphism classes of Serre weights represented by $F(\lambda)$ where $\lambda$ is $2n-1$-deep in its alcove.
\end{defn}

\begin{thm}\label{thm:we}
If $\rhobar$ is $5n$-generic, then $W_{\mathrm{elim}}(\rhobar) \subset W^?(\rhobar^\mathrm{ss})$.
\end{thm}
\begin{proof}
If $\lambda$ is $(2n-1)$-deep in its alcove and $F(\lambda) \notin W^?(\rhobar^{\mathrm{ss}})$, then by Corollary \ref{cor:we} there is a regular ($n$-generic even) tame inertial type $\tau$ such that $F(\lambda) \in \JH(\overline{\sigma}(\tau))$ and $\rhobar$ does not have a potentially crystalline lift of type $(\eta, \tau)$. 
We conclude that $F(\lambda) \notin W(\rhobar)$.
\end{proof}

\cite[Theorem 8]{Enns} shows that for sufficiently generic $\rhobar$, $W(\rhobar)$ does not contain weights near the boundary of alcoves.
Combining this with Theorem \ref{thm:we}, we obtain an upper bound for $W(\rhobar)$ rather than $W_{\mathrm{elim}}(\rhobar)$.

\begin{cor}\label{cor:awe}
If $\rhobar$ is $(6n-2)$-generic, then $W(\rhobar) \subset W^?(\rhobar^\mathrm{ss})$.
\end{cor}
\begin{proof}
Suppose that $\lambda\in X_1(\un{T})$ is $0$-deep but not $(2n-1)$-deep in its alcove.
Let $\lambda^{(0)}$ be the unique weight in $\bun{C}_0$ linked to $\lambda$.
Then $\lambda^{(0)}$ is also not $(2n-1)$-deep in its alcove so that $\lambda^{(0)}+\eta$ is not $2n$-generic in the sense of \cite[Definition 2]{Enns}.
We conclude that $\lambda$ is not $(4n-2)$-generic in the sense of \cite[Definition 1]{Enns}. Note that even if $\lambda$ is not $0$-deep in its alcove, this last conclusion still holds.
Now $\rhobar$ is $(6n-2)$-generic in the sense of \cite[Definition 2]{Enns} by Remark \ref{rmk:comparison}. 
Then the proof of \cite[Theorem 8]{Enns} shows that $\rhobar$ does not have a potentially semistable lift of type $(\eta,\tau(1,\lambda))$, while $F(\lambda) \in \JH(\ovl{R}_1(\lambda))$ by \cite[Lemma 2.5]{herzig-satake}.
We conclude that $F(\lambda) \notin W(\rhobar)$.
\end{proof}

\subsubsection{Algebraic modular forms for unitary groups}\label{sec:algmf}

We closely follow the setup of \cite[\S 7.1]{EGH} (see also \cite[\S 4.1]{HLM}).
Let $F/\Q$ be a CM field and $F^+$ its maximal totally real subfield. Assume that $F^+\neq \Q$, and that all places of $F^+$ dividing $p$ split in $F$ and are unramified.
We write $c$ for the generator of $\Gal(F/F^+)$.
For $u\nmid \infty$ (resp.\ $v\nmid \infty$) a place of $F$ (resp.\ $F^+$) we denote by $k_u$ (resp.\ $k_v$) the residue field of $F_u$ (resp.\ $F^+_v$).

We let $G_{/F^+}$ be a reductive group, which is an outer form of $\GL_n$ which splits over~$F$. We assume that $G(F^{+}_{v})\cong U_n(\R)$ for all $v|\infty$.  
By the argument of \cite{CHT}, \S3.3, $G$ admits a model $\cG$ over $\cO_{F^+}$ such that $\cG \times \cO_{F^+_v}$ is reductive for all places $v$ of $F^+$ that
split in $F$.
For any such place $v$ of $F^+$ and $u|v$ of $F$ we get an isomorphism $\iota_u:\,G(F^+_v)\stackrel{\sim}{\longrightarrow}\GL_n(F_u)$
which restricts moreover to an isomorphism $\iota_u:\,\cG(\cO_{F^+_v})\stackrel{\sim}{\longrightarrow}\GL_n(\cO_{F_u})$.
Let $\Sigma$ be a finite set of places in $F^+$. 
If $U^\Sigma\leq G(\bA_{F^+}^{\infty,\Sigma})$ is a compact open subgroup, the space $S(U^\Sigma)$ of infinite level algebraic automorphic forms on $G$ is defined to be the set of continuous functions 
$f:\,G(F^{+})\backslash G(\bA^{\infty}_{F^{+}})/ U^\Sigma\rightarrow \F,$
where $\F$ is given the discrete topology.

We recall that the level $U\leq G(\bA_{F^+}^\infty)$ is said to be \emph{sufficiently small} if for all $t \in G(\bA^{\infty}_{F^+})$, the finite group $t^{-1} G(F^+) t \cap U$ is of order prime to $p$.
We say that $U^\Sigma$ is \emph{sufficiently small} if 
\[U = U^\Sigma \prod_{v\in \Sigma} \cG(\cO_{F^+_v})\]
is sufficiently small.
For a finite place $v$ of $F^+$ we say that $U$ is \emph{unramified} at $v$ if $v$ splits in $F$ and one has a decomposition $U=\cG(\cO_{F_v^+})U^{v}$ for some compact open subgroup $U^v\leq G(\bA^{\infty,v}_{F^+})$. 

Let $\cP_U$ denote the set consisting of finite places $u$ of $F$ such that $v\defeq u\vert_{F^+}$ is split in~$F$, $v\nmid p$, and $U$ is unramified at $v$.
If $\cP\subseteq \cP_U$ is a subset of finite complement that is closed under complex conjugation and disjoint from $\Sigma$, we write $\bT^{\cP}=\cO[T^{(i)}_u : u\in\cP,\,i\in\{0,1,\ldots,n\}]$ for the abstract Hecke algebra on $\cP$, where the Hecke operator $T_u^{(i)}$ acts on the space $S(U^\Sigma)$ as the usual double coset operator
$$
\iota_u^{-1}\left[ \GL_n(\cO_{F_u}) \left(\begin{matrix}
      \varpi_{u}\mathrm{Id}_i &  \cr  & \mathrm{Id}_{n-i} \end{matrix} \right)
\GL_n(\cO_{F_u}) \right],
$$
where $\varpi_u$ denotes a uniformizer of $F_u$.
If $\rbar:G_F\rightarrow \GL_n(\F)$ is a continuous, absolutely irreducible Galois representation, we further write $\mathfrak{m}_{\rbar}$ for the maximal ideal of $\bT^{\cP}$ with residue field $\F$ defined by the formula
$$
\det\left(1-\overline{r}^{\vee}(\mathrm{Frob}_u)X\right)=\sum_{j=0}^n (-1)^j \mathbf{N}_{F_u/\Q_p}(u)^{\binom{j}{2}}(T_u^{(j)}\bmod \mathfrak{m}_{\rbar})X^j\quad \forall u\in \cP.
$$

\begin{defn} \label{defn:rbarmod}
We say that $\rbar$ is {\it modular} of (prime-to-$\Sigma$) level $U^\Sigma$ if $S(U^\Sigma)_{\mathfrak{m}_{\rbar}} \neq 0$.
We say $\rbar$ is modular if $\rbar$ is modular of some level $U^\Sigma$.
\end{defn}

Assume that $\rbar$ is modular of level $U^\Sigma$ and that $\Sigma$ contains all places dividing $p$ and all places divisible by places in $F$ where $\rbar$ is ramified.
Fix places $u|v|p$ of $F$ and $F^+$.
Then we define the functor $S:\F[\GL_n(k_u)]\textrm{-mod} \ra \mathrm{Vect}_{/\F}$ by $S(V) \defeq \Hom_{\GL_n(\cO_{F_u})}(V^\vee,S(U^\Sigma)_{\mathfrak{m}_{\rbar}})$, where $(\cdot)^\vee$ denotes the contragradient representation and $\GL_n(\cO_{F_u})$ acts on $V$ by inflation and on $S(U^\Sigma)_{\mathfrak{m}_{\rbar}}$ via $\iota_u^{-1}$.

\begin{prop}\label{prop:acf}
If $U^\Sigma$ is sufficiently small, then $S$ is an arithmetic cohomology functor for $\rbar|_{G_{F_u}}$.
\end{prop}
\begin{proof}
This essentially follows from the proof of Proposition 7.4.4 of \cite{EGH}.
Note that modularity of $\rbar$ of level $U^\Sigma$ implies that the functor $S$ is nonzero.
Suppose $S(V) \neq 0$ and $V \in \JH(\overline{\sigma}(\tau))$. 
Let $v' \in \Sigma$ be a place not dividing $p$ (if one exists).
Since $S(V)$ is a smooth representation of $G(F^+_{v'})$, there exists a compact open subgroup $K_{v'}$ of $G(\cO_{F^+_{v'}})$ with nonzero invariants.
Inductively choosing $K_{v'}$ and replacing $S(U^\Sigma)_{\mathfrak{m}_{\rbar}}$ with 
\[S(U^\Sigma\prod_{v'\in \Sigma, v'\nmid p} K_{v'})_{\mathfrak{m}_{\rbar}},\]
we assume without loss of generality that $\Sigma$ is exactly the places of $F^+$ dividing $p$, and write $U^{\Sigma}=U^p$.
For each place $v'|p$ of $F^+$, choose a place $\tld{v}'|v'$ of $F$ such that $\tld{v}'|_{F^+} = v'$ and that $\tld{v} = u$.
For an $\cO[\![\GL_n(\cO_{F,\tld{v}'})]\!]$-representation $V_{\tld{v}'}$, let $V_{v'}$ be the corresponding $\cO[\![\cG(\cO_{F^+,v'})]\!]$-representation via $\iota_{\tld{v}}$.
There exist irreducible $\F[\![\GL_n(\cO_{F,\tld{v}'})]\!]$-representations $V_{\tld{v}'}$ for every $v'|p$ such that $V_{\tld{v}} = V$ and $\Hom_{\cG(\cO_{F^+,p})}(\otimes_{v'|p} V_{v'}^\vee,S(U^p)_{\mathfrak{m}_{\rbar}}) \neq 0$.
Choose an $\cO[\![\cG(\cO_{F^+,p})]\!]$-representation $W = \otimes_{v'|p} W_{v'}$ such that $W_{\tld{v}}$ is an $\cO$-lattice in $\sigma(\tau)$, $W_{\tld{v'}}$ is an $\cO$-lattice in an algebraic representation over $E$ for $v'\neq v$ (say, a suitable Weyl module), and $V_{\tld{v}'} \in \JH(\overline{W}_{\tld{v}'})$ for all $v'|p$ (where the bar denotes mod $\varpi$ reduction). For $U^p$ sufficiently small, $S(U^p)_{\mathfrak{m}_{\rbar}}$ is an injective $\F[\![\cG(\cO_{F^+,p})]\!]$-module, hence $\Hom_{\F[\![\cG(\cO_{F^+,p})]\!]}(-,S(U^p)_{\mathfrak{m}_{\rbar}})$ is exact. Thus $\Hom_{\cG(\cO_{F^+,p})}(\overline{W}^\vee,S(U^p)_{\mathfrak{m}_{\rbar}})\neq 0$.
Then the proof of \cite[Proposition 7.4.4 \~A1]{EGH} holds without modification and one then constructs $\pi$ as in the proof of {\it loc.~cit.}
Then $r_\pi^\vee|_{G_{F_u}}$ provides the required lift by \cite[Theorem 7.2.1]{EGH}, Theorem \ref{thm:ill}, and Proposition \ref{illtame}.
\end{proof}


Assume that $\rbar$ is modular.
Then $\rbar$ is modular of some sufficiently small level $U^\Sigma$.
Let $W(\rbar)$ be the set of isomorphism classes of irreducible $\mathcal{G}(\cO_{F^+,p})$-representations $V$ over $\F$ such that \[\Hom_{\cG(\cO_{F^+,p})}(V^\vee,S(U^\Sigma)_{\mathfrak{m}_{\rbar}}) \neq 0.\]
For each place $v|p$ of $F^+$ choose a place $\tld{v}|v$ of $F$.
Let $W_{\tld{v},\mathrm{ss}}^?(\rbar)$ be $W^?((\rbar|_{G_{F_{\tld{v}}}})^{\mathrm{ss}})$ and $W_{v,\mathrm{ss}}^?(\rbar)$ be the set of isomorphism classes of $\cG(\cO_{F^+_v})$-representations over $\F$ corresponding to $W_{\tld{v},\mathrm{ss}}^?(\rbar)$ via $\iota_{\tld{v}}$. 
Note that the definition of $W^?_{v,\mathrm{ss}}(\rbar)$ does not depend on the choice of place $\tld{v}$.
Let $W^?_\mathrm{ss}(\rbar)$ be $\underset{v|p}{\otimes} W_{v,\mathrm{ss}}^?(\rbar)$.

\begin{cor}\label{cor:globalwe}
Suppose that $(\rbar|_{G_{F_{\tld{v}}}})^{\mathrm{ss}}$ is $(6n-2)$-generic for all places $v|p$ of $F^+$.
We have $W(\rbar)\subset W^?_{\mathrm{ss}}(\rbar)$.
\end{cor}
\begin{proof}
Suppose that $\Hom_{\cG(\cO_{F^+,p})}(\otimes_{v|p} F(\lambda_v)^\vee,S(U^p)_{\mathfrak{m}_{\rbar}}) \neq 0$.
For each place $v|p$ of $F^+$, let $S_{\tld{v}}$ be $\Hom_{\GL_n(\cO_{\tld{v}})}((-)^\vee, S(U^p)_{\mathfrak{m}_{\rbar}})$.
By Proposition \ref{prop:acf}, $S_{\tld{v}}$ is an arithmetic cohomology functor for $\rbar|_{G_{F_{\tld{v}}}}$, from which we define $W(\rbar|_{G_{F_{\tld{v}}}})$ as in Definition \ref{defn:W}.
If $F(\lambda_{\tld{v}})$ corresponds to $F(\lambda_v)$ via $\iota_{\tld{v}}$, then $S_{\tld{v}}(F(\lambda_{\tld{v}}))$ is nonzero, and so $F(\lambda_{\tld{v}}) \in W(\rbar|_{G_{F_{\tld{v}}}})$.
By Corollary \ref{cor:awe}, $F(\lambda_{\tld{v}}) \in W^?_{\mathrm{ss}}(\rbar|_{G_{F_{\tld{v}}}})$.
Thus $\otimes_{v|p} F(\lambda_v) \in W^?_{\mathrm{ss}}(\rbar)$.
\end{proof}

\subsection{Modularity of ``obvious'' Serre weights}\label{sec:obv}

In this section, we deduce the modularity of obvious Serre weights for sufficiently generic semisimple local Galois representations under mild hypotheses.
We show that for each obvious weight there is a type containing it as a Jordan--H\"older factor so that no other Jordan--H\"older factor is modular (Corollary \ref{cor:globalwe} and Proposition \ref{prop:obvtype}).
Fortunately, these are precisely the types to which we can apply Corollary \ref{cor:pd} and the results of \cite{BLGGT} to deduce the modularity of obvious Serre weights.
That these types isolate weights can be seen as a consequence of Theorem \ref{thm:smoothdef} and the Breuil--M\'ezard philosophy (see Remark \ref{rmk:bm}).
We first summarize the results we need from \cite{BLGGT} and refer the reader to \emph{ibid.}~for any undefined notation and terminology.

\begin{thm}\label{thm:cow}
Let $p>2$ and $F$ be a CM field with maximal totally real subfield $F^+$ such that $\zeta_p \notin F$.
Assume that $F/F^+$ is split at all places dividing $p$.
Suppose that $\rbar:G_F\ra \GL_n(\overline{\F}_p)$ is an irreducible representation with the following additional properties.
\begin{enumerate}
\item \label{pdaut} $\rbar$ is potentially diagonalizably automorphic, i.e.\ there is a RACSDC automorphic representation $\Pi$ of $\GL_n(\A_F)$ such that 
\begin{itemize}
\item $\rbar\cong \rbar_{p,\iota}(\Pi)$; and
\item For each place $u|p$ of $F$, $r_{p,\iota}(\Pi)|_{G_{F_u}}$ is potentially diagonalizable.
\end{itemize}
\item The image of $\rbar(G_{F(\zeta_p)})$ is adequate.
\end{enumerate}
Let $\Sigma$ be a finite set of places of $F^+$ containing all places dividing $p$ and all the places of $F^+$ divisible by places at which $\rbar$ ramifies.
For each place $v\in \Sigma$, choose a place $\tld{v}|v$ of $F$ and a lift $\rho_{\tld{v}}:G_{F_{\tld{v}}} \ra \GL_n(\overline{\Z}_p)$.
Suppose that for $v|p$, $\rho_{\tld{v}}$ is potentially crystalline and potentially diagonalizable with distinct Hodge--Tate weights for every embedding $F_{\tld{v}} \into \overline{\Q}_p$.

Then there is a RACSDC automorphic representation $\pi$ such that
\begin{itemize}
\item $\rbar \cong \rbar_{p,\iota}(\pi)$;
\item $\pi_u$ is unramified at all places $u$ of $F$ that do not divide a place in $\Sigma$; and 
\item $r_{p,\iota(\pi)}|_{G_{F_{\tld{v}}}} \sim \rho_{\tld{v}}$ for all places $v\in \Sigma$.
\end{itemize}
\end{thm}
\begin{proof}
This is \cite[Theorem 3.1.3]{BLGG} except for two differences:
\begin{enumerate}
\item for places $v|p$ of $F^+$, $\rhobar_{\tld{v}}$ allowed to be potentially crystalline rather than crystalline and
\item $\Sigma$ may contain places which do not split in $F$.
\end{enumerate}
However, the proof of \cite[Theorem 3.1.3]{BLGG} still applies with two corresponding modifications:
\begin{enumerate}
\item $\pi_u$ is not necessarily unramified if $u$ is a place of $F$ dividing $p$ and
\item we replace the use of \cite[Theorem A.4.1]{BLGG} with \cite[Theorem 5.2.1]{BG}. 
\end{enumerate}
\end{proof}


Define
\[
W^{?}(\rhobar, \tau) := W^{?}(\rhobar) \cap \JH(\taubar)
\] 
(see Definition \ref{defn:SWC} and \S \ref{sec:ill}). We can characterize $W^{?}(\rhobar,\tau)$ in terms of the element $\tld{w}^*(\rhobar, \tau)$ from Definition \ref{defn:relshape}.

\begin{prop}\label{prop:alcovecrit}
Fix \S \ref{sec:we}\ref{item:rhobartau1}-\ref{item:rhobartau3} with $\mu-\eta$ $2n$-deep in $\bun{C}_0$, and let $\widetilde{w} = \widetilde{w}^*(\rhobar,\tau)=t_{\nu}w$. 
Then for $\lambda$ a dominant $p$-restricted character, $F(\lambda)$ is in $W^?(\rhobar,\tau)$ if and only if there exist $\tld{w}_\lambda$, $\tld{w}_1$, and $\tld{w}_2\in \tld{\bun{W}}^+$ with $\tld{w}_\lambda\cdot \bun{C}_0$ $p$-restricted, and $w'\in W(\bun{G})$ such that 
\begin{itemize}
\item $\pi^{-1}(\tld{w}) = \tld{w}_2^{-1} w' \tld{w}_1$;
\item $\tld{w}_1 \uparrow \tld{w}_\lambda \uparrow \tld{w}_h^{-1} \tld{w}_2$; and 
\item $\lambda=\tld{w}_\lambda\cdot (\mu-s\pi\nu_2-\eta)$ where $\tld{w}_2 = w_2 t_{\nu_2}$.
\end{itemize}
\end{prop}
\begin{proof}
Recall from Definition \ref{defn:SWC} that $F(\lambda) \in W^?(\rhobar)$ if and only if $F(\tld{w}_h^{-1}\cdot \lambda) \in \JH(\overline{R}_{sw}(\mu+s\nu))$ since $\tld{w}_h \cdot \lambda = \mathcal{R}(\lambda)$ (recall that $\mu+s\nu = \mu_{\rhobar}+\eta$ so that $\nu\equiv \eta\pmod {\un{\Lambda}_R}$ and $sw = s_{\rhobar}$). 
By Proposition \ref{prop:JH}, $F(\lambda) \in \JH(\taubar)$ (resp. $F(\tld{w}_h^{-1} \cdot \lambda) \in \JH(\overline{R}_{sw}(\mu+s\nu))$) if and only if there exists a $\tld{w}_2 \in \tld{\bun{W}}^+$ (resp. $\tld{w}_1\in \tld{\bun{W}}^+$) such that $\tld{w}_2\cdot(\mu-s\pi\nu_2-\eta) \uparrow \tld{w}_h \cdot \lambda$ (resp. $\tld{w}_1\cdot(\mu+s\nu-sw\pi\nu_1-\eta) \uparrow \lambda$) where $\tld{w}_2 = w_2t_{\nu_2}$ (resp. $\tld{w}_1 = w_1t_{\nu_1}$) is such that $\tld{w}_2\cdot \bun{C}_0\uparrow \tld{w}_h\cdot \bun{C}_0$ (and $\tld{w}_1\cdot \bun{C}_0\uparrow \tld{w}_h\cdot \bun{C}_0$).
In summary, $F(\lambda) \in W^?(\rhobar,\tau)$ if and only if there exist $\tld{w}_1$ and $\tld{w}_2 \in \tld{\bun{W}}^+$ with $\tld{w}_i\cdot \bun{C}_0\uparrow \tld{w}_h\cdot \bun{C}_0$ for $i=1,2$ such that \[ \tld{w}_1\cdot (\mu+s\nu - sw\pi\nu_1 -\eta) \uparrow \lambda \uparrow \tld{w}_h^{-1} \tld{w}_2 \cdot (\mu-s\pi\nu_2 -\eta). \]

By our assumption, $\mu_{\rhobar}=\mu+s\nu-\eta$ and $\mu-\eta$ are both $2n$-deep in $\bun{C}_0$, so Remark \ref{rmk:bounderror} implies both $\mu+s\nu - sw\pi\nu_1 -\eta$ and $\mu-s\pi\nu_2 -\eta$ are in $\bun{C}_0$.
Thus the above condition is equivalent to the existence of $\tld{w}_\lambda$, $\tld{w}_1$, and $\tld{w}_2\in \tld{\bun{W}}^+$ with $\tld{w}_\lambda\cdot \bun{C}_0$ $p$-restricted such that 
\begin{itemize}
\item $\tld{w}_1\cdot(\mu+s\nu-sw\nu_1-\eta)$ is linked to $\tld{w}_h^{-1}\tld{w}_2\cdot (\mu-s\pi\nu_2-\eta)$;
\item $\tld{w}_1\uparrow \tld{w}_\lambda\uparrow\tld{w}_h^{-1}\tld{w}_2$; and 
\item $\lambda = \tld{w}_\lambda\cdot(\mu-s\pi\nu_2-\eta)$.
\end{itemize}
We claim that $\tld{w}_1\cdot(\mu+s\nu-sw\pi\nu_1-\eta)$ and $\tld{w}_h^{-1}\tld{w}_2\cdot (\mu-s\pi\nu_2-\eta)$ are linked if and only if $\tld{w}_2 \pi^{-1}(\tld{w}) \tld{w}_1^{-1} \in W(\bun{G})$.

We first show that $\tld{w}_1\cdot(\mu+s\nu-sw\pi\nu_1-\eta)$ and $\tld{w}_h^{-1}\tld{w}_2\cdot (\mu-s\pi\nu_2-\eta)$ are linked if and only if $\nu+\pi\nu_2 = w\pi \nu_1$.
If $\nu+\pi\nu_2 = w\pi \nu_1$, then one sees directly that $\tld{w}_1^{-1}\tld{w}_h^{-1}\tld{w}_2$ is in $\bun{W}_a$, from which we see that $\tld{w}_1\cdot(\mu+s\nu-sw\pi\nu_1-\eta)$ and $\tld{w}_h^{-1}\tld{w}_2\cdot (\mu-s\pi\nu_2-\eta)$ are linked.
Now suppose that $\tld{w}_1\cdot(\mu+s\nu-sw\pi\nu_1-\eta)$ and $\tld{w}_h^{-1}\tld{w}_2\cdot (\mu-s\pi\nu_2-\eta)$ are linked.
Then the restriction of the difference of $\tld{w}_1\cdot(\mu+s\nu-sw\pi\nu_1-\eta)$ and $\tld{w}_h^{-1}\tld{w}_2\cdot (\mu-s\pi\nu_2-\eta)$ to the center $\bun{Z}$ of $\bun{G}$ must be trivial.
Equivalently, the restriction of $(p-\pi)(\nu-\nu_1+\nu_2)$ and therefore $\nu-\nu_1+\nu_2$ to $\bun{Z}$ must be trivial.
Noting that $\nu\equiv \eta \mod \un{\Lambda}_R$, this implies that $\tld{w}^{-1}\tld{w}_h^{-1}\tld{w}_2$ is in $\bun{W}_a$.
We conclude that $\mu+s\nu-sw\pi\nu_1-\eta$ and $\mu-s\pi\nu_2-\eta$ are linked.
Since $\mu+s\nu-sw\pi\nu_1-\eta$ and $\mu-s\pi\nu_2-\eta$ are both in $\bun{C}_0$, and must therefore be equal.
This equality implies the equality $\nu+\pi\nu_2 = w\pi \nu_1$.

Finally, $\nu+\pi\nu_2 = w\pi \nu_1$ if and only if 
\[\tld{w}_2 \pi^{-1}(\tld{w}) \tld{w}_1^{-1} = w_2 t_{\nu_2} \pi^{-1} (t_\nu w) t_{-\nu_1} w_1^{-1} = w_2 \pi^{-1}(t_{\pi\nu_2+\nu-w\pi \nu_1} w) w_1^{-1} \in W(\bun{G}).\]
\end{proof}
\begin{rmk}\label{rmk:obv}
Note that $F(\lambda)$ is an obvious weight (Definition \ref{defn:obvweight}) if and only if $\tld{w}_\lambda = \tld{w}_1$.
\end{rmk}
\begin{lemma}\label{lemma:domineq}
Let $\tld{w} \in \tld{\bun{W}}$ and let $w\in W(\bun{G})$, $\tld{w}^+ \in \tld{\bun{W}}^+$ be the unique elements such that $\tld{w} = w \tld{w}^+$.
Then $\tld{w}^+ \leq \tld{w}$.
\end{lemma}
\begin{proof}
The length $\ell(w\tld{w}^+)$ is the length of a minimal gallery from $\bun{C}_0$ to $w\tld{w}^+ \cdot \bun{C}_0$, which is the length of a minimal gallery from $w^{-1} \cdot \bun{C}_0$ to $\tld{w}^+ \cdot \bun{C}_0$.
A minimal gallery can be taken through $\bun{C}_0$ in the dominant direction.
Hence $\ell(w\tld{w}^+) = \ell(w)+\ell(\tld{w}^+)$, and therefore $\tld{w}^+\leq \tld{w}$.
\end{proof}

\begin{prop}\label{prop:comp}
If $\tld{w}_2\uparrow \tld{w}_1$ and $\tld{w}_2\in \tld{\bun{W}}^+$, then $\tld{w}_2\leq \tld{w}_1$.
\end{prop}
\begin{proof}
Let $w\in W$ be the unique element such that $w\tld{w}_1\in \tld{\bun{W}}^+$.
Then $\tld{w}_1 \uparrow w\tld{w}_1$ by II 6.5(5) of \cite{RAGS}.
Then $\tld{w}_2 \leq w\tld{w}_1$ by Theorem \ref{thm:pocomp}.
Since $w\tld{w}_1 \leq \tld{w}_1$ by Lemma \ref{lemma:domineq}, $\tld{w}_2\leq \tld{w}_1$.
\end{proof}

\begin{prop}\label{prop:obvtype}
Let $\rhobar$, $\tau$, and $\tld{w}$ be as in Proposition \ref{prop:alcovecrit} and suppose that $\tld{w} = t_{\pi s^{-1}(\eta)}$.
Then $W^?(\rhobar,\tau) = \{F(\lambda)\}$, where $F(\lambda)\in W_{\mathrm{obv}}(\rhobar)$ is the obvious weight corresponding to $s$ $($see Definition \ref{defn:obvweight}$)$.
\end{prop}
\begin{proof}
Suppose that $\pi^{-1}(\tld{w}) = t_{s^{-1}(\eta)} = \tld{w}_2^{-1}w'\tld{w}_1$ where $w'\in W(\bun{G})$, $\tld{w}_1$ and $\tld{w}_2\in \tld{\bun{W}}^+$, and $\tld{w}_1\uparrow \tld{w}_\lambda\uparrow \tld{w}_h^{-1}\tld{w}_2$ for some $\tld{w}_\lambda$ with $\tld{w}_\lambda\cdot \bun{C}_0$ $p$-restricted. 
We have
\[\ell(t_{s^{-1}(\eta)}) \leq \ell(\tld{w}_2^{-1})+\ell(w')+\ell(\tld{w}_1) \leq \ell((\tld{w}_h\tld{w}_\lambda)^{-1})+\ell(w_0)+\ell(\tld{w}_\lambda) = \ell((\tld{w}_h\tld{w}_\lambda)^{-1}w_0\tld{w}_\lambda) = \ell(t_{w_\lambda^{-1}(\eta)}),\]
where $w_\lambda\in W(\bun{G})$ is the projection of $\tld{w}_\lambda$.
 The first inequality is obvious, while the first equality follows from Lemma \ref{lemma:gallery}.
 For the second inequality, since $w'\leq w_0$, it suffices to show that $\tld{w}_2\leq \tld{w}_h\tld{w}_\lambda$ and $\tld{w}_1 \leq \tld{w}_\lambda$.
Since $\tld{w}_\lambda\uparrow \tld{w}_h^{-1}\tld{w}_2$, we have that $\tld{w}_2\uparrow \tld{w}_h\tld{w}_\lambda$.
Since $\tld{w}_1$ and $\tld{w}_2 \in \tld{\bun{W}}^+$, we have that $\tld{w}_2\leq \tld{w}_h\tld{w}_\lambda$ and $\tld{w}_1\leq \tld{w}_\lambda$ by Proposition \ref{prop:comp}.

Since $\ell(t_{s^{-1}(\eta)}) = \ell(t_{w_\lambda^{-1}(\eta)})$, we have $w' = w_0$ and $\tld{w}_2 = \tld{w}_h\tld{w}_\lambda = \tld{w}_h\tld{w}_1$.
This implies that $s = w_\lambda$. 
Now suppose that $F(\lambda) \in W^?(\rhobar, \tau)$.
We now use notation from Proposition \ref{prop:alcovecrit}, particularly from \ref{item:rhobartau2} and Definition \ref{defn:relshape}.
Then by Proposition \ref{prop:alcovecrit}, $\lambda = \tld{w}_\lambda \cdot (\mu_{\rhobar}-s_{\rhobar}\pi\nu_\lambda)$ where we write $\tld{w}_\lambda = w_\lambda t_{\nu_\lambda}$ and $\rhobar|_{I_K} = \tau(s_{\rhobar},\mu_{\rhobar} + \eta)$.
This is exactly the obvious weight corresponding to $s = w_\lambda$ (Definition \ref{defn:obvweight}).
\end{proof}

\begin{rmk}\label{rmk:bm}
One could show using Theorem \ref{thm:smoothdef}, Corollary \ref{cor:globalwe}, and Kisin's approach to the Breuil--M\'ezard conjecture that with the hypotheses of Proposition \ref{prop:obvtype}, $\#W^?(\rhobar,\tau) \leq 1$.
This leads to an alternate proof of Proposition \ref{prop:obvtype}, which we eschew in favor of our more direct approach.
\end{rmk}

In the setting of Proposition \ref{prop:obvtype}, if $\tld{w}^*(\rhobar,\tau) = t_{\pi w^{-1}\eta}$ for some $w\in W(\bun{G})$, we say that $\tau$ is the \emph{obvious type} for the obvious weight of $\rhobar$ corresponding to $w$ (note that this notion depends on the choice of $(s_{\rhobar},\mu_{\rhobar} )$). Such a type $\tau$ always exists, and is uniquely determined by the corresponding obvious weight.

We use the setup and notation of \S \ref{sec:algmf}.
For each place $v|p$ of $F^+$ choose a place $\tld{v}|v$ of $F$.
Let $\rbar:G_F\ra \GL_n(\F)$ be a modular Galois representation such that for each place $v|p$ of $F^+$, $\rbar|_{G_{F_{\tld{v}}}}$ is semisimple.
Let $W_{\mathrm{obv},\tld{v}}(\rbar) \defeq W_{\mathrm{obv}}(\rbar|_{G_{F_{\tld{v}}}})$.
Let $W_{\mathrm{obv},v}(\rbar)$ be the set of isomorphism classes of $\cG(\cO_{F^+_v})$-representations over $\F$ corresponding to $W_{\mathrm{obv},\tld{v}}(\rbar)$.
Note that the definition of $W_{\mathrm{obv},v}(\rbar)$ does not depend on the choice of place $\tld{v}$.
Let $W_{\mathrm{obv}}(\rbar)$ be $\underset{v|p}{\otimes} W_{\mathrm{obv},v}(\rbar)$.

\begin{thm} \label{thm:obvwt}
Suppose that $\zeta_p \notin F$ and that $\rbar:G_F \ra \GL_n(\F)$ is a modular Galois representation such that $\rbar(G_{F(\zeta_p)})$ is adequate.  Assume that for all $\tld{v} \mid p,$ $\rbar|_{G_{F_{\tld{v}}}}$ is semisimple and $(6n-2)$-generic.
Then the following are equivalent:
\begin{enumerate}
\item \label{item:nonempt} $W_{\mathrm{obv}}(\rbar) \cap W(\rbar)\neq \emptyset$;
\item \label{item:pdaut} $\rbar$ is potentially diagonalizably automorphic \emph{(}see Theorem \ref{thm:cow}\emph{(}\ref{pdaut}\emph{))}; and
\item \label{item:obv} $W_{\mathrm{obv}}(\rbar) \subset W(\rbar)$.
\end{enumerate}
\end{thm}
\begin{proof}
Clearly, (\ref{item:obv}) implies (\ref{item:nonempt}).
We next show that (\ref{item:nonempt}) implies (\ref{item:pdaut}).
For each place $v|p$ of $F^+$, choose a place $\tld{v}|v$ of $F$.
Suppose that $\otimes_{v|p} F(\lambda_v) \in W_{\mathrm{obv}}(\rbar)\cap W(\rbar)$.
Then 
\[\Hom_{\cG(\cO_{F^+,p})}(\otimes_{v|p} F(\lambda_v)^\vee, S(U^p)_{\mathfrak{m}_{\rbar}}) \neq 0\]
for some sufficiently small compact open subgroup $U^p\leq G(\A_{F^+}^{\infty,p})$ (we can replace $\Sigma$ with the set of places dividing $p$ as in the proof of Proposition \ref{prop:acf}).
Say $F(\lambda_v)$ corresponds to $F(\lambda_{\tld{v}}) \in W_{\mathrm{obv}}(\rbar|_{G_{F_{\tld{v}}}})$ via $\iota_{\tld{v}}$, and that $F(\lambda_{\tld{v}}) \in W_{\mathrm{obv}}(\rbar|_{G_{F_{\tld{v}}}})$ (resp. $\tau_{\tld{v}}$) is the obvious weight (resp. the obvious type) for $\rbar|_{G_{F_{\tld{v}}}}$ corresponding to $w_v$ (after choosing a lowest alcove presentation of $\rbar|_{I_{F_{\tld{v}}}}$).
One checks directly from the definition of obvious type that $\tau_{\tld{v}}$ is $2n$-generic,
and hence any lowest alcove presentation of $\tau_{\tld{v}}$ satisfies the hypothesis of Theorem \ref{thm:smoothdef}, by Proposition \ref{prop:alcoveC0}.
Let $\sigma(\tau_v)$ be the $\cG(\cO_{F^+_v})$-representation corresponding to $\sigma(\tau_{\tld{v}})$ via $\iota_{\tld{v}}$.
Note that $S(U^p)_{\mathfrak{m}_{\rbar}}$ is an injective $\F[\![\cG(\cO_{F^+,p})]\!]$-module as $U^p$ is sufficiently small, thus $\Hom_{\F[\![\cG(\cO_{F^+,p})]\!]}(-, S(U^p)_{\mathfrak{m}_{\rbar}})$ is exact.
Since 
\[\Hom_{\cG(\cO_{F^+,p})}(\otimes_{v|p} F(\lambda_v)^\vee, S(U^p)_{\mathfrak{m}_{\rbar}}) \neq 0,\]
we have that
\[\Hom_{\cG(\cO_{F^+,p})}(\otimes_{v|p} \overline{\sigma}(\tau_v)^\vee, S(U^p)_{\mathfrak{m}_{\rbar}}) \neq 0,\]
where $\overline{\sigma}(\tau_v)$ is the reduction of some $\cO$-lattice for each $v|p$.
A nonzero element of 
\[
\Hom_{\cG(\cO_{F^+,p})}(\otimes_{v|p} \overline{\sigma}(\tau_v)^\vee, S(U^p)_{\mathfrak{m}_{\rbar}}) \neq 0
\]
gives an automorphic lift $r_{p,\iota}(\Pi)$ of $\rbar$ whose restriction at $\tld{v}$ is potentially crystalline of type $(\eta,\tau_{\tld{v}})$ by \cite[Theorem 7.2.1]{EGH}, Theorem \ref{thm:ill}, and Proposition \ref{illtame}.
Thus $r_{p,\iota}(\Pi)|_{G_{F_u}}$ is potentially diagonalizable for each place $u|p$ of $F$ by Corollary \ref{cor:pd}.

Finally, we show that (\ref{item:pdaut}) implies (\ref{item:obv}).
Assuming (\ref{item:pdaut}), we see that $\rbar$ satisfies the enumerated hypotheses of Theorem \ref{thm:cow}.
Suppose now that $\otimes_{v|p} F(\lambda_v) \in W_{\mathrm{obv}}(\rbar)$ is arbitrary and let $F(\lambda_{\tld{v}})$, $\tau_{\tld{v}}$, $\sigma(\tau_{\tld{v}})$, and $\sigma(\tau_v)$ be as in the last paragraph.
For each place $v|p$ of $F^+$, let $\rho_{\tld{v}}$ be a potentially diagonalizable potentially crystalline lift of $\rbar|_{G_{F_{\tld{v}}}}$ of type $(\eta,\tau_{\tld{v}})$ (say the lift from Corollary \ref{cor:existencelift}).
By Theorem \ref{thm:cow}, there is an automorphic lift $r_{p,\iota}(\pi)$ of $\rbar$ whose restriction at $\tld{v}$ is potentially crystalline of type $(\eta,\tau_{\tld{v}})$, which is unramified outside the places where $\rbar$ is ramified.
Thus $\Hom_{\cG(\cO_{F^+,p})}(\otimes_{v|p} \overline{\sigma}(\tau_v)^\vee, S(U^\Sigma)_{\mathfrak{m}_{\rbar}}) \neq 0$ for any $\Sigma$ containing all places dividing $p$ and all places divisible by places in $F$ where $\rbar$ is ramified and any sufficiently small $U^\Sigma$.
By Corollary \ref{cor:globalwe} and Proposition \ref{prop:obvtype}, we conclude that $\Hom_{\cG(\cO_{F^+,p})}(\otimes_{v|p} F(\lambda_v)^\vee, S(U^\Sigma)_{\mathfrak{m}_{\rbar}}) \neq 0$. 
Thus $\otimes_{v|p} F(\lambda_v)\in W(\rbar)$.
\end{proof}

\begin{rmk}
In \cite{BLGG}, it is shown that if $\rbar$ is modular of a Fontaine--Laffaille weight, $\otimes_{v|p} F(\lambda_v) \in W_{\mathrm{obv}}(\rbar)$, and $p$ splits completely in $F$, then 
\[\Hom_{\cG(\cO_{F^+,p})}(\otimes_{v|p} \overline{W}(\lambda_v)^\vee, S(U^\Sigma)_{\mathfrak{m}_{\rbar}}) \neq 0,\]
which is strictly weaker than Theorem \ref{thm:obvwt}.

In \S 6 of \cite{gee-geraghty}, it is shown that if $\rbar$ is assumed to be modular and ordinary at $p$, then $\rbar$ is modular of all ordinary obvious weights (this is all obvious weights if $p$ splits completely but is strictly smaller otherwise). In \emph{loc.\ cit.}, $\rbar$ is no longer assumed to be semisimple above $p$.   
\end{rmk}


\subsection{Type changing congruences and a local lifting problem}\label{sec:red}

In this section, we give a classification (Theorem \ref{thm:lift}) of congruences between RACSDC $\GL_n$-automorphic representations of trivial weight and generic tame type whose associated Galois representations are residually tamely ramified at $p$.
We also solve the corresponding local Galois lifting problem.
Throughout this section, we are in the setting of \S \ref{sec:we}\ref{item:rhobartau1}-\ref{item:rhobartau3}, that is, we fix
\begin{enumerate}
\item a generic semisimple Galois representation $\rhobar:G_K \ra \GL_n(\F)$; 
\item a pair $(s_{\rhobar}, \mu_{\rhobar})$ such that $\rhobar|_{I_K} \cong \ovl{\tau}(s_{\rhobar},\mu_{\rhobar}+\eta)$ with $\mu_{\rhobar}$ 
in $\bun{C}_0$; and 
\item A lowest alcove presentation $(s,\mu-\eta)$ of a tame inertial type $\tau\cong\tau(s,\mu):I_K \ra \GL_n(\cO)$ such that $\mu_{\rhobar}-\mu \in \un{\Lambda}_R$.     
\end{enumerate}
\begin{prop}\label{prop:obvintersect}
Let $\rhobar$ and $\tau$ be as above with $\mu-\eta$ $2n$-deep in $\bun{C}_0$.
The set $W^?(\rhobar,\tau)$ is nonempty if and only if the set $W_{\mathrm{obv}}(\rhobar)\cap \JH(\taubar)$ is nonempty.
\end{prop}
\begin{proof} 
The ``if" part of the claim is clear.
Suppose that $W^?(\rhobar,\tau)$ is nonempty.
Let $\widetilde{w}=\widetilde{w}^*(\rhobar,\tau)$.
By Proposition \ref{prop:alcovecrit}, $\pi^{-1}(\widetilde{w}) = \widetilde{w}_2^{-1} w' \widetilde{w}_1$ where $w' \in W$, $\widetilde{w}_1$ and $\widetilde{w}_2\in \tld{\bun{W}}^+$, and $\widetilde{w}_1 \uparrow \tld{w}_h^{-1}\widetilde{w}_2$.
Let $\omega$ be a weight (unique up to weights whose restrictions to the derived group are trivial) such that $t_{-\omega}\widetilde{w}_1 \cdot (\bun{C}_0)$ is $p$-restricted.
Note that $\omega$ is dominant since the set of dominant alcoves is exactly the set of dominant translates of the restricted ones.
Then $\pi^{-1}(\widetilde{w}) = (t_{-w'\omega}\widetilde{w}_2)^{-1} w' (t_{-\omega}\widetilde{w}_1)$.
Let $t_{-w'\omega}\widetilde{w}_2 = w^{-1} \widetilde{w}_3$ where $w \in W$ and $\widetilde{w}_3\in \tld{\bun{W}}^+$.
It suffices to show that $t_{-\omega}\widetilde{w}_1 \uparrow \tld{w}_h^{-1}\widetilde{w}_3$,
since then by Proposition \ref{prop:alcovecrit}, taking $\tld{w}_\lambda=t_{-\omega}\widetilde{w}_1$ we see that $W^{?}(\rhobar,\tau)$ contains the obvious weight corresponding to the permutation part of $t_{-\omega}\widetilde{w}_1$ via the bijection in the proof of Corollary \ref{cor:obvweight} (see also Remark \ref{rmk:obv}).

Using that $\widetilde{w}_1 \uparrow \tld{w}_h^{-1}\widetilde{w}_2$,
it suffices to show that
$\tld{w}_h^{-1}\widetilde{w}_2 \uparrow \tld{w}_h^{-1}t_{w_0\omega} \widetilde{w}_3,$
or equivalently that $t_{w_0\omega} \widetilde{w}_3 \uparrow \widetilde{w}_2$. 
Now $t_{w_0\omega} \widetilde{w}_3 = t_{w_0\omega-ww'\omega} w\widetilde{w}_2$ by definition.
Note that $w_0\omega-ww'\omega$ is a sum of negative roots since $\omega$ is dominant.
Then $t_{w_0\omega-ww'\omega} w\widetilde{w}_2 \uparrow w\widetilde{w}_2 \uparrow \widetilde{w}_2$ by II 6.5(3) and (5) of \cite{RAGS}.
\end{proof}

\begin{prop}\label{prop:intersect}
Let $\rhobar$ and $\tau$ be as in Proposition \ref{prop:obvintersect}.
Then $\tld{w}^*(\rhobar,\tau) \in \Adm(\eta)$ if and only if $W^?(\rhobar,\tau)$ is nonempty.
\end{prop}
\begin{proof}
Suppose that $W^?(\rhobar,\tau)$ is nonempty and let $\tld{w} = \tld{w}^*(\rhobar,\tau)$.
By Proposition \ref{prop:alcovecrit}, $\pi^{-1}(\tld{w}) = \tld{w}_2^{-1}w'\tld{w}_1$ where $w'\in W$, $\tld{w}_1$ and $\tld{w}_2 \in \tld{\bun{W}}^+$, and $\tld{w}_1\uparrow \tld{w}_h^{-1}\tld{w}_2$.
By an argument analogous to the proof of Proposition \ref{prop:obvintersect} applied to $\pi^{-1}(\tld{w}^{-1}) = \tld{w}_1^{-1} w'^{-1} \tld{w}_2$ (that is we replace $\tld{w}$, $w'$, $\tld{w}_1$, and $\tld{w}_2$ with $\tld{w}^{-1}$, $w'^{-1}$, $\tld{w}_2$, and $\tld{w}_1$, respectively), we can assume without loss of generality that $\tld{w}_2$ is $p$-restricted.
It suffices to show that $\pi^{-1}(\tld{w})\leq t_{w_2^{-1}w_0\eta}$.
Note that 
\[t_{w_2^{-1}w_0\eta} = \tld{w}_2^{-1}\tld{w}_h t_\eta \tld{w}_h^{-1} \tld{w}_2 = \tld{w}_2^{-1}w_0(\tld{w}_h^{-1} \tld{w}_2),\] and that
$\ell(\tld{w}_2^{-1}w_0(\tld{w}_h^{-1} \tld{w}_2)) = \ell(\tld{w}_2^{-1})+\ell(w_0)+\ell(\tld{w}_h^{-1} \tld{w}_2)$ by Lemma \ref{lemma:gallery}.
Then since $w'\leq w_0$ and $\tld{w}_1 \leq \tld{w}_h^{-1} \tld{w}_2$ by Proposition \ref{prop:comp} (since $\tld{w}_1 \uparrow \tld{w}_h^{-1} \tld{w}_2$), we have that $\pi^{-1}(\tld{w})\leq t_{w_2^{-1}w_0\eta}$.

Conversely, suppose that $\tld{w} \in \Adm(\eta)$.
Then there exists $w_2\in W$ such that $\pi^{-1}(\tld{w}) \leq t_{w_2^{-1}w_0\eta}$.
Let $\tld{w}_2 \in \tld{\bun{W}}^+$ be such that $\tld{w}_2$ has projection $w_2\in W$ and $\tld{w}_2\cdot (\bun{C}_0)$ is $p$-restricted.
(Such elements differ by weights whose restrictions to the derived group are trivial.)
Since $t_{w_2^{-1}w_0\eta} = \tld{w}_2^{-1}w_0(\tld{w}_h^{-1} \tld{w}_2)$ and $\ell(\tld{w}_2^{-1}w_0(\tld{w}_h^{-1} \tld{w}_2)) = \ell(\tld{w}_2^{-1})+\ell(w_0)+\ell(\tld{w}_h^{-1} \tld{w}_2)$ as in the last paragraph, $\pi^{-1}(\tld{w}) = (\tld{w}_2')^{-1}w'\tld{w}_1'$ where $\tld{w}_2'\leq \tld{w}_2$, $\tld{w}_1'\leq \tld{w}_h^{-1}\tld{w}_2$ and $w' \leq w_0$.
In particular, $w'\in W(\bun{G})$.
If $w_1'$ and $w_2'\in W(\bun{G})$ and $\tld{w}_1^+$ and $\tld{w}_2^+\in \tld{\bun{W}}^+$ are the unique elements such that $\tld{w}_2' = w_2'\tld{w}_2^+$ and $\tld{w}_1' = w_1'\tld{w}_1^+$, then $\tld{w}_2^+\leq\tld{w}'_2\leq \tld{w}_2$ and $\tld{w}_1^+\leq \tld{w}'_1\leq \tld{w}_h^{-1}\tld{w}_2$ by Lemma \ref{lemma:domineq}.
Thus $\tld{w}_2^+\uparrow \tld{w}_2$ and $\tld{w}_1^+ \uparrow \tld{w}_h^{-1}\tld{w}_2$.
Letting $w'' = (w'_2)^{-1}w'w'_1$, we get that $\pi^{-1}(\tld{w}) = (\tld{w}_2^+)^{-1} w'' \tld{w}_1^+$.
Since $\tld{w}_1^+\uparrow \tld{w}_h^{-1}\tld{w}_2$ and $\tld{w}_2^+\uparrow \tld{w}_2$, or equivalently $\tld{w}_h^{-1}\tld{w}_2 \uparrow \tld{w}_h^{-1}\tld{w}_2^+$, we have that $\tld{w}_1^+\uparrow \tld{w}_h^{-1} \tld{w}_2^+$.
By the proof of Proposition \ref{prop:obvintersect} applied to $\pi^{-1}(\tld{w}) = (\tld{w}_2^+)^{-1} w'' \tld{w}_1^+$ (that is we replace $w'$, $\tld{w}_1$, and $\tld{w}_2$ with $w''$, $\tld{w}_1^+$, and $\tld{w}_2^+$, respectively), modifying the factorization if necessary, we can assume without loss of generality that $\tld{w}_1^+\cdot \bun{C}_0$ is $p$-restricted.
By Proposition \ref{prop:alcovecrit} taking $\tld{w}_\lambda$ to be $\tld{w}_1^+$, we see that $W^?(\rhobar,\tau)$ is nonempty.
\end{proof}

\begin{thm}\label{thm:lift}
Let $\rhobar:G_K\ra \GL_n(\F)$ be a $(6n-2)$-generic semisimple Galois representation and let $\tau$ be $2n$-generic tame inertial type.
Let $F$ be a CM field such that $\zeta_p\notin F$ and let $\rbar:G_F \ra \GL_n(\F)$ be a Galois representation as in \S \ref{sec:algmf} satisfying the following hypotheses.
\begin{itemize}
\item  $\rbar$ is potentially diagonalizably automorphic, i.e.\ there is a RACSDC automorphic representation $\Pi$ of $\GL_n(\A_F)$ such that 
\begin{itemize}
\item $\rbar\cong \rbar_{p,\iota}(\Pi)$; and
\item For each place $u|p$ of $F$, $r_{p,\iota}(\Pi)|_{G_{F_u}}$ is potentially diagonalizable.
\end{itemize}
\item The image of $\rbar(G_{F(\zeta_p)})$ is adequate.
\item $\rbar$ is generic and semisimple at all places dividing $p$ and $\rbar|_{G_{F_{\tld{v}}}} \cong \rhobar$ for a place $\tld{v}|p$ of $F$.
\end{itemize}
Then the following are equivalent:
\begin{enumerate}
\item \label{item:globallift} There is a RACSDC representation $\Pi$ of $\GL_n(\A_F)$ such that $\rbar\cong \rbar_{p,\iota}(\Pi)$ and the restriction of $r_{p,\iota}(\Pi)$ at $\tld{v}$ is potentially crystalline of type $(\eta,\tau)$;
\item \label{item:locallift} $\rhobar$ has a potentially crystalline lift of type $(\eta,\tau)$;
\item \label{item:adm} $\tld{w}^*(\rhobar,\tau) \in \Adm(\eta)$;
\item \label{item:intersect} $W^?(\rhobar,\tau) \neq \emptyset$; and
\item \label{item:obvintersect} $W_{\mathrm{obv}}(\rhobar) \cap \JH(\taubar) \neq \emptyset$.
\end{enumerate}
\end{thm}
\begin{rmk}
The assumption that $\tau$ is $2n$-generic can be relaxed to $n$-generic.
Suppose that $\tau$ is $n$-generic, but not $2n$-generic.
Then (\ref{item:intersect}) can be checked to be false since $\JH(\taubar)$ will contain only weights which are not $3n$-deep in their alcoves by \cite[Theorem 5.2]{herzig-duke} (note that by the linkage principle, the depth of the Jordan--H\"older factors appearing in \emph{loc.~cit.}~coincides with that of $\mu-w\epsilon'_{w_0\tau}$ and $\max_{\alpha^\vee}|\langle \epsilon'_{w_0 \tau},\alpha^\vee \rangle| \leq n-1$ for all $\tau \in W$ by Remark \ref{rmk:bounderror}).
This implies that (\ref{item:obvintersect}) is false as well.
If (\ref{item:adm}) holds, then a direct computation shows that $\tau$ is $4n$-generic, using that $\max_{\alpha^\vee}|\langle \tld{w}(0),\alpha^\vee\rangle| \leq n-1$ for any $\tld{w} \in \Adm(\eta)$.
So (\ref{item:adm}) is false.
Then (\ref{item:locallift}) is false by Theorem \ref{thm:admcrit}, which immediately implies that (\ref{item:globallift}) is false.
\end{rmk}
\begin{rmk}
For $\rhobar$ as in the theorem, there always exists a representation $\rbar$ as in the theorem.
Indeed, since $\rhobar$ is Fontaine--Laffaille, \cite[Conjecture A.3]{EG} holds for $\rhobar$ (alternatively, one can use Corollary \ref{cor:pd}).
Let $\rbar:G_F \ra \GL_n(\F)$ be a suitable globalization of $\rhobar$ as constructed by \cite[Corollary A.7]{EG}.
By \emph{loc.~cit.}~and \cite[Lemma A.5]{EG}, $\rbar$ satisfies the required hypotheses.
Thus, by removing sentences containing $\rbar$, the above theorem can be interpreted as giving existence criteria for potentially crystalline lifts of type $(\eta,\tau)$ for generic semisimple $\rhobar$ and tame generic types $\tau$.
\end{rmk}
\begin{proof}
(\ref{item:globallift}) immediately implies (\ref{item:locallift}). 
(\ref{item:locallift}) implies (\ref{item:adm}) by Theorem \ref{thm:admcrit} and Proposition \ref{rmk:equalshape}.
The equivalence of (\ref{item:adm}), (\ref{item:intersect}), and (\ref{item:obvintersect}) follows from Propositions \ref{prop:obvintersect} and \ref{prop:intersect}.
It remains to show that (\ref{item:obvintersect}) implies (\ref{item:globallift}).

Assume that $W_{\mathrm{obv}}(\rhobar) \cap \JH(\taubar) \neq \emptyset$.
By Theorem \ref{thm:obvwt}, $W_{\mathrm{obv}}(\rbar) \subset W(\rbar)$ in the notation of \S \ref{sec:algmf}.
Following the notation of the proof of Theorem \ref{thm:obvwt}, let $\otimes_{v'|p} F(\lambda_{v'})\in W_{\mathrm{obv}}(\rbar)$ be such that $F(\lambda_{\tld{v}}) \in W_{\mathrm{obv}}(\rhobar) \cap \JH(\taubar)$.
Then, as in the proof of Theorem \ref{thm:obvwt}, we obtain the required automorphic representation $\Pi$.
\end{proof}

\newpage
\bibliography{Biblio}
\bibliographystyle{amsalpha}

\end{document}